\PassOptionsToPackage{pagebackref=true}{hyperref}
\documentclass[english,11pt]{article}
\usepackage{tikz}
\usetikzlibrary{decorations.pathreplacing}
\usepackage{enumerate}
\usepackage[shortlabels]{enumitem}
\usepackage{verbatim}
\usepackage[margin=1in]{geometry}
\usepackage{amssymb}
\usepackage{caption}
\usepackage{bbm}
\usepackage{amsthm}
\usepackage{enumitem}
\usepackage{xcolor}
	\definecolor{my_blue}{rgb}{0.0,0.0,0.6}
	\definecolor{my_green}{rgb}{0.0,0.5,0.0}
	\definecolor{my_light_gray}{gray}{0.6}
\usepackage{scalerel}    
\usepackage{stmaryrd}   
\usepackage{tocbibind} 
\usepackage{microtype} 
\usepackage{stackengine}

\makeatletter
\DeclareRobustCommand\widecheck[1]{{\mathpalette\@widecheck{#1}}}
\def\@widecheck#1#2{%
    \setbox\z@\hbox{\m@th$#1#2$}%
    \setbox\tw@\hbox{\m@th$#1%
       \widehat{%
          \vrule\@width\z@\@height\ht\z@
          \vrule\@height\z@\@width\wd\z@}$}%
    \dp\tw@-\ht\z@
    \@tempdima\ht\z@ \advance\@tempdima2\ht\tw@ \divide\@tempdima\thr@@
    \setbox\tw@\hbox{%
       \raise\@tempdima\hbox{\scalebox{1}[-1]{\lower\@tempdima\box
\tw@}}}%
    {\ooalign{\box\tw@ \cr \box\z@}}}
\makeatother

\usepackage{rotating,centernot,cancel}    

\usepackage{amsmath}
\usepackage{tocloft}
\usepackage{float}

\usepackage{tcolorbox}   
\PassOptionsToPackage{hyphens}{url}\usepackage{hyperref}
    \hypersetup{colorlinks=true, linkcolor=my_blue, citecolor=my_green, urlcolor=my_light_gray}
\usepackage{babel}
\theoremstyle{plain}

\input amssym.def
\input amssym.tex

\usepackage{mathtools}

\numberwithin{equation}{section}

\newtheorem{theorem}{Theorem}[section]
\newtheorem{corollary}[theorem]{Corollary} 
\newtheorem{lemma}[theorem]{Lemma} 
\newtheorem*{lemma*}{Lemma}
\newtheorem{proposition}[theorem]{Proposition} 

\theoremstyle{definition}
\newtheorem{remark}[theorem]{Remark}

\theoremstyle{definition}

\def\R{\mathbb R}

\def\Z{\mathbb Z}

\def\to{\rightarrow}
\def\dlim[#1][#2]{\lim_{#1 \to #2, #1 \neq #2}}

\def\Exp{\textup{Exp}}
\def\Cov{\textup{$\mathbb{C}$ov}}
\def\Rec{\mathsf{Rec}}

\def\bfe{\mathbf e}
\def\bfu{\mathbf u}
\def\bfv{\mathbf v}
\def\bfp{\mathbf p}

\newcommand{\be}{\begin{equation}}
\newcommand{\ee}{\end{equation}}

\newcommand\bbullet{{{\scaleobj{0.6}{\bullet}}}} 
\newcommand\mydots{\hbox to 1em{.\hss.\hss.}}

\def\wt{\widetilde}  \def\wh{\widehat}  \def\wc{\widecheck}

\newcommand{\myfootnote}[1]{
    \renewcommand{\thefootnote}{}
    \footnotetext{\scriptsize#1}
    \renewcommand{\thefootnote}{\arabic{footnote}}
}

	\newcommand{\eq}[1]{\begin{align*} #1 \end{align*}}
	\newcommand{\eeq}[1]{\begin{align} \begin{split} #1 \end{split} \end{align}}

        \newcommand{\stackref}[2]{
        		\readlist*\mylist{#1}
        		\stackrel{\mbox{\footnotesize\foreachitem\x\in\mylist[]{\ifnum\xcnt=1\else,\fi\eqref{\x}}}}{#2}
        }
        \newcommand{\stackrefp}[2]{
        		\readlist*\mylist{#1}
        		\stackrel{\hphantom{\mbox{\footnotesize\foreachitem\x\in\mylist[]{\ifnum\xcnt=1\else,\fi\eqref{\x}}}}}{#2}
        }
        \newcommand{\stackrefpp}[3]{
       		\readlist*\mylist{#1}
        		\readlist*\mylistt{#2}
        		\stackrel{\parbox{\widthof{\footnotesize\foreachitem\x\in\mylistt[]{\ifnum\xcnt=1\else,\fi\eqref{\x}}}}{\centering\footnotesize\foreachitem\x\in\mylist[]{{\ifnum\xcnt=1\else,\fi\eqref{\x}}}}}{#3}
        }

\allowdisplaybreaks

\title{Negative association of Busemann functions 
\\in exponential last-passage percolation}
\author{
  Erik Bates\thanks{\scriptsize{Department of Mathematics, North Carolina State University.
 \texttt{ebates@ncsu.edu}}}  \qquad  Xiao Shen\thanks{\scriptsize{Department of Mathematics, North Carolina State University. \texttt{xshen9@ncsu.edu}}}
  }
\date{}
\begin{document}
\maketitle

\begin{abstract}

One hallmark of exactly solvable KPZ random growth models is product-form invariant measures.
In the setting of exponential last-passage percolation (LPP), this corresponds to the independence of Busemann increments along any down-right path. 
However, this independence breaks down when multiple asymptotic directions are considered simultaneously, owing to the fact that \textit{jointly} invariant measures are not \textit{jointly} product-form.
This paper shows that the failure of independence is one-sided: Busemann increments across arbitrary directions are negatively associated.
As an application, we derive an exponential concentration inequality for sums of Busemann increments on the diffusive scale, even when the increments are not independent.
While our argument relies on a Burke property that is special to exponential weights, all other proof ingredients---including hidden LPP monotonicities and braid relations for queueing maps---hold for arbitrary weights.

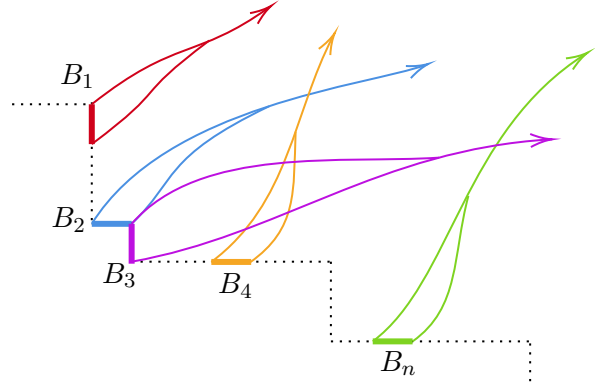
\begin{figure}[h]
\begin{center}

\tikzset{every picture/.style={line width=0.75pt}} 

\begin{tikzpicture}[x=0.75pt,y=0.75pt,yscale=-1,xscale=1]

\draw  [dash pattern={on 0.84pt off 2.51pt}]  (129.2,89.3) -- (169.2,89.3) ;
\draw  [dash pattern={on 0.84pt off 2.51pt}]  (169.2,88.3) -- (169.2,148.3) ;
\draw    (169.2,149.3) -- (189.2,149.3) ;
\draw  [dash pattern={on 0.84pt off 2.51pt}]  (189.2,148.3) -- (189.2,168.3) ;
\draw  [dash pattern={on 0.84pt off 2.51pt}]  (189.2,168.3) -- (289.2,168.3) ;
\draw  [dash pattern={on 0.84pt off 2.51pt}]  (289.2,168.3) -- (289.2,208.3) ;
\draw  [dash pattern={on 0.84pt off 2.51pt}]  (289.2,208.3) -- (389.2,208.3) ;
\draw  [dash pattern={on 0.84pt off 2.51pt}]  (389.2,209.3) -- (389.2,229.3) ;
\draw [color={rgb, 255:red, 208; green, 2; blue, 27 }  ,draw opacity=1 ][line width=2.25]    (169.2,89.3) -- (169.2,109.3) ;
\draw [color={rgb, 255:red, 208; green, 2; blue, 27 }  ,draw opacity=1 ]   (169.2,89.3) .. controls (211.77,55.64) and (221.21,67.16) .. (258.27,40.62) ;
\draw [shift={(259.4,39.8)}, rotate = 144.01] [color={rgb, 255:red, 208; green, 2; blue, 27 }  ,draw opacity=1 ][line width=0.75]    (10.93,-3.29) .. controls (6.95,-1.4) and (3.31,-0.3) .. (0,0) .. controls (3.31,0.3) and (6.95,1.4) .. (10.93,3.29)   ;
\draw [color={rgb, 255:red, 208; green, 2; blue, 27 }  ,draw opacity=1 ]   (169.2,109.3) .. controls (209.2,79.3) and (188.2,87.3) .. (228.2,57.3) ;
\draw [color={rgb, 255:red, 74; green, 144; blue, 226 }  ,draw opacity=1 ][line width=2.25]    (169.2,149.3) -- (189.2,149.3) ;
\draw [color={rgb, 255:red, 74; green, 144; blue, 226 }  ,draw opacity=1 ]   (169.2,149.3) .. controls (206.44,89.52) and (308.3,79.89) .. (336.56,69.63) ;
\draw [shift={(338.2,69)}, rotate = 157.75] [color={rgb, 255:red, 74; green, 144; blue, 226 }  ,draw opacity=1 ][line width=0.75]    (10.93,-3.29) .. controls (6.95,-1.4) and (3.31,-0.3) .. (0,0) .. controls (3.31,0.3) and (6.95,1.4) .. (10.93,3.29)   ;
\draw [color={rgb, 255:red, 74; green, 144; blue, 226 }  ,draw opacity=1 ]   (189.2,149.3) .. controls (212.32,127.28) and (204.2,116.3) .. (258.2,90.3) ;
\draw [color={rgb, 255:red, 245; green, 166; blue, 35 }  ,draw opacity=1 ][line width=2.25]    (229.2,168.43) -- (249.2,168.43) ;
\draw [color={rgb, 255:red, 245; green, 166; blue, 35 }  ,draw opacity=1 ]   (229.2,168.43) .. controls (268.6,138.88) and (270.16,93.8) .. (290.27,55.18) ;
\draw [shift={(291.2,53.43)}, rotate = 118.3] [color={rgb, 255:red, 245; green, 166; blue, 35 }  ,draw opacity=1 ][line width=0.75]    (10.93,-3.29) .. controls (6.95,-1.4) and (3.31,-0.3) .. (0,0) .. controls (3.31,0.3) and (6.95,1.4) .. (10.93,3.29)   ;
\draw [color={rgb, 255:red, 245; green, 166; blue, 35 }  ,draw opacity=1 ]   (248.58,168.68) .. controls (272.58,151.68) and (270.38,133.19) .. (271.5,103.06) ;
\draw [color={rgb, 255:red, 189; green, 16; blue, 224 }  ,draw opacity=1 ][line width=2.25]    (189.2,149.3) -- (189.2,169.3) ;
\draw [color={rgb, 255:red, 126; green, 211; blue, 33 }  ,draw opacity=1 ][line width=2.25]    (310.2,208.3) -- (330.2,208.3) ;
\draw [color={rgb, 255:red, 126; green, 211; blue, 33 }  ,draw opacity=1 ]   (310.2,208.3) .. controls (350,178.45) and (355.15,103.06) .. (417.26,63.89) ;
\draw [shift={(418.2,63.3)}, rotate = 148.24] [color={rgb, 255:red, 126; green, 211; blue, 33 }  ,draw opacity=1 ][line width=0.75]    (10.93,-3.29) .. controls (6.95,-1.4) and (3.31,-0.3) .. (0,0) .. controls (3.31,0.3) and (6.95,1.4) .. (10.93,3.29)   ;
\draw [color={rgb, 255:red, 126; green, 211; blue, 33 }  ,draw opacity=1 ]   (330.2,208.3) .. controls (354.2,191.3) and (350.2,165.3) .. (358.2,135.3) ;
\draw [color={rgb, 255:red, 189; green, 16; blue, 224 }  ,draw opacity=1 ]   (189.2,168.3) .. controls (262.83,155.37) and (313.41,112.12) .. (399.03,106.95) ;
\draw [shift={(400.32,106.88)}, rotate = 176.82] [color={rgb, 255:red, 189; green, 16; blue, 224 }  ,draw opacity=1 ][line width=0.75]    (10.93,-3.29) .. controls (6.95,-1.4) and (3.31,-0.3) .. (0,0) .. controls (3.31,0.3) and (6.95,1.4) .. (10.93,3.29)   ;
\draw [color={rgb, 255:red, 189; green, 16; blue, 224 }  ,draw opacity=1 ]   (189.2,149.3) .. controls (223.92,109.68) and (295.52,119.28) .. (343.52,116.08) ;

\draw (151.75,67.53) node [anchor=north west][inner sep=0.75pt]    {$B_{1}$};
\draw (147.45,139.6) node [anchor=north west][inner sep=0.75pt]    {$B_{2}$};
\draw (172.9,167.28) node [anchor=north west][inner sep=0.75pt]    {$B_{3}$};
\draw (231.2,172.7) node [anchor=north west][inner sep=0.75pt]    {$B_{4}$};
\draw (312.2,211.7) node [anchor=north west][inner sep=0.75pt]    {$B_{n}$};

\end{tikzpicture}

\captionsetup{width=.8\linewidth, font=small, singlelinecheck=false}
\caption{
A special case of our main result.
Consider Busemann increments $B_1,\ldots,B_n$ along a down-right path, with $n$ arbitrary slope parameters. 
Unless the slopes are non-decreasing, these increments are not independent. 
Nevertheless, our result shows they are negatively dependent in the sense that \\[-0.5\baselineskip]
\begin{minipage}[t]{\linewidth}
\[
\mathbb{P}\Big(\bigcap_{i=1}^n\{B_i < t_i\}\Big)
\le
\prod_{i=1}^n\mathbb{P}\big(B_i < t_i\big) \quad \text{and} \quad 
\mathbb{P}\Big(\bigcap_{i=1}^n\{B_i > t_i\}\Big)
\le
\prod_{i=1}^n\mathbb{P}\big(B_i > t_i\big). \vspace{0.5\baselineskip}
\]
\end{minipage}
This roughly suggests the joint Busemann process is more concentrated than the product of its marginals.
}
\label{fig_titlepage}
\end{center}
\end{figure}
\end{abstract}

\myfootnote{Date: \today}
\myfootnote{2020 Mathematics Subject Classification: 60K35, 60K37}
\myfootnote{Keywords: last-passage percolation, Kardar--Parisi--Zhang, Busemann function, negative association, random growth}

\addtocontents{toc}{\protect\setcounter{tocdepth}{-1}}
\tableofcontents
\addtocontents{toc}{\protect\setcounter{tocdepth}{2}}

\section{Introduction}\label{intro}

In the study of random growth models, one of the central problems is to prove the universality of fluctuation statistics.
In last-passage percolation (LPP), the two fluctuations of greatest interest are of \textit{passage time} and \textit{geodesic location}.
A wide range of models---also including first-passage percolation, polymers, and many exclusion processes---are believed to belong to the Kardar--Parisi--Zhang (KPZ) universality class, meaning their asymptotic fluctuations admit a common scaling limit.
Moreover, the size of fluctuations are believed to be governed by exponents whose values reflect certain fractal dimensions of the scaling limit.
In our setting, it is common to name a \textit{fluctuation exponent} $\chi$ for passage times, and a \textit{transversal fluctuation exponent} $\xi$ for geodesic location.
The so-called KPZ relation says that $\chi = 2\xi-1$ \cite{chatterjee13, auffinger_damron14, auffinger_damron13b}.

In the two-dimensional case, a second relation $\xi = 2\chi$ is believed to be true (hence $\chi = \tfrac13$ and $\xi = \tfrac23$) and has been verified for various exactly solvable directed models.
Directed-ness means the growth dynamics are Markovian, and exact solvability means the associated invariant measures can be explicitly identified.
When these invariant measures have independent increments, it is possible to confirm the KPZ fluctuation exponents using probabilistic arguments.
Examples include Poissonian LPP \cite{cator_groeneboom06}, exponential LPP \cite{balazs_cator_seppalainen06}, the asymmetric simple exclusion process \cite{balazs_seppalainen10}, the inverse-gamma polymer \cite{seppalainen12}, the O'Connell--Yor polymer \cite{seppalainen_valko10, morenoflores_seppalainen_valko14}, the beta, inverse-beta, and strict-weak polymers \cite{chaumont_noack18b}, strict-strict Bernoulli LPP \cite{ciech_georgiou19}, and the KPZ equation itself \cite{balazs_quastel_seppalainen11}.

Beyond exactly solvable models, invariant measures can be accessed indirectly via Busemann functions.
In the setting of LPP, Busemann functions are defined as limits of differences in passage times from two initial points to a common terminal point that is sent to infinity in a prescribed direction; see \eqref{busemann_def}.
It turns out the increments of a Busemann function are those of an invariant measure, and the prescribed direction corresponds to a conserved quantity that (should) parameterize all extremal invariant measures.
The solvable models mentioned above have Busemann functions with \textit{independent} increments (see Remark \ref{rem_ind}), which is implicitly the key force behind those results.
But in general this property is not believed to hold; in fact, a result to this effect is obtained in \cite{chaumont_noack18a}.

Nevertheless, Busemann functions have the potential to deliver fluctuation exponents for general models, not through independence but rather \textit{negative correlation}.
The latter still implies that fluctuations of Busemann functions scale (at most) diffusively, which is the essential component for showing (the $\ge$ direction of) $\xi = 2\chi$.
A heuristic explanation can be found in \cite[Appendix B]{alevy_krishnan22}.
The same paper identified a sufficient condition for one particular pair of Busemann increments to have negative correlations, along with a family of weight distributions for which the condition could be rigorously verified.
This was not strong enough to prove $\xi \ge 2\chi$, because one needs negative correlations for many increment-pairs not covered by \cite{alevy_krishnan22}.
Even so, \cite{alevy_krishnan22} provided some rigorous footing for this approach.
Otherwise the evidence of negative correlations is scant, despite the potentially far-reaching implications.

Although Busemann increments might not be negatively correlated among all LPP models (see \cite[Questions 1 and 2]{alevy_krishnan22}), there is already a setting within exactly solvable models in which to approach this problem.
Specifically, this paper considers Busemann increments across \textit{different directions}, which are not necessarily independent (e.g.\ Proposition~\ref{prop_strict}).
The available descriptions of the joint law reveal a much more complicated structure, and many aspects of the joint process remain elusive.
The main result of this paper is a negative association statement (Theorem~\ref{thm_1}) that accommodates arbitrarily many directions.
This statement is more than strong enough to imply the type of diffusivity mentioned above, even for sums of Busemann increments with potentially different directions (Corollary~\ref{thm_diffusive}).
We only work with exponential last-passage percolation (LPP), but unlike for the delicate independence properties, we anticipate that some version of our main result holds beyond exactly solvable settings.
In fact, the bulk of our proof works for general weights; the only use of the exponential distribution comes in the Burke property (Proposition~\ref{burk}), albeit a crucial ingredient.

\subsection{The corner growth model and Busemann functions}
Denote the standard basis vectors in $\R^2$ by $\bfe_1 = (1,0)$ and $\bfe_2 = (0,1)$.
Consider a collection of i.i.d.\ $\Exp(1)$ random variables $\{\omega_{\bfv}\}_{\bfv\in \Z^2}$, which are regarded as weights attached to the vertices of $\Z^2$.
The {\it last-passage value} between two coordinate-wise ordered points $\bfu \leq \bfv$ in $\Z^2$, is the maximum sum of weights along an up-right nearest-neighbor path from $\bfu$ to $\bfv$. 
More precisely, let $\Pi^{\bfu,\bfv}$ denote the collection of paths $\mathbf p_{\bbullet}=(\bfp_i)_{i=0}^{|\bfv-\bfu|_1}$ that start at $\bfp_0=\bfu$, end at $\bfp_{|\bfv-\bfu|_1}=\bfv$, and satisfy $\bfp_{i+1}-\bfp_i\in\{\mathbf e_1, \mathbf e_2\}$ for every $i$.
Then the last-passage value between $\bfu$ and $\bfv$ is denoted by
\begin{equation}\label{def_G}
G_{\bfu,\bfv} = \max_{\bfp_{\bbullet}\in\Pi^{\bfu,\bfv}} \sum_{i=0}^{|\bfv-\bfu|_1} \omega_{\bfp_i}.
\end{equation}
We set $G_{\bfu, \bfv} = -\infty$ if $\Pi^{\bfu, \bfv}$ is empty (i.e.~$\bfu \centernot\leq\bfv$). 
Because the weights are continuously distributed, there is almost surely a unique maximizing path which is called the \textit{geodesic} between $\bfu$ and $\bfv$. 

The regime of interest is when $\mathbf v\to\infty$.
We parametrize the allowed asymptotic directions by the following $\ell^1$-unit vectors:
\begin{equation}
\boldsymbol{\xi}[\rho] = \tfrac{1}{|(\rho^2, (1-\rho)^2)|_1}\left(\rho^2, (1-\rho)^2 \right), \quad \rho\in(0,1). \label{char1}
\end{equation}
The \textit{Busemann function} in direction $\boldsymbol{\xi}[\rho]$ is the following almost sure limit:
\begin{align} \label{busemann_def}
B^\rho_{\mathbf{x},\mathbf{y}} = \lim_{n\to \infty}(G_{\mathbf{x}, \mathbf{v}_n} - G_{\mathbf{y}, \mathbf{v}_n}),
\end{align}
where $(\mathbf{v}_n)$ is any sequence with $|\mathbf{v}_n|_1 \to \infty$ and $\mathbf{v}_n / |\mathbf{v}_n|_1 \to \boldsymbol{\xi}[\rho]$ as $n\to\infty$. The existence of this limit was first established in \cite{ferrari_pimentel05}, and can be read from \cite[Theorem 4.2(iii)]{seppalainen20a} in the present notation. 
Over unit edges, the marginal distributions are
\begin{equation}\label{mdist}B^{\rho}_{\mathbf x, \mathbf x+\mathbf e_1} \sim \Exp(\rho) \qquad \textup{ and } \qquad B^{\rho}_{\mathbf x, \mathbf x+\mathbf e_2} \sim \Exp(1-\rho).
\end{equation}
When we evaluate the Busemann functions over these unit edges, we will often refer to them as \textit{Busemann increments}.
Note that the order of vertices matters, so we write $[\![\mathbf a,\mathbf b]\!]$ to refer to a unit edge with the orientation $\mathbf{a}\to\mathbf{b}$, allowing $\mathbf b\in\{\mathbf a\pm\mathbf e_1,\mathbf a\pm\mathbf e_2\}$.

Even though it does not directly affect this paper, we mention that the almost sure event on which the limit \eqref{busemann_def} exists depends on $\rho$, and in fact \eqref{busemann_def} does not hold simultaneously for every $\rho\in(0,1)$.
For general i.i.d.\ weights $(\omega_{\bfv})_{\bfv\in\Z^2}$, the limit \eqref{busemann_def} still makes sense \cite{georgiou_rassoulagha_seppalainen17a, janjigian_rassoulagha20a}, but an explicit distribution as in \eqref{mdist} is not available.
In higher dimensions, existence of the limit is not known, but there is still a process $(B^{\rho}_{\mathbf x,\mathbf y})_{\rho,\mathbf x,\mathbf y}$ with the correct distributional properties \cite{groathouse_janjigian_rassoulagha25}.

\subsection{Main results}\label{sec:main}

The setup for our central result, Theorem \ref{thm_1}, is illustrated in Figure \ref{fig_main}. 
Given vertices $\mathbf x = (x_1,x_2)$ and $\mathbf y = (y_1,y_2)$ in $\mathbb{Z}^2$, we write $\mathbf{x} \preceq \mathbf y$ if $x_1 \le y_1$ and $x_2 \ge y_2$, meaning $\mathbf{x}$ precedes $\mathbf{y}$ along a down-right path. 
Given $\mathbf{x} \preceq \mathbf y$, let $\Rec[\mathbf{x}, \mathbf{y}]\subset\Z^2$ denote the rectangle with upper left corner at $\mathbf{x}$ and lower right corner at $\mathbf{y}$. 
Consider a sequence of vertices $(\mathbf x_i)_{i\in \mathbb{Z}}$ such that $\mathbf x_{i-1} \preceq \mathbf x_i$ and $\mathbf x_{2i}- \mathbf x_{2i-1} \in \{\mathbf{e}_1, -\mathbf{e}_2\}$ for every $i$. 
In other words, we have a sequence of unit edges $([\![\mathbf{x}_{2i-1},\mathbf{x}_{2i}]\!])_{i\in\Z}$ that lie on a common down-right path but are not necessarily adjacent.
Given some direction parameter $\rho\in(0,1)$, denote the corresponding Busemann increments by
$$B_i  = B^\rho_{\mathbf x_{2i-1}, \mathbf x_{2i}}.$$
For any direction parameters $\lambda_1,\ldots,\lambda_k\in(0,1)$, consider Busemann functions $R_1,\ldots,R_k$ defined as
$$R_j = B^{\lambda_j}_{\mathbf a_j, \mathbf b_j} \quad \text{for some $\mathbf a_j \preceq \mathbf b_j$ such that $\mathbf a_j, \mathbf b_j \in \Rec[\mathbf x_{2i}, \mathbf x_{2i+1}]$ for some $i\in \mathbb{Z}$}.$$
The following result states that, in a probabilistic sense, increasing the values of $(B_i)_{i=1}^n$ tends to decrease the values of $(R_j)_{j=1}^k$.

\begin{figure}[t]
\begin{center}

\tikzset{every picture/.style={line width=0.75pt}} 

\begin{tikzpicture}[x=0.75pt,y=0.75pt,yscale=-1,xscale=1]

\draw [color={rgb, 255:red, 208; green, 2; blue, 27 }  ,draw opacity=0.33 ][line width=2.25]    (234.25,150.31) -- (246.56,150.58) ;
\draw [color={rgb, 255:red, 208; green, 2; blue, 27 }  ,draw opacity=0.33 ]   (222.37,134.02) .. controls (257.68,120.52) and (258.59,127.64) .. (296.09,104.89) ;
\draw [color={rgb, 255:red, 208; green, 2; blue, 27 }  ,draw opacity=0.33 ]   (246.56,150.58) .. controls (294.92,118.6) and (309.12,81.33) .. (332.78,53.53) ;
\draw [shift={(334.62,51.41)}, rotate = 131.53] [fill={rgb, 255:red, 208; green, 2; blue, 27 }  ,fill opacity=0.33 ][line width=0.08]  [draw opacity=0] (8.93,-4.29) -- (0,0) -- (8.93,4.29) -- cycle    ;
\draw [color={rgb, 255:red, 0; green, 90; blue, 200 }  ,draw opacity=1 ]   (253.96,177.7) .. controls (316.8,107.55) and (355.14,122.25) .. (447.6,72.95) ;
\draw [shift={(449,72.2)}, rotate = 151.76] [fill={rgb, 255:red, 0; green, 90; blue, 200 }  ,fill opacity=1 ][line width=0.08]  [draw opacity=0] (8.93,-4.29) -- (0,0) -- (8.93,4.29) -- cycle    ;
\draw [color={rgb, 255:red, 0; green, 90; blue, 200 }  ,draw opacity=1 ]   (268.47,177.87) .. controls (292.29,164.69) and (286.78,160.63) .. (297.28,138.88) ;
\draw [color={rgb, 255:red, 208; green, 2; blue, 27 }  ,draw opacity=0.33 ][line width=2.25]    (219.87,134.44) -- (233.74,134.24) ;
\draw  [color={rgb, 255:red, 208; green, 2; blue, 27 }  ,draw opacity=0.33 ][fill={rgb, 255:red, 0; green, 0; blue, 0 }  ,fill opacity=1 ] (244.89,150.58) .. controls (244.89,149.66) and (245.64,148.92) .. (246.56,148.92) .. controls (247.47,148.92) and (248.22,149.66) .. (248.22,150.58) .. controls (248.22,151.5) and (247.47,152.24) .. (246.56,152.24) .. controls (245.64,152.24) and (244.89,151.5) .. (244.89,150.58) -- cycle ;
\draw [color={rgb, 255:red, 0; green, 90; blue, 200 }  ,draw opacity=1 ][line width=2.25]    (252.32,177.72) -- (268.47,177.87) ;
\draw  [color={rgb, 255:red, 208; green, 2; blue, 27 }  ,draw opacity=0.33 ][fill={rgb, 255:red, 0; green, 0; blue, 0 }  ,fill opacity=1 ] (218.2,134.44) .. controls (218.2,133.53) and (218.95,132.78) .. (219.87,132.78) .. controls (220.78,132.78) and (221.53,133.53) .. (221.53,134.44) .. controls (221.53,135.36) and (220.78,136.11) .. (219.87,136.11) .. controls (218.95,136.11) and (218.2,135.36) .. (218.2,134.44) -- cycle ;
\draw [color={rgb, 255:red, 208; green, 2; blue, 27 }  ,draw opacity=0.33 ][line width=2.25]    (173.52,150.4) -- (207.6,150.4) ;
\draw [color={rgb, 255:red, 208; green, 2; blue, 27 }  ,draw opacity=0.33 ]   (176.53,134.62) .. controls (211.85,121.12) and (239.64,129) .. (278.98,123.33) ;
\draw [color={rgb, 255:red, 208; green, 2; blue, 27 }  ,draw opacity=0.33 ][line width=2.25]    (173.54,134.62) -- (173.52,150.4) ;
\draw  [color={rgb, 255:red, 208; green, 2; blue, 27 }  ,draw opacity=0.33 ][fill={rgb, 255:red, 0; green, 0; blue, 0 }  ,fill opacity=1 ] (205.94,150.4) .. controls (205.94,149.48) and (206.68,148.74) .. (207.6,148.74) .. controls (208.52,148.74) and (209.26,149.48) .. (209.26,150.4) .. controls (209.26,151.32) and (208.52,152.06) .. (207.6,152.06) .. controls (206.68,152.06) and (205.94,151.32) .. (205.94,150.4) -- cycle ;
\draw  [color={rgb, 255:red, 208; green, 2; blue, 27 }  ,draw opacity=0.33 ][fill={rgb, 255:red, 0; green, 0; blue, 0 }  ,fill opacity=1 ] (171.87,134.62) .. controls (171.87,133.7) and (172.62,132.95) .. (173.54,132.95) .. controls (174.46,132.95) and (175.2,133.7) .. (175.2,134.62) .. controls (175.2,135.53) and (174.46,136.28) .. (173.54,136.28) .. controls (172.62,136.28) and (171.87,135.53) .. (171.87,134.62) -- cycle ;
\draw [color={rgb, 255:red, 0; green, 90; blue, 200 }  ,draw opacity=1 ]   (342.89,230.17) .. controls (351.2,208.68) and (365.7,189.67) .. (373.95,170.67) ;
\draw [color={rgb, 255:red, 0; green, 90; blue, 200 }  ,draw opacity=1 ]   (342.87,248.09) .. controls (366.69,234.92) and (382.22,119.75) .. (392.72,98) ;
\draw [color={rgb, 255:red, 208; green, 2; blue, 27 }  ,draw opacity=0.33 ]   (207.6,150.4) .. controls (244.85,125.63) and (367,120.83) .. (403.67,114.12) ;
\draw [shift={(406.32,113.6)}, rotate = 168.02] [fill={rgb, 255:red, 208; green, 2; blue, 27 }  ,fill opacity=0.33 ][line width=0.08]  [draw opacity=0] (8.93,-4.29) -- (0,0) -- (8.93,4.29) -- cycle    ;
\draw [color={rgb, 255:red, 208; green, 2; blue, 27 }  ,draw opacity=0.33 ][line width=2.25]    (356.42,275.24) -- (386.62,275.33) ;
\draw [color={rgb, 255:red, 208; green, 2; blue, 27 }  ,draw opacity=0.33 ]   (356.25,258.29) .. controls (391.57,244.79) and (388.68,247.19) .. (426.18,224.44) ;
\draw [color={rgb, 255:red, 208; green, 2; blue, 27 }  ,draw opacity=0.33 ]   (386.62,275.33) .. controls (434.98,243.35) and (440.56,205.83) .. (463.8,178.03) ;
\draw [shift={(465.62,175.91)}, rotate = 131.53] [fill={rgb, 255:red, 208; green, 2; blue, 27 }  ,fill opacity=0.33 ][line width=0.08]  [draw opacity=0] (8.93,-4.29) -- (0,0) -- (8.93,4.29) -- cycle    ;
\draw [color={rgb, 255:red, 208; green, 2; blue, 27 }  ,draw opacity=0.33 ]   (374.69,249.14) .. controls (376.97,240.28) and (414.84,223.56) .. (446.4,204.67) ;
\draw [color={rgb, 255:red, 208; green, 2; blue, 27 }  ,draw opacity=0.33 ]   (389.62,263.24) .. controls (437.73,231.42) and (436.26,194) .. (439.02,163.39) ;
\draw [shift={(439.29,160.57)}, rotate = 95.81] [fill={rgb, 255:red, 208; green, 2; blue, 27 }  ,fill opacity=0.33 ][line width=0.08]  [draw opacity=0] (8.93,-4.29) -- (0,0) -- (8.93,4.29) -- cycle    ;
\draw [color={rgb, 255:red, 208; green, 2; blue, 27 }  ,draw opacity=0.33 ][line width=2.25]    (356.25,258.29) -- (356.42,275.24) ;
\draw  [color={rgb, 255:red, 208; green, 2; blue, 27 }  ,draw opacity=0.33 ][fill={rgb, 255:red, 0; green, 0; blue, 0 }  ,fill opacity=1 ] (384.96,275.33) .. controls (384.96,274.42) and (385.7,273.67) .. (386.62,273.67) .. controls (387.54,273.67) and (388.28,274.42) .. (388.28,275.33) .. controls (388.28,276.25) and (387.54,277) .. (386.62,277) .. controls (385.7,277) and (384.96,276.25) .. (384.96,275.33) -- cycle ;
\draw [color={rgb, 255:red, 208; green, 2; blue, 27 }  ,draw opacity=0.33 ][line width=2.25]    (389.02,249.28) -- (388.89,263.42) ;
\draw  [color={rgb, 255:red, 208; green, 2; blue, 27 }  ,draw opacity=0.33 ][fill={rgb, 255:red, 0; green, 0; blue, 0 }  ,fill opacity=1 ] (373.02,249.14) .. controls (373.02,248.22) and (373.77,247.48) .. (374.69,247.48) .. controls (375.6,247.48) and (376.35,248.22) .. (376.35,249.14) .. controls (376.35,250.06) and (375.6,250.8) .. (374.69,250.8) .. controls (373.77,250.8) and (373.02,250.06) .. (373.02,249.14) -- cycle ;
\draw  [color={rgb, 255:red, 208; green, 2; blue, 27 }  ,draw opacity=0.33 ][fill={rgb, 255:red, 0; green, 0; blue, 0 }  ,fill opacity=1 ] (354.59,258.29) .. controls (354.59,257.37) and (355.33,256.63) .. (356.25,256.63) .. controls (357.17,256.63) and (357.91,257.37) .. (357.91,258.29) .. controls (357.91,259.21) and (357.17,259.95) .. (356.25,259.95) .. controls (355.33,259.95) and (354.59,259.21) .. (354.59,258.29) -- cycle ;
\draw [color={rgb, 255:red, 208; green, 2; blue, 27 }  ,draw opacity=0.33 ][line width=2.25]    (287.2,217.56) -- (302.48,217.6) ;
\draw [color={rgb, 255:red, 208; green, 2; blue, 27 }  ,draw opacity=0.33 ]   (288.69,202.04) .. controls (327.28,197.6) and (317.15,201.27) .. (356.48,195.6) ;
\draw [color={rgb, 255:red, 208; green, 2; blue, 27 }  ,draw opacity=0.33 ][line width=2.25]    (287.02,202.04) -- (287.2,217.56) ;
\draw  [color={rgb, 255:red, 208; green, 2; blue, 27 }  ,draw opacity=0.33 ][fill={rgb, 255:red, 0; green, 0; blue, 0 }  ,fill opacity=1 ] (285.36,202.04) .. controls (285.36,201.12) and (286.11,200.38) .. (287.02,200.38) .. controls (287.94,200.38) and (288.69,201.12) .. (288.69,202.04) .. controls (288.69,202.96) and (287.94,203.7) .. (287.02,203.7) .. controls (286.11,203.7) and (285.36,202.96) .. (285.36,202.04) -- cycle ;
\draw [color={rgb, 255:red, 208; green, 2; blue, 27 }  ,draw opacity=0.33 ]   (302.48,217.6) .. controls (339.73,192.83) and (386.67,190.69) .. (419.58,184.12) ;
\draw [shift={(422.08,183.6)}, rotate = 168.02] [fill={rgb, 255:red, 208; green, 2; blue, 27 }  ,fill opacity=0.33 ][line width=0.08]  [draw opacity=0] (8.93,-4.29) -- (0,0) -- (8.93,4.29) -- cycle    ;
\draw [color={rgb, 255:red, 208; green, 2; blue, 27 }  ,draw opacity=0.33 ][line width=2.25]    (374.69,249.14) -- (389.02,249.28) ;
\draw  [color={rgb, 255:red, 208; green, 2; blue, 27 }  ,draw opacity=0.33 ][fill={rgb, 255:red, 0; green, 0; blue, 0 }  ,fill opacity=1 ] (387.96,263.24) .. controls (387.96,262.32) and (388.7,261.58) .. (389.62,261.58) .. controls (390.54,261.58) and (391.28,262.32) .. (391.28,263.24) .. controls (391.28,264.16) and (390.54,264.9) .. (389.62,264.9) .. controls (388.7,264.9) and (387.96,264.16) .. (387.96,263.24) -- cycle ;
\draw  [draw opacity=0][fill={rgb, 255:red, 155; green, 155; blue, 155 }  ,fill opacity=0.15 ] (267.87,178.38) -- (342.89,178.38) -- (342.89,230.17) -- (267.87,230.17) -- cycle ;
\draw  [draw opacity=0][fill={rgb, 255:red, 155; green, 155; blue, 155 }  ,fill opacity=0.15 ] (342.87,248.09) -- (443.68,248.09) -- (443.68,301.6) -- (342.87,301.6) -- cycle ;
\draw  [draw opacity=0][fill={rgb, 255:red, 155; green, 155; blue, 155 }  ,fill opacity=0.15 ] (163.31,118.02) -- (252.32,118.02) -- (252.32,177.72) -- (163.31,177.72) -- cycle ;
\draw  [color={rgb, 255:red, 208; green, 2; blue, 27 }  ,draw opacity=0.33 ][fill={rgb, 255:red, 0; green, 0; blue, 0 }  ,fill opacity=1 ] (300.82,217.6) .. controls (300.82,216.68) and (301.56,215.94) .. (302.48,215.94) .. controls (303.4,215.94) and (304.14,216.68) .. (304.14,217.6) .. controls (304.14,218.52) and (303.4,219.26) .. (302.48,219.26) .. controls (301.56,219.26) and (300.82,218.52) .. (300.82,217.6) -- cycle ;
\draw [color={rgb, 255:red, 208; green, 2; blue, 27 }  ,draw opacity=0.33 ][line width=2.25]    (301.82,188.44) -- (316.95,188.55) ;
\draw [color={rgb, 255:red, 208; green, 2; blue, 27 }  ,draw opacity=0.33 ]   (303.49,188.44) .. controls (338.8,174.94) and (319.68,166) .. (360.48,133.2) ;
\draw [color={rgb, 255:red, 208; green, 2; blue, 27 }  ,draw opacity=0.33 ][line width=2.25]    (316.95,188.55) -- (317.12,205.5) ;
\draw  [color={rgb, 255:red, 208; green, 2; blue, 27 }  ,draw opacity=0.33 ][fill={rgb, 255:red, 0; green, 0; blue, 0 }  ,fill opacity=1 ] (300.16,188.44) .. controls (300.16,187.52) and (300.91,186.78) .. (301.82,186.78) .. controls (302.74,186.78) and (303.49,187.52) .. (303.49,188.44) .. controls (303.49,189.36) and (302.74,190.1) .. (301.82,190.1) .. controls (300.91,190.1) and (300.16,189.36) .. (300.16,188.44) -- cycle ;
\draw [color={rgb, 255:red, 208; green, 2; blue, 27 }  ,draw opacity=0.33 ]   (317.28,204) .. controls (354.33,179.36) and (364.52,122.36) .. (368.74,93.04) ;
\draw [shift={(369.12,90.4)}, rotate = 98.02] [fill={rgb, 255:red, 208; green, 2; blue, 27 }  ,fill opacity=0.33 ][line width=0.08]  [draw opacity=0] (8.93,-4.29) -- (0,0) -- (8.93,4.29) -- cycle    ;
\draw  [color={rgb, 255:red, 208; green, 2; blue, 27 }  ,draw opacity=0.33 ][fill={rgb, 255:red, 0; green, 0; blue, 0 }  ,fill opacity=1 ] (315.62,204) .. controls (315.62,203.08) and (316.36,202.34) .. (317.28,202.34) .. controls (318.2,202.34) and (318.94,203.08) .. (318.94,204) .. controls (318.94,204.92) and (318.2,205.66) .. (317.28,205.66) .. controls (316.36,205.66) and (315.62,204.92) .. (315.62,204) -- cycle ;
\draw [color={rgb, 255:red, 208; green, 2; blue, 27 }  ,draw opacity=0.33 ][line width=2.25]    (233.74,134.24) -- (233.89,150.96) ;
\draw  [fill={rgb, 255:red, 0; green, 0; blue, 0 }  ,fill opacity=1 ] (266.8,177.87) .. controls (266.8,176.95) and (267.55,176.2) .. (268.47,176.2) .. controls (269.38,176.2) and (270.13,176.95) .. (270.13,177.87) .. controls (270.13,178.78) and (269.38,179.53) .. (268.47,179.53) .. controls (267.55,179.53) and (266.8,178.78) .. (266.8,177.87) -- cycle ;
\draw  [fill={rgb, 255:red, 0; green, 0; blue, 0 }  ,fill opacity=1 ] (250.63,177.7) .. controls (250.63,176.78) and (251.38,176.04) .. (252.29,176.04) .. controls (253.21,176.04) and (253.96,176.78) .. (253.96,177.7) .. controls (253.96,178.62) and (253.21,179.36) .. (252.29,179.36) .. controls (251.38,179.36) and (250.63,178.62) .. (250.63,177.7) -- cycle ;
\draw [color={rgb, 255:red, 0; green, 90; blue, 200 }  ,draw opacity=1 ][line width=2.25]    (342.87,248.09) -- (342.89,230.17) ;
\draw  [fill={rgb, 255:red, 0; green, 0; blue, 0 }  ,fill opacity=1 ] (341.23,230.17) .. controls (341.23,229.25) and (341.98,228.5) .. (342.89,228.5) .. controls (343.81,228.5) and (344.56,229.25) .. (344.56,230.17) .. controls (344.56,231.08) and (343.81,231.83) .. (342.89,231.83) .. controls (341.98,231.83) and (341.23,231.08) .. (341.23,230.17) -- cycle ;
\draw  [fill={rgb, 255:red, 0; green, 0; blue, 0 }  ,fill opacity=1 ] (341.2,248.09) .. controls (341.2,247.17) and (341.95,246.43) .. (342.87,246.43) .. controls (343.78,246.43) and (344.53,247.17) .. (344.53,248.09) .. controls (344.53,249.01) and (343.78,249.75) .. (342.87,249.75) .. controls (341.95,249.75) and (341.2,249.01) .. (341.2,248.09) -- cycle ;
\draw [color={rgb, 255:red, 208; green, 2; blue, 27 }  ,draw opacity=0.33 ][line width=2.25]    (409.42,276.08) -- (424.55,276.19) ;
\draw [color={rgb, 255:red, 208; green, 2; blue, 27 }  ,draw opacity=0.33 ]   (411.09,276.08) .. controls (446.4,262.58) and (434.08,258.8) .. (474.88,226) ;
\draw [color={rgb, 255:red, 208; green, 2; blue, 27 }  ,draw opacity=0.33 ][line width=2.25]    (424.55,276.19) -- (424.72,293.14) ;
\draw  [color={rgb, 255:red, 208; green, 2; blue, 27 }  ,draw opacity=0.33 ][fill={rgb, 255:red, 0; green, 0; blue, 0 }  ,fill opacity=1 ] (407.76,276.08) .. controls (407.76,275.16) and (408.51,274.42) .. (409.42,274.42) .. controls (410.34,274.42) and (411.09,275.16) .. (411.09,276.08) .. controls (411.09,277) and (410.34,277.74) .. (409.42,277.74) .. controls (408.51,277.74) and (407.76,277) .. (407.76,276.08) -- cycle ;
\draw [color={rgb, 255:red, 208; green, 2; blue, 27 }  ,draw opacity=0.33 ]   (424.88,291.64) .. controls (462.51,266.62) and (448.13,234.19) .. (499.3,210.29) ;
\draw [shift={(501.68,209.2)}, rotate = 156.04] [fill={rgb, 255:red, 208; green, 2; blue, 27 }  ,fill opacity=0.33 ][line width=0.08]  [draw opacity=0] (8.93,-4.29) -- (0,0) -- (8.93,4.29) -- cycle    ;
\draw  [color={rgb, 255:red, 208; green, 2; blue, 27 }  ,draw opacity=0.33 ][fill={rgb, 255:red, 0; green, 0; blue, 0 }  ,fill opacity=1 ] (423.22,291.64) .. controls (423.22,290.72) and (423.96,289.98) .. (424.88,289.98) .. controls (425.8,289.98) and (426.54,290.72) .. (426.54,291.64) .. controls (426.54,292.56) and (425.8,293.3) .. (424.88,293.3) .. controls (423.96,293.3) and (423.22,292.56) .. (423.22,291.64) -- cycle ;

\draw (249,183.1) node [anchor=north west][inner sep=0.75pt]    {\textcolor{blue}{$B_{1}$}};
\draw (320,232) node [anchor=north west][inner sep=0.75pt]    {\textcolor{blue}{$B_{2}$}};
\draw (453.33,51.73) node [anchor=north west][inner sep=0.75pt]    {$\boldsymbol\xi[\rho]$};
\draw (162.24,152.92) node [anchor=north west][inner sep=0.75pt]    {\textcolor{red}{$ \begin{array}{l}
R_{1}\\
\end{array}$}};
\draw (222.24,151.72) node [anchor=north west][inner sep=0.75pt]    {\textcolor{red}{$ \begin{array}{l}
R_{2}\\
\end{array}$}};
\draw (271.44,218.52) node [anchor=north west][inner sep=0.75pt]    {\textcolor{red}{$ \begin{array}{l}
R_{3}\\
\end{array}$}};
\draw (291.84,165.32) node [anchor=north west][inner sep=0.75pt]    {\textcolor{red}{$ \begin{array}{l}
R_{4}\\
\end{array}$}};
\draw (344.87,278.25) node [anchor=north west][inner sep=0.75pt]    {\textcolor{red}{$ \begin{array}{l}
R_{5}\\
\end{array}$}};
\draw (395.47,279.45) node [anchor=north west][inner sep=0.75pt]    {\textcolor{red}{$ \begin{array}{l}
R_{7}\\
\end{array}$}};
\draw (359.87,251.8) node [anchor=north west][inner sep=0.75pt]    {\textcolor{red}{$\begin{array}{l}
R_{6}\\
\end{array}$}};

\end{tikzpicture}

\captionsetup{width=.8\linewidth}
\caption{Setting of Theorem \ref{thm_1}, in the case $n=2$. 
Increasing the Busemann increments $B_1$ and $B_2$ (shown in blue), tends to probabilistically decrease the Busemann functions within any of the gray regions.
This diagram shows $k=7$ such Busemann functions (in light red), which can have any directions and any vertex pairs $(\mathbf a_j,\mathbf b_j)$ that respect the down-right order $\preceq$.}
\label{fig_main}
\end{center}
\end{figure}
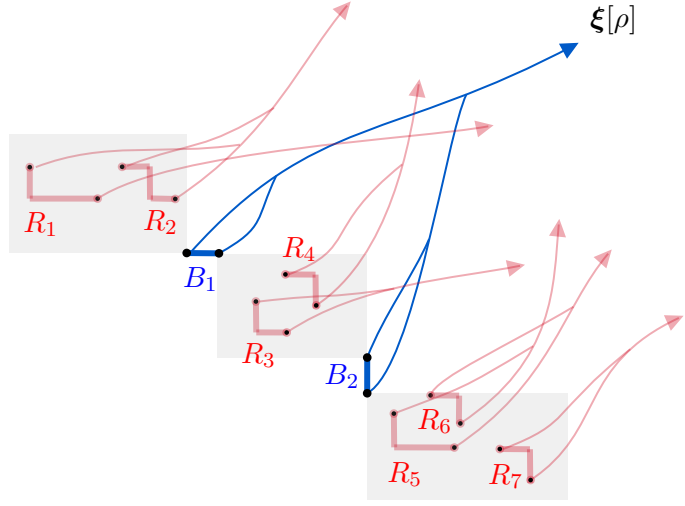

\begin{theorem}[Negative association] \label{thm_1}
For any non-negative functions $f_1,\dots,f_n\colon\R\to[0,\infty)$ and $g\colon\R^k\to[0,\infty)$ which are all coordinate-wise monotone in the same direction,
we have
\[
\mathbb{E}\Big[g\big(R_1, \dots, R_k\big) \prod_{i=1}^n f_i(B_i)\Big] \le \mathbb{E}\Big[g\big(R_1, \dots, R_k\big)\Big] \prod_{i=1}^n \mathbb{E}\Big[f_i(B_i)\Big].
\]
\end{theorem}

\begin{remark}[Terminology disclaimer] \label{rem_terminology1}
In general, random variables $X_1,\ldots,X_N$ are said to be \textit{negatively associated} if for any disjoint subsets $\mathcal I,\mathcal{J}$ of $\{1,\ldots,N\}$, and any coordinate-wise non-decreasing functions $f$ and $g$, one has
\begin{align} \label{na_term}
\mathbb E\Big[f\big((X_i)_{i\in\mathcal I}\big)g\big((X_j)_{j\in\mathcal J}\big)\Big] \le
\mathbb E\Big[f\big((X_i)_{i\in\mathcal I}\big)\Big]\mathbb E\Big[g\big((X_j)_{j\in\mathcal J}\big)\Big].
\end{align}
Theorem~\ref{thm_1} is slightly weaker than this for the collection $B_1,\ldots,B_n,R_1,\ldots,R_k$, since we are insisting on specific types of subsets.
Nevertheless, the result is strong enough for our applications, discussed below.
One consequence of Theorem~\ref{thm_1} is that any $B_i$ is negatively associated with the random vector $(R_1,\ldots,R_k)$, in the sense that \eqref{na_term} holds with $(X_i)_{i\in\mathcal I} = B_i$ and $(X_j)_{j\in\mathcal J} = (R_1,\ldots,R_k)$.
\end{remark}

There are several corollaries, which we present next.
Let $\mathcal{Y} = (\mathbf y_i)_{i\in\mathbb Z}$ be a down-right path, meaning
$\mathbf y_{i}-\mathbf y_{i-1} \in \{\mathbf e_1,-\mathbf e_2\}$ for every $i$.
The following result gives control of the joint moment generating function of Busemann increments along $\mathcal{Y}$.

\begin{corollary}[Domination of MGF by products]\label{mgf}
For any direction parameters $\rho_1, \dots, \rho_n \in (0,1)$, any down-right path $\mathcal{Y} = (\mathbf y_i)_{i\in \mathbb{Z}}$, and any $\alpha_1,\ldots,\alpha_n\ge0$, we have
\begin{subequations} 
\begin{align} \label{mgf_a}
\mathbb{E}\Big[\exp\Big(\sum_{i=1}^n \alpha_i B^{\rho_i}_{\mathbf y_i, \mathbf y_{i+1}}\Big)\Big]
&\leq \prod_{i=1}^n \mathbb{E}\big[\exp\big(\alpha_i B^{\rho_i}_{\mathbf y_i, \mathbf y_{i+1}}\big)\big] \\
\text{and}\quad \label{mgf_b}
\mathbb{E}\Big[\exp\Big(-\sum_{i=1}^n \alpha_i B^{\rho_i}_{\mathbf y_i, \mathbf y_{i+1}}\Big)\Big]
&\leq \prod_{i=1}^n \mathbb{E}\big[\exp\big(-\alpha_i B^{\rho_i}_{\mathbf y_i, \mathbf y_{i+1}}\big)\big].
\end{align}
\end{subequations}
\end{corollary}

This, in turn, leads to an exponential concentration inequality at the diffusive scale $\sqrt{n}$ for sums of Busemann increments, even in the absence of independence.

\begin{corollary}[Concentration of sums] \label{thm_diffusive}
Fix $\varepsilon >0$. 
For any $\rho_1,\rho_2, \dots,\rho_n \in [\varepsilon,1 -\varepsilon]$, any down-right path $\mathcal{Y}= (\mathbf y_i)_{i\in \mathbb{Z}}$, and any $t\ge 0$, we have
\begin{subequations}
\begin{align} \label{diffusive_a}
\mathbb{P}\Big(\sum_{i=1}^n B^{\rho_{i}}_{\mathbf y_i, \mathbf y_{i+1}} - \mathbb{E}\Big[\sum_{i=1}^n B^{\rho_{i}}_{\mathbf y_i, \mathbf y_{i+1}} \Big]\geq t \sqrt n \Big) 
&\leq e^{-\min\{\varepsilon^2t^2, \varepsilon t\sqrt n\}/4}\\
\text{and} \quad \label{diffusive_b}
\mathbb{P}\Big(\sum_{i=1}^n B^{\rho_{i}}_{\mathbf y_i, \mathbf y_{i+1}} - \mathbb{E}\Big[\sum_{i=1}^n B^{\rho_{i}}_{\mathbf y_i, \mathbf y_{i+1}} \Big]\leq -t \sqrt n \Big) 
&\leq e^{-\min\{\varepsilon^2t^2, \varepsilon t\sqrt n\}/4}.
\end{align}
\end{subequations}
\end{corollary}

From Theorem \ref{thm_1}, one can also determine that Busemann increments along a down-right path are negatively dependent in the following sense; see Figure~\ref{fig_titlepage} for an illustration.

\begin{corollary}[Negative dependence]\label{downright_path}
For every down-right path $\mathcal{Y} = \{\mathbf y_i\}_{i\in \mathbb{Z}}$, every choice of $\rho_1,\dots,\rho_n \in (0,1)$, and every $t_1,\dots,t_n \in \mathbb{R}$, we have
\begin{subequations} \label{negative_orthant_dependence}
\begin{align} \label{negative_orthant_dependence_a}
\mathbb{P}\Big(\bigcap_{i=1}^n \Big\{B^{\rho_{i}}_{\mathbf y_{i}, \mathbf y_{i+1}} > t_i\Big\}\Big) 
&\leq \prod_{i=1}^n\mathbb{P}\big( B^{\rho_{i}}_{\mathbf y_{i}, \mathbf y_{i+1}} > t_i\big) \\
\text{and} \quad \label{negative_orthant_dependence_b}
\mathbb{P}\Big(\bigcap_{i=1}^n \Big\{B^{\rho_{i}}_{\mathbf y_{i}, \mathbf y_{i+1}} < t_i\Big\}\Big) 
&\leq \prod_{i=1}^n\mathbb{P}\big( B^{\rho_{i}}_{\mathbf y_{i}, \mathbf y_{i+1}} < t_i\big).
\end{align}
\end{subequations}
\end{corollary}

\begin{remark}[Notions of negative dependence and vector dominance] \label{rem_terminology2}
In the language of \cite{joagdev_proschan83}, \eqref{negative_orthant_dependence} says that 
$B^{\rho_{1}}_{\mathbf y_{1}, \mathbf y_{2}},\ldots,B^{\rho_{n}}_{\mathbf y_{n}, \mathbf y_{n+1}}$ are \textit{negatively orthant dependent}.
In general, this is weaker than negative association as defined in Remark~\ref{rem_terminology1}.
Alternatively, in the language of \cite{joe90}, \eqref{negative_orthant_dependence} says that the random vector $(B^{\rho_{i}}_{\mathbf y_{i}, \mathbf y_{i+1}})_{i=1}^n$ is \textit{less concordant} than $(\wh B^{\rho_{i}}_{\mathbf y_{i}, \mathbf y_{i+1}})_{i=1}^n$, where $\wh B^{\rho_{i}}_{\mathbf y_{i}, \mathbf y_{i+1}}\stackrel{\textup{d}}{=}B^{\rho_{i}}_{\mathbf y_{i}, \mathbf y_{i+1}}$ and $\wh B^{\rho_{1}}_{\mathbf y_{1}, \mathbf y_{2}},\ldots,\wh B^{\rho_{n}}_{\mathbf y_{n}, \mathbf y_{n+1}}$ are independent.
If $n=2$, this concordance domination is equivalent to supermodular domination, e.g.\ see \cite[Theorem~6]{epstein_tanny80}.
\end{remark}

\begin{remark}[Equality in \eqref{negative_orthant_dependence} when the direction parameters are non-increasing]\label{rem_ind}
It is known that if $1> \rho_1 \geq \rho_2 \geq \dots \geq \rho_n>0$ (as in the left of Figure \ref{fig:thm1}), then the collection of random variables $\{B^{\rho_i}_{\mathbf{y}_i, \mathbf{y}_{i+1}}\}_{i=1}^n$ is independent.
In the case $\rho_1=\rho_2=\cdots=\rho_n$, this fact goes back to \cite{balazs_cator_seppalainen06,cator_pimentel12}; see also \cite[Theorem~4.2(i)]{seppalainen20a}.
When inequalities are allowed, this independence was established in \cite[Theorem~2.1]{ind_buse}.
In Corollary \ref{downright_path} we are not assuming any ordering (see the right of Figure \ref{fig:thm1}), and independence is replaced with negative dependence.  
\end{remark}

\begin{figure}[t]
\begin{center}

\tikzset{every picture/.style={line width=0.75pt}} 

\begin{tikzpicture}[x=0.75pt,y=0.75pt,yscale=-0.8,xscale=0.8]

\draw  [dash pattern={on 0.84pt off 2.51pt}]  (22.2,102.3) -- (62.2,102.3) ;
\draw  [dash pattern={on 0.84pt off 2.51pt}]  (62.2,101.3) -- (62.2,161.3) ;
\draw    (62.2,162.3) -- (82.2,162.3) ;
\draw  [dash pattern={on 0.84pt off 2.51pt}]  (82.2,161.3) -- (82.2,181.3) ;
\draw  [dash pattern={on 0.84pt off 2.51pt}]  (82.2,181.3) -- (182.2,181.3) ;
\draw  [dash pattern={on 0.84pt off 2.51pt}]  (182.2,181.3) -- (182.2,221.3) ;
\draw  [dash pattern={on 0.84pt off 2.51pt}]  (182.2,221.3) -- (282.2,221.3) ;
\draw  [dash pattern={on 0.84pt off 2.51pt}]  (282.2,222.3) -- (282.2,242.3) ;
\draw [color={rgb, 255:red, 208; green, 2; blue, 27 }  ,draw opacity=1 ][line width=2.25]    (62.2,102.3) -- (62.2,122.3) ;
\draw [color={rgb, 255:red, 208; green, 2; blue, 27 }  ,draw opacity=1 ]   (62.2,102.3) .. controls (104.77,68.64) and (258.39,53.55) .. (316.76,50.83) ;
\draw [shift={(318.5,50.75)}, rotate = 177.49] [color={rgb, 255:red, 208; green, 2; blue, 27 }  ,draw opacity=1 ][line width=0.75]    (10.93,-3.29) .. controls (6.95,-1.4) and (3.31,-0.3) .. (0,0) .. controls (3.31,0.3) and (6.95,1.4) .. (10.93,3.29)   ;
\draw [color={rgb, 255:red, 208; green, 2; blue, 27 }  ,draw opacity=1 ]   (62.2,122.3) .. controls (102.2,92.3) and (74.5,109.25) .. (114.5,79.25) ;
\draw [color={rgb, 255:red, 74; green, 144; blue, 226 }  ,draw opacity=1 ][line width=2.25]    (62.2,162.3) -- (82.2,162.3) ;
\draw [color={rgb, 255:red, 74; green, 144; blue, 226 }  ,draw opacity=1 ]   (62.2,162.3) .. controls (99.44,102.52) and (282.26,54.72) .. (313.76,42.93) ;
\draw [shift={(315.5,42.25)}, rotate = 157.75] [color={rgb, 255:red, 74; green, 144; blue, 226 }  ,draw opacity=1 ][line width=0.75]    (10.93,-3.29) .. controls (6.95,-1.4) and (3.31,-0.3) .. (0,0) .. controls (3.31,0.3) and (6.95,1.4) .. (10.93,3.29)   ;
\draw [color={rgb, 255:red, 74; green, 144; blue, 226 }  ,draw opacity=1 ]   (82.2,162.3) .. controls (105.32,140.28) and (112.5,132.25) .. (137,107.75) ;
\draw [color={rgb, 255:red, 245; green, 166; blue, 35 }  ,draw opacity=1 ][line width=2.25]    (122.2,181.1) -- (142.2,181.1) ;
\draw [color={rgb, 255:red, 245; green, 166; blue, 35 }  ,draw opacity=1 ]   (122.2,181.1) .. controls (161.8,151.4) and (273.24,72.16) .. (306.52,34.19) ;
\draw [shift={(307.5,33.05)}, rotate = 130.48] [color={rgb, 255:red, 245; green, 166; blue, 35 }  ,draw opacity=1 ][line width=0.75]    (10.93,-3.29) .. controls (6.95,-1.4) and (3.31,-0.3) .. (0,0) .. controls (3.31,0.3) and (6.95,1.4) .. (10.93,3.29)   ;
\draw [color={rgb, 255:red, 245; green, 166; blue, 35 }  ,draw opacity=1 ]   (142.2,181.1) .. controls (166.2,164.1) and (191.5,154.05) .. (199.5,124.05) ;
\draw [color={rgb, 255:red, 126; green, 211; blue, 33 }  ,draw opacity=1 ][line width=2.25]    (203.2,221.3) -- (223.2,221.3) ;
\draw [color={rgb, 255:red, 126; green, 211; blue, 33 }  ,draw opacity=1 ]   (203.2,221.3) .. controls (242.8,191.6) and (282.99,90.74) .. (297.08,32.5) ;
\draw [shift={(297.5,30.75)}, rotate = 103.3] [color={rgb, 255:red, 126; green, 211; blue, 33 }  ,draw opacity=1 ][line width=0.75]    (10.93,-3.29) .. controls (6.95,-1.4) and (3.31,-0.3) .. (0,0) .. controls (3.31,0.3) and (6.95,1.4) .. (10.93,3.29)   ;
\draw [color={rgb, 255:red, 126; green, 211; blue, 33 }  ,draw opacity=1 ]   (223.2,221.3) .. controls (247.2,204.3) and (245,181.25) .. (253,151.25) ;
\draw  [dash pattern={on 0.84pt off 2.51pt}]  (394.2,102.3) -- (434.2,102.3) ;
\draw  [dash pattern={on 0.84pt off 2.51pt}]  (434.2,101.3) -- (434.2,161.3) ;
\draw    (434.2,162.3) -- (454.2,162.3) ;
\draw  [dash pattern={on 0.84pt off 2.51pt}]  (454.2,161.3) -- (454.2,181.3) ;
\draw  [dash pattern={on 0.84pt off 2.51pt}]  (454.2,181.3) -- (554.2,181.3) ;
\draw  [dash pattern={on 0.84pt off 2.51pt}]  (554.2,181.3) -- (554.2,221.3) ;
\draw  [dash pattern={on 0.84pt off 2.51pt}]  (554.2,221.3) -- (654.2,221.3) ;
\draw  [dash pattern={on 0.84pt off 2.51pt}]  (654.2,222.3) -- (654.2,242.3) ;
\draw [color={rgb, 255:red, 208; green, 2; blue, 27 }  ,draw opacity=1 ][line width=2.25]    (434.2,102.3) -- (434.2,122.3) ;
\draw [color={rgb, 255:red, 208; green, 2; blue, 27 }  ,draw opacity=1 ]   (434.2,102.3) .. controls (476.77,68.64) and (486.21,80.16) .. (523.27,53.62) ;
\draw [shift={(524.4,52.8)}, rotate = 144.01] [color={rgb, 255:red, 208; green, 2; blue, 27 }  ,draw opacity=1 ][line width=0.75]    (10.93,-3.29) .. controls (6.95,-1.4) and (3.31,-0.3) .. (0,0) .. controls (3.31,0.3) and (6.95,1.4) .. (10.93,3.29)   ;
\draw [color={rgb, 255:red, 208; green, 2; blue, 27 }  ,draw opacity=1 ]   (434.2,122.3) .. controls (474.2,92.3) and (453.2,100.3) .. (493.2,70.3) ;
\draw [color={rgb, 255:red, 74; green, 144; blue, 226 }  ,draw opacity=1 ][line width=2.25]    (434.2,162.3) -- (454.2,162.3) ;
\draw [color={rgb, 255:red, 74; green, 144; blue, 226 }  ,draw opacity=1 ]   (434.2,162.3) .. controls (471.44,102.52) and (573.3,92.89) .. (601.56,82.63) ;
\draw [shift={(603.2,82)}, rotate = 157.75] [color={rgb, 255:red, 74; green, 144; blue, 226 }  ,draw opacity=1 ][line width=0.75]    (10.93,-3.29) .. controls (6.95,-1.4) and (3.31,-0.3) .. (0,0) .. controls (3.31,0.3) and (6.95,1.4) .. (10.93,3.29)   ;
\draw [color={rgb, 255:red, 74; green, 144; blue, 226 }  ,draw opacity=1 ]   (454.2,162.3) .. controls (477.32,140.28) and (469.2,129.3) .. (523.2,103.3) ;
\draw [color={rgb, 255:red, 245; green, 166; blue, 35 }  ,draw opacity=1 ][line width=2.25]    (494,181.1) -- (514,181.1) ;
\draw [color={rgb, 255:red, 245; green, 166; blue, 35 }  ,draw opacity=1 ]   (494,181.1) .. controls (533.4,151.55) and (534.96,106.48) .. (555.07,67.86) ;
\draw [shift={(556,66.1)}, rotate = 118.3] [color={rgb, 255:red, 245; green, 166; blue, 35 }  ,draw opacity=1 ][line width=0.75]    (10.93,-3.29) .. controls (6.95,-1.4) and (3.31,-0.3) .. (0,0) .. controls (3.31,0.3) and (6.95,1.4) .. (10.93,3.29)   ;
\draw [color={rgb, 255:red, 245; green, 166; blue, 35 }  ,draw opacity=1 ]   (514,181.1) .. controls (538,164.1) and (526.2,150.9) .. (534.2,120.9) ;
\draw [color={rgb, 255:red, 126; green, 211; blue, 33 }  ,draw opacity=1 ][line width=2.25]    (575.2,221.3) -- (595.2,221.3) ;
\draw [color={rgb, 255:red, 126; green, 211; blue, 33 }  ,draw opacity=1 ]   (575.2,221.3) .. controls (615,191.45) and (620.15,116.06) .. (682.26,76.89) ;
\draw [shift={(683.2,76.3)}, rotate = 148.24] [color={rgb, 255:red, 126; green, 211; blue, 33 }  ,draw opacity=1 ][line width=0.75]    (10.93,-3.29) .. controls (6.95,-1.4) and (3.31,-0.3) .. (0,0) .. controls (3.31,0.3) and (6.95,1.4) .. (10.93,3.29)   ;
\draw [color={rgb, 255:red, 126; green, 211; blue, 33 }  ,draw opacity=1 ]   (595.2,221.3) .. controls (619.2,204.3) and (615.2,178.3) .. (623.2,148.3) ;

\end{tikzpicture}

\captionsetup{width=.8\linewidth}
\caption{Illustration of Remark \ref{rem_ind}. On the left, the direction parameters are non-increasing ($\rho_i \geq \rho_{i+1}$), meaning the asymptotic slopes are non-decreasing and thus cross one another.
On the right, the direction parameters are allowed to be non-monotone.}
\label{fig:thm1}
\end{center}
\end{figure}
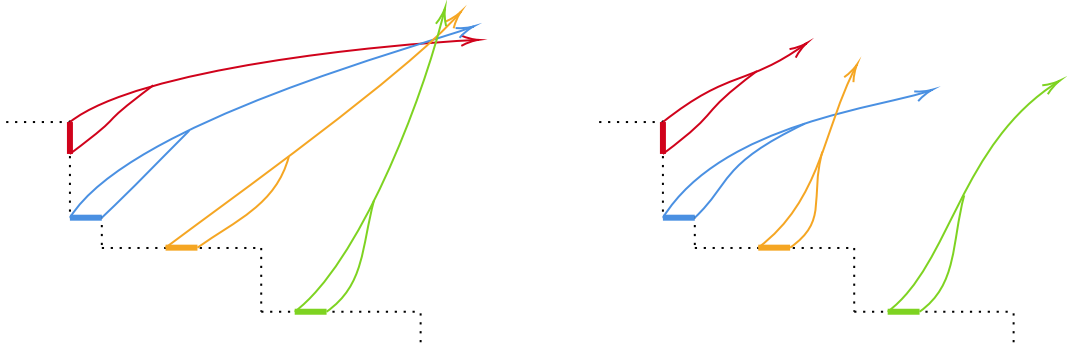

In addition to the corollaries just presented, the proof strategy of Theorem~\ref{thm_1} allows us to improve the independence result from \cite{ind_buse} stated in Remark \ref{rem_ind}.

\begin{theorem}[Improved independence]\label{imp_ind}
Fix a direction parameter $\lambda\in(0,1)$ and a unit edge $[\![\mathbf{a}, \mathbf{b}]\!]$, where $\mathbf {b}\in\{\mathbf a+\bfe_1,\mathbf a-\bfe_2\}$.
Let $\mathcal{F}_{\preceq}$ denote the $\sigma$-algebra generated by all Busemann functions $B^\rho_{\mathbf x,\mathbf y}$ of the form
\begin{align} \label{ind_order}
\text{$\rho\ge\lambda$ and $\mathbf x,\mathbf y\preceq \mathbf a$}
\qquad \text{or} \qquad
\text{$\rho\le\lambda$ and $\mathbf x,\mathbf y\succeq \mathbf b$.}
\end{align}
Then $B^\lambda_{\mathbf a, \mathbf b}$ is independent of $\mathcal{F}_\preceq$.
\end{theorem}

\begin{remark}[Differences from earlier result]
The result mentioned in Remark~\ref{rem_ind} gives independence along a single down-right path passing through $\{\mathbf a,\mathbf b\}$, with totally ordered direction parameters.
By comparison, Theorem~\ref{imp_ind} covers the union of all such paths, and only requires the direction parameters to satisfy \eqref{ind_order}.
The tradeoff is that we obtain independence only for a pair of objects, rather than mutual independence of many Busemann increments.
\end{remark}

Finally, we state a proposition demonstrating that the inequality in Theorem \ref{thm_1} can be strict, which shows the correlation structure of Busemann functions across different directions is truly nonzero. 

\begin{proposition}[Example of dependence]\label{prop_strict}
If $0 < \lambda < \rho < 1$, then
$\Cov\big(B^\lambda_{-\mathbf e_1, \mathbf{0}},B^\rho_{\mathbf{0}, \mathbf e_1}\big)<0.$
\end{proposition}

\subsection{Related literature}

To our knowledge, the only other negative correlation result for Busemann functions is due to Alevy and Krishnan~\cite{alevy_krishnan22}, who studied adjacent increments $(B^\rho_{\mathbf{e}_2,\mathbf{0}}, B^\rho_{\mathbf{0}, \mathbf{e}_1})$ under general i.i.d.\ weights. Building on a convexity argument by Sepp\"al\"ainen that provided a sufficient condition for negative correlation, Alevy and Krishnan identified an explicit family of weight distributions for which this condition can be rigorously verified.
The exponential case is not included, since $B^\rho_{\mathbf{e}_2,\mathbf{0}}$ and $B^\rho_{\mathbf{0}, \mathbf{e}_1}$ are independent in this setting.
This makes our results effectively disjoint from theirs; nevertheless, our proof strategy exploits several hidden monotonicity properties of two-dimensional LPP that do not rely on the choice of exponential weights.

Despite the exactly-solvable structure of exponential LPP, the joint distribution of Busemann functions across different parameters is much more complicated than the distribution in a single parameter.
The first distributional description of the full Busemann process was given in \cite{fan_seppalainen20}.
It is a queueing-theoretic description, making use of ideas originally developed for multiclass particle systems \cite{ferrari_martin06,ferrari_martin07,ferrari_martin09}.
Besides aiding the strategy of the present paper, the distribution from \cite{fan_seppalainen20} has played a role in several other developments, including geodesic coalesence estimates \cite{seppalainen_shen20}, local agreement between finite and semi-infinite geodesics \cite{balazs_busani_seppalainen21}, local agreement between geodesics with respect to different initial conditions \cite{uni_tree}, and a full description of the geometry of all infinite geodesics \cite{balazs_busani_seppalainen20,janjigian_rassoulagha_seppalainen23}.

More recently, a second description of the Busemann process was found in \cite{bates_emrah_martin_seppalainen_sorensen26}, by identifying a finite-volume inhomogeneous LPP (i.e.~weights are still independent exponential variables, but with different rate parameters) whose increments have the same joint law as the Busemann increments.
The proof leveraged a distributional symmetry under permutations of the rate parameters, as well as the existence and properties of Busemann functions for inhomogeneous models \cite{emrah_janjigian_seppalainen25}.
We too use manipulations of rate parameters, but in a simpler setting.
Our negative association result~\ref{thm_1} must be true for the inhomogeneous finite-volume LPP in \cite[Theorem~2.4]{bates_emrah_martin_seppalainen_sorensen26}, and it would be interesting to see a direct proof not passing through Busemann functions. 

Descriptions of joint Busemann processes (also known as stationary horizons, and whose laws are called jointly invariant measures) have been completed in other settings, including zero-temperature models such Brownian LPP and the directed landscape \cite{busani24, seppalainen_sorensen23a, seppalainen_sorensen23b, busani_seppalainen_sorensen24a}, positive-temperature models such as the inverse-gamma polymer \cite{bates_fan_seppalainen25} and the KPZ equation \cite{groathouse_rassoulagha_seppalainen_sorensen25}, periodic models \cite{corwin_gu_sorensen26}, and half-space models \cite{dauvergne_zhang26_arxiv, das_sorensen_yang26_arxiv}.
Moreover, these processes are related to one another by various scaling limits, e.g.\ \cite{busani_seppalainen_sorensen24b}.
This makes us all the more optimistic that our results can be adapted to other solvable settings---if not further---an endeavor we leave for future work.

\subsection{Organization of the paper}
In Section~\ref{q_basic}, we define relevant queueing maps and state several of their properties. 
In Section~\ref{proof1}, we prove a negative association result for the increments in a closely related model known as increment-stationary LPP.
Using that result, we prove in Section~\ref{proof_main} the main theorems and corollaries.

\vspace{3mm}
\noindent\textbf{Acknowledgements.} 
E.B. was partially supported by National Science Foundation grant DMS-2412473.

\section{Preliminaries}\label{q_basic}

In this section, we record several useful tools that will be needed in later proofs.
Sections~\ref{subsec_queueing_map} through \ref{subsec_monotonicities} are purely deterministic; in particular, there is no mention of exponentially distributed weights. 
This distributional hypothesis is only invoked in Section~\ref{subsec_burke}, which discusses the Burke property.

\subsection{The upward queueing map} \label{subsec_queueing_map}

Let $\omega = \{\omega_{\mathbf{z}}\}_{\mathbf{z} \in \mathbb{Z}^2}$ 
be a collection of non-negative numbers which can be thought of as weights attached to the vertices of $\mathbb{Z}^2$. 
We will define another environment $ h = \{h^m\}_{m \in \mathbb{Z}}$ which can be viewed as a transformation of the vertex weights to horizontal unit-edge weights. 
More precisely, for each fixed $m$, define the function $h^m\colon \mathbb{Z} \to \mathbb{R}$ such that 
\begin{equation}\label{defh}
h^m(0) = 0 \qquad \text{and} \qquad h^m(i)  - h^m(i-1) = \omega_{(i,m)} \quad  \textup{ for }i \in \mathbb{Z}.
\end{equation}
The value $h^m(i)  - h^m(i-1)$ can be thought of as a weight attached to the unit edge $[\![(i-1, m), (i,m)]\!]$. For readers who are familiar with the semi-discrete and continuous models, $h^m$ can be thought of as the discrete version of a two-sided Brownian motion.
We refer to the horizontal line $\{(x,m):\, x\in\Z\}$ as \textit{level} $m$, and regard $h^m$ as the environment on level $m$.

The two environments $\omega$ and $h$ are equivalent, and the inverse map recovering $\omega$ from $h$ is given directly by taking the differences in \eqref{defh}. In addition, recalling the definition of the last-passage value $G_{\mathbf{x}, \mathbf{y}}$ from \eqref{def_G}, we have 
\begin{equation}\label{defGh}
G_{(a,m), (b,n)} = G_{(a,m), (b,n)} (h) = \max_{a = x_m \leq \dots \leq x_{n+1} =b} \sum_{j=m}^n h^j(x_{j+1}) - h^j(x_{j}-1).
\end{equation}
For an illustration of this alternative definition, see Figure \ref{figh}.
The notation \( G_{(a,m), (b,n)}(h) \) indicates that the last-passage value \( G_{(a,m), (b,n)} \) is computed with respect to the environment \( h \). 
This will be especially important later when we compare $G_{(a,m), (b,n)}(h)$ and $G_{(a,m), (b,n)}(h')$ for two different environments $h$ and $h'$.

\begin{figure}[t]
\begin{center}

\tikzset{every picture/.style={line width=0.75pt}} 

\begin{tikzpicture}[x=0.75pt,y=0.75pt,yscale=-1.5,xscale=1.5]

\draw  [draw opacity=0] (101,189.9) -- (221.4,189.9) -- (221.4,250.25) -- (101,250.25) -- cycle ; \draw  [color={rgb, 255:red, 155; green, 155; blue, 155 }  ,draw opacity=0.62 ] (101,189.9) -- (101,250.25)(121,189.9) -- (121,250.25)(141,189.9) -- (141,250.25)(161,189.9) -- (161,250.25)(181,189.9) -- (181,250.25)(201,189.9) -- (201,250.25)(221,189.9) -- (221,250.25) ; \draw  [color={rgb, 255:red, 155; green, 155; blue, 155 }  ,draw opacity=0.62 ] (101,189.9) -- (221.4,189.9)(101,209.9) -- (221.4,209.9)(101,229.9) -- (221.4,229.9)(101,249.9) -- (221.4,249.9) ; \draw  [color={rgb, 255:red, 155; green, 155; blue, 155 }  ,draw opacity=0.62 ]  ;
\draw  [color={rgb, 255:red, 155; green, 155; blue, 155 }  ,draw opacity=0.62 ][fill={rgb, 255:red, 155; green, 155; blue, 155 }  ,fill opacity=1 ] (118.17,229.9) .. controls (118.17,228.34) and (119.44,227.07) .. (121,227.07) .. controls (122.56,227.07) and (123.83,228.34) .. (123.83,229.9) .. controls (123.83,231.46) and (122.56,232.72) .. (121,232.72) .. controls (119.44,232.72) and (118.17,231.46) .. (118.17,229.9) -- cycle ;
\draw  [color={rgb, 255:red, 155; green, 155; blue, 155 }  ,draw opacity=0.62 ][fill={rgb, 255:red, 155; green, 155; blue, 155 }  ,fill opacity=1 ] (138.17,249.9) .. controls (138.17,248.34) and (139.44,247.07) .. (141,247.07) .. controls (142.56,247.07) and (143.83,248.34) .. (143.83,249.9) .. controls (143.83,251.46) and (142.56,252.72) .. (141,252.72) .. controls (139.44,252.72) and (138.17,251.46) .. (138.17,249.9) -- cycle ;
\draw  [color={rgb, 255:red, 155; green, 155; blue, 155 }  ,draw opacity=0.62 ][fill={rgb, 255:red, 155; green, 155; blue, 155 }  ,fill opacity=1 ] (138.17,229.9) .. controls (138.17,228.34) and (139.44,227.07) .. (141,227.07) .. controls (142.56,227.07) and (143.83,228.34) .. (143.83,229.9) .. controls (143.83,231.46) and (142.56,232.72) .. (141,232.72) .. controls (139.44,232.72) and (138.17,231.46) .. (138.17,229.9) -- cycle ;
\draw  [color={rgb, 255:red, 155; green, 155; blue, 155 }  ,draw opacity=0.62 ][fill={rgb, 255:red, 155; green, 155; blue, 155 }  ,fill opacity=1 ] (158.17,229.9) .. controls (158.17,228.34) and (159.44,227.07) .. (161,227.07) .. controls (162.56,227.07) and (163.83,228.34) .. (163.83,229.9) .. controls (163.83,231.46) and (162.56,232.72) .. (161,232.72) .. controls (159.44,232.72) and (158.17,231.46) .. (158.17,229.9) -- cycle ;
\draw  [color={rgb, 255:red, 155; green, 155; blue, 155 }  ,draw opacity=0.62 ][fill={rgb, 255:red, 155; green, 155; blue, 155 }  ,fill opacity=1 ] (178.17,249.9) .. controls (178.17,248.34) and (179.44,247.07) .. (181,247.07) .. controls (182.56,247.07) and (183.83,248.34) .. (183.83,249.9) .. controls (183.83,251.46) and (182.56,252.72) .. (181,252.72) .. controls (179.44,252.72) and (178.17,251.46) .. (178.17,249.9) -- cycle ;
\draw  [color={rgb, 255:red, 155; green, 155; blue, 155 }  ,draw opacity=0.62 ][fill={rgb, 255:red, 155; green, 155; blue, 155 }  ,fill opacity=1 ] (158.17,249.9) .. controls (158.17,248.34) and (159.44,247.07) .. (161,247.07) .. controls (162.56,247.07) and (163.83,248.34) .. (163.83,249.9) .. controls (163.83,251.46) and (162.56,252.72) .. (161,252.72) .. controls (159.44,252.72) and (158.17,251.46) .. (158.17,249.9) -- cycle ;
\draw  [color={rgb, 255:red, 155; green, 155; blue, 155 }  ,draw opacity=0.62 ][fill={rgb, 255:red, 155; green, 155; blue, 155 }  ,fill opacity=1 ] (178.17,229.9) .. controls (178.17,228.34) and (179.44,227.07) .. (181,227.07) .. controls (182.56,227.07) and (183.83,228.34) .. (183.83,229.9) .. controls (183.83,231.46) and (182.56,232.72) .. (181,232.72) .. controls (179.44,232.72) and (178.17,231.46) .. (178.17,229.9) -- cycle ;
\draw  [color={rgb, 255:red, 155; green, 155; blue, 155 }  ,draw opacity=0.62 ][fill={rgb, 255:red, 155; green, 155; blue, 155 }  ,fill opacity=1 ] (198.17,229.9) .. controls (198.17,228.34) and (199.44,227.07) .. (201,227.07) .. controls (202.56,227.07) and (203.83,228.34) .. (203.83,229.9) .. controls (203.83,231.46) and (202.56,232.72) .. (201,232.72) .. controls (199.44,232.72) and (198.17,231.46) .. (198.17,229.9) -- cycle ;
\draw  [color={rgb, 255:red, 155; green, 155; blue, 155 }  ,draw opacity=0.62 ][fill={rgb, 255:red, 155; green, 155; blue, 155 }  ,fill opacity=1 ] (218.17,249.9) .. controls (218.17,248.34) and (219.44,247.07) .. (221,247.07) .. controls (222.56,247.07) and (223.83,248.34) .. (223.83,249.9) .. controls (223.83,251.46) and (222.56,252.72) .. (221,252.72) .. controls (219.44,252.72) and (218.17,251.46) .. (218.17,249.9) -- cycle ;
\draw  [color={rgb, 255:red, 155; green, 155; blue, 155 }  ,draw opacity=0.62 ][fill={rgb, 255:red, 155; green, 155; blue, 155 }  ,fill opacity=1 ] (198.17,249.9) .. controls (198.17,248.34) and (199.44,247.07) .. (201,247.07) .. controls (202.56,247.07) and (203.83,248.34) .. (203.83,249.9) .. controls (203.83,251.46) and (202.56,252.72) .. (201,252.72) .. controls (199.44,252.72) and (198.17,251.46) .. (198.17,249.9) -- cycle ;
\draw  [color={rgb, 255:red, 155; green, 155; blue, 155 }  ,draw opacity=0.62 ][fill={rgb, 255:red, 155; green, 155; blue, 155 }  ,fill opacity=1 ] (218.17,229.9) .. controls (218.17,228.34) and (219.44,227.07) .. (221,227.07) .. controls (222.56,227.07) and (223.83,228.34) .. (223.83,229.9) .. controls (223.83,231.46) and (222.56,232.72) .. (221,232.72) .. controls (219.44,232.72) and (218.17,231.46) .. (218.17,229.9) -- cycle ;
\draw  [color={rgb, 255:red, 155; green, 155; blue, 155 }  ,draw opacity=0.62 ][fill={rgb, 255:red, 155; green, 155; blue, 155 }  ,fill opacity=1 ] (98.17,249.9) .. controls (98.17,248.34) and (99.44,247.07) .. (101,247.07) .. controls (102.56,247.07) and (103.83,248.34) .. (103.83,249.9) .. controls (103.83,251.46) and (102.56,252.72) .. (101,252.72) .. controls (99.44,252.72) and (98.17,251.46) .. (98.17,249.9) -- cycle ;
\draw  [color={rgb, 255:red, 155; green, 155; blue, 155 }  ,draw opacity=0.62 ][fill={rgb, 255:red, 155; green, 155; blue, 155 }  ,fill opacity=1 ] (98.17,229.9) .. controls (98.17,228.34) and (99.44,227.07) .. (101,227.07) .. controls (102.56,227.07) and (103.83,228.34) .. (103.83,229.9) .. controls (103.83,231.46) and (102.56,232.72) .. (101,232.72) .. controls (99.44,232.72) and (98.17,231.46) .. (98.17,229.9) -- cycle ;
\draw  [color={rgb, 255:red, 155; green, 155; blue, 155 }  ,draw opacity=0.62 ][fill={rgb, 255:red, 155; green, 155; blue, 155 }  ,fill opacity=1 ] (98.17,189.9) .. controls (98.17,188.34) and (99.44,187.07) .. (101,187.07) .. controls (102.56,187.07) and (103.83,188.34) .. (103.83,189.9) .. controls (103.83,191.46) and (102.56,192.72) .. (101,192.72) .. controls (99.44,192.72) and (98.17,191.46) .. (98.17,189.9) -- cycle ;
\draw  [color={rgb, 255:red, 155; green, 155; blue, 155 }  ,draw opacity=0.62 ][fill={rgb, 255:red, 155; green, 155; blue, 155 }  ,fill opacity=1 ] (158.17,189.9) .. controls (158.17,188.34) and (159.44,187.07) .. (161,187.07) .. controls (162.56,187.07) and (163.83,188.34) .. (163.83,189.9) .. controls (163.83,191.46) and (162.56,192.72) .. (161,192.72) .. controls (159.44,192.72) and (158.17,191.46) .. (158.17,189.9) -- cycle ;
\draw  [color={rgb, 255:red, 155; green, 155; blue, 155 }  ,draw opacity=0.62 ][fill={rgb, 255:red, 155; green, 155; blue, 155 }  ,fill opacity=1 ] (158.17,189.9) .. controls (158.17,188.34) and (159.44,187.07) .. (161,187.07) .. controls (162.56,187.07) and (163.83,188.34) .. (163.83,189.9) .. controls (163.83,191.46) and (162.56,192.72) .. (161,192.72) .. controls (159.44,192.72) and (158.17,191.46) .. (158.17,189.9) -- cycle ;
\draw  [color={rgb, 255:red, 155; green, 155; blue, 155 }  ,draw opacity=0.62 ][fill={rgb, 255:red, 155; green, 155; blue, 155 }  ,fill opacity=1 ] (178.17,209.9) .. controls (178.17,208.34) and (179.44,207.07) .. (181,207.07) .. controls (182.56,207.07) and (183.83,208.34) .. (183.83,209.9) .. controls (183.83,211.46) and (182.56,212.72) .. (181,212.72) .. controls (179.44,212.72) and (178.17,211.46) .. (178.17,209.9) -- cycle ;
\draw  [color={rgb, 255:red, 155; green, 155; blue, 155 }  ,draw opacity=0.62 ][fill={rgb, 255:red, 155; green, 155; blue, 155 }  ,fill opacity=1 ] (138.17,189.9) .. controls (138.17,188.34) and (139.44,187.07) .. (141,187.07) .. controls (142.56,187.07) and (143.83,188.34) .. (143.83,189.9) .. controls (143.83,191.46) and (142.56,192.72) .. (141,192.72) .. controls (139.44,192.72) and (138.17,191.46) .. (138.17,189.9) -- cycle ;
\draw  [color={rgb, 255:red, 155; green, 155; blue, 155 }  ,draw opacity=0.62 ][fill={rgb, 255:red, 155; green, 155; blue, 155 }  ,fill opacity=1 ] (158.17,209.9) .. controls (158.17,208.34) and (159.44,207.07) .. (161,207.07) .. controls (162.56,207.07) and (163.83,208.34) .. (163.83,209.9) .. controls (163.83,211.46) and (162.56,212.72) .. (161,212.72) .. controls (159.44,212.72) and (158.17,211.46) .. (158.17,209.9) -- cycle ;
\draw  [color={rgb, 255:red, 155; green, 155; blue, 155 }  ,draw opacity=0.62 ][fill={rgb, 255:red, 155; green, 155; blue, 155 }  ,fill opacity=1 ] (118.17,209.9) .. controls (118.17,208.34) and (119.44,207.07) .. (121,207.07) .. controls (122.56,207.07) and (123.83,208.34) .. (123.83,209.9) .. controls (123.83,211.46) and (122.56,212.72) .. (121,212.72) .. controls (119.44,212.72) and (118.17,211.46) .. (118.17,209.9) -- cycle ;
\draw  [color={rgb, 255:red, 155; green, 155; blue, 155 }  ,draw opacity=0.62 ][fill={rgb, 255:red, 155; green, 155; blue, 155 }  ,fill opacity=1 ] (98.17,209.9) .. controls (98.17,208.34) and (99.44,207.07) .. (101,207.07) .. controls (102.56,207.07) and (103.83,208.34) .. (103.83,209.9) .. controls (103.83,211.46) and (102.56,212.72) .. (101,212.72) .. controls (99.44,212.72) and (98.17,211.46) .. (98.17,209.9) -- cycle ;
\draw  [color={rgb, 255:red, 155; green, 155; blue, 155 }  ,draw opacity=0.62 ][fill={rgb, 255:red, 155; green, 155; blue, 155 }  ,fill opacity=1 ] (178.17,189.9) .. controls (178.17,188.34) and (179.44,187.07) .. (181,187.07) .. controls (182.56,187.07) and (183.83,188.34) .. (183.83,189.9) .. controls (183.83,191.46) and (182.56,192.72) .. (181,192.72) .. controls (179.44,192.72) and (178.17,191.46) .. (178.17,189.9) -- cycle ;
\draw  [color={rgb, 255:red, 155; green, 155; blue, 155 }  ,draw opacity=0.62 ][fill={rgb, 255:red, 155; green, 155; blue, 155 }  ,fill opacity=1 ] (198.17,209.9) .. controls (198.17,208.34) and (199.44,207.07) .. (201,207.07) .. controls (202.56,207.07) and (203.83,208.34) .. (203.83,209.9) .. controls (203.83,211.46) and (202.56,212.72) .. (201,212.72) .. controls (199.44,212.72) and (198.17,211.46) .. (198.17,209.9) -- cycle ;
\draw  [color={rgb, 255:red, 155; green, 155; blue, 155 }  ,draw opacity=0.62 ][fill={rgb, 255:red, 155; green, 155; blue, 155 }  ,fill opacity=1 ] (198.17,189.9) .. controls (198.17,188.34) and (199.44,187.07) .. (201,187.07) .. controls (202.56,187.07) and (203.83,188.34) .. (203.83,189.9) .. controls (203.83,191.46) and (202.56,192.72) .. (201,192.72) .. controls (199.44,192.72) and (198.17,191.46) .. (198.17,189.9) -- cycle ;
\draw  [color={rgb, 255:red, 155; green, 155; blue, 155 }  ,draw opacity=0.62 ][fill={rgb, 255:red, 155; green, 155; blue, 155 }  ,fill opacity=1 ] (218.17,209.9) .. controls (218.17,208.34) and (219.44,207.07) .. (221,207.07) .. controls (222.56,207.07) and (223.83,208.34) .. (223.83,209.9) .. controls (223.83,211.46) and (222.56,212.72) .. (221,212.72) .. controls (219.44,212.72) and (218.17,211.46) .. (218.17,209.9) -- cycle ;
\draw  [color={rgb, 255:red, 155; green, 155; blue, 155 }  ,draw opacity=0.62 ][fill={rgb, 255:red, 155; green, 155; blue, 155 }  ,fill opacity=1 ] (118.17,189.9) .. controls (118.17,188.34) and (119.44,187.07) .. (121,187.07) .. controls (122.56,187.07) and (123.83,188.34) .. (123.83,189.9) .. controls (123.83,191.46) and (122.56,192.72) .. (121,192.72) .. controls (119.44,192.72) and (118.17,191.46) .. (118.17,189.9) -- cycle ;
\draw  [color={rgb, 255:red, 155; green, 155; blue, 155 }  ,draw opacity=0.62 ][fill={rgb, 255:red, 155; green, 155; blue, 155 }  ,fill opacity=1 ] (138.17,209.9) .. controls (138.17,208.34) and (139.44,207.07) .. (141,207.07) .. controls (142.56,207.07) and (143.83,208.34) .. (143.83,209.9) .. controls (143.83,211.46) and (142.56,212.72) .. (141,212.72) .. controls (139.44,212.72) and (138.17,211.46) .. (138.17,209.9) -- cycle ;
\draw [color={rgb, 255:red, 74; green, 144; blue, 226 }  ,draw opacity=1 ][line width=2.25]    (121,249.9) -- (150.9,249.9) -- (161,249.9) ;
\draw [color={rgb, 255:red, 74; green, 144; blue, 226 }  ,draw opacity=1 ][line width=2.25]    (161,249.9) -- (161,209.9) ;
\draw [color={rgb, 255:red, 74; green, 144; blue, 226 }  ,draw opacity=1 ][line width=2.25]    (161,209.9) -- (201,209.9) ;
\draw [color={rgb, 255:red, 74; green, 144; blue, 226 }  ,draw opacity=1 ][fill={rgb, 255:red, 74; green, 144; blue, 226 }  ,fill opacity=1 ][line width=2.25]    (201,209.9) -- (201,189.9) ;
\draw [color={rgb, 255:red, 74; green, 144; blue, 226 }  ,draw opacity=1 ][line width=2.25]    (201,189.9) -- (221,189.9) ;
\draw  [color={rgb, 255:red, 155; green, 155; blue, 155 }  ,draw opacity=0.62 ][fill={rgb, 255:red, 155; green, 155; blue, 155 }  ,fill opacity=1 ] (118.17,249.9) .. controls (118.17,248.34) and (119.44,247.07) .. (121,247.07) .. controls (122.56,247.07) and (123.83,248.34) .. (123.83,249.9) .. controls (123.83,251.46) and (122.56,252.72) .. (121,252.72) .. controls (119.44,252.72) and (118.17,251.46) .. (118.17,249.9) -- cycle ;
\draw  [color={rgb, 255:red, 155; green, 155; blue, 155 }  ,draw opacity=0.62 ][fill={rgb, 255:red, 155; green, 155; blue, 155 }  ,fill opacity=1 ] (218.17,189.9) .. controls (218.17,188.34) and (219.44,187.07) .. (221,187.07) .. controls (222.56,187.07) and (223.83,188.34) .. (223.83,189.9) .. controls (223.83,191.46) and (222.56,192.72) .. (221,192.72) .. controls (219.44,192.72) and (218.17,191.46) .. (218.17,189.9) -- cycle ;
\draw  [draw opacity=0] (249,189.4) -- (369.4,189.4) -- (369.4,249.75) -- (249,249.75) -- cycle ; \draw  [color={rgb, 255:red, 155; green, 155; blue, 155 }  ,draw opacity=0.62 ] (249,189.4) -- (249,249.75)(269,189.4) -- (269,249.75)(289,189.4) -- (289,249.75)(309,189.4) -- (309,249.75)(329,189.4) -- (329,249.75)(349,189.4) -- (349,249.75)(369,189.4) -- (369,249.75) ; \draw  [color={rgb, 255:red, 155; green, 155; blue, 155 }  ,draw opacity=0.62 ] (249,189.4) -- (369.4,189.4)(249,209.4) -- (369.4,209.4)(249,229.4) -- (369.4,229.4)(249,249.4) -- (369.4,249.4) ; \draw  [color={rgb, 255:red, 155; green, 155; blue, 155 }  ,draw opacity=0.62 ]  ;
\draw [color={rgb, 255:red, 155; green, 155; blue, 155 }  ,draw opacity=1 ][line width=1]    (249,249.4) -- (369,249.4) ;
\draw [color={rgb, 255:red, 155; green, 155; blue, 155 }  ,draw opacity=1 ][line width=1]    (249,229.4) -- (369,229.4) ;
\draw [color={rgb, 255:red, 155; green, 155; blue, 155 }  ,draw opacity=1 ][line width=1]    (249,209.4) -- (369,209.4) ;
\draw [color={rgb, 255:red, 155; green, 155; blue, 155 }  ,draw opacity=1 ][line width=1]    (249,189.4) -- (369,189.4) ;
\draw [color={rgb, 255:red, 245; green, 166; blue, 35 }  ,draw opacity=1 ][line width=2.25]    (249,248.9) -- (309,248.9) ;
\draw [dashed, color={rgb, 255:red, 245; green, 166; blue, 35 }  ,draw opacity=1 ][line width=2.25]    (309,248.9) -- (289,228.9) ;
\draw [color={rgb, 255:red, 245; green, 166; blue, 35 }  ,draw opacity=1 ][line width=2.25]    (289,228.9) -- (309,228.9) ;
\draw [color={rgb, 255:red, 245; green, 166; blue, 35 }  ,draw opacity=1 ][line width=2.25]    (289,208.9) -- (349,208.9) ;
\draw [dashed, color={rgb, 255:red, 245; green, 166; blue, 35 }  ,draw opacity=1 ][line width=2.25]    (289,208.9) -- (309,228.9) ;
\draw [dashed, color={rgb, 255:red, 245; green, 166; blue, 35 }  ,draw opacity=1 ][line width=2.25]    (349,208.9) -- (336.13,196.02) -- (329,188.9) ;
\draw [color={rgb, 255:red, 245; green, 166; blue, 35 }  ,draw opacity=1 ][line width=2.25]    (329,188.9) -- (369,188.9) ;
\draw  [color={rgb, 255:red, 155; green, 155; blue, 155 }  ,draw opacity=0.62 ][fill={rgb, 255:red, 155; green, 155; blue, 155 }  ,fill opacity=1 ] (266.17,249.4) .. controls (266.17,247.84) and (267.44,246.57) .. (269,246.57) .. controls (270.56,246.57) and (271.83,247.84) .. (271.83,249.4) .. controls (271.83,250.96) and (270.56,252.22) .. (269,252.22) .. controls (267.44,252.22) and (266.17,250.96) .. (266.17,249.4) -- cycle ;
\draw  [color={rgb, 255:red, 155; green, 155; blue, 155 }  ,draw opacity=0.62 ][fill={rgb, 255:red, 155; green, 155; blue, 155 }  ,fill opacity=1 ] (366.17,189.4) .. controls (366.17,187.84) and (367.44,186.57) .. (369,186.57) .. controls (370.56,186.57) and (371.83,187.84) .. (371.83,189.4) .. controls (371.83,190.96) and (370.56,192.22) .. (369,192.22) .. controls (367.44,192.22) and (366.17,190.96) .. (366.17,189.4) -- cycle ;

\draw (116.4,256) node [anchor=north west][inner sep=0.75pt]    {$\mathbf u$};
\draw (216.9,178.4) node [anchor=north west][inner sep=0.75pt]    {$\bfv$};
\draw (262.9,256) node [anchor=north west][inner sep=0.75pt]    {$\mathbf u$};
\draw (363.9,178.4) node [anchor=north west][inner sep=0.75pt]    {$\bfv$};

\end{tikzpicture}
\captionsetup{width=.8\linewidth}
\caption{The passage value of a path is preserved when switching from vertex weights $\omega$ (left) to horizontal edge weights $h$ (right) via \eqref{defh}.}\label{figh}
\end{center}
\end{figure}
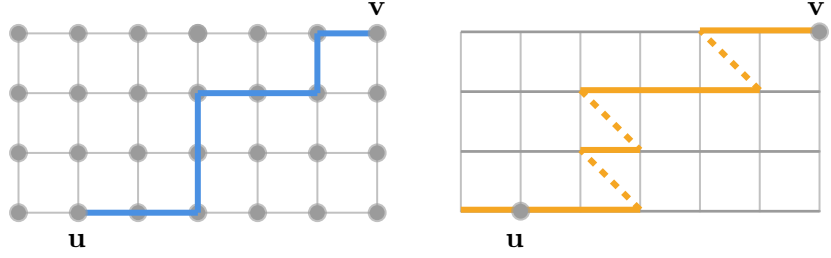

Next we define LPP with a south boundary located at level $m$, with boundary weights $h^m$. 
The last-passage value will be denoted by \( G^{\textup{S}}_{m, \bbullet} (h)\), where the superscript ``S'' stands for south boundary. For $x\in \mathbb{Z}$ and $n> m$, the last-passage value from level $m$ to vertex $(x,n)$ is given by
\begin{equation}\label{GS}
G^{\textup{S}}_{m, (x,n)}(h) = 
\begin{cases}
h^m(x) & \text{if } n=m \\
\sup_{z\leq x} \{h^m(z) + G_{(z,m+1), (x,n)}(h)\} & \text{if } n > m.
\end{cases}
\end{equation}
The interpretation is that the allowed paths start at $(0,m)$, travel horizontally to $(z,m)$ for some $z\le x$, and then travel to $(x,n)$ along an up-right path.
Because we allow any $z\le x$ (including negative $z$), the last-passage value $G^{\textup{S}}_{m,(x,n)}$ could be infinite.
To address this, we introduce the following \textit{asymptotic slope condition}: for each $m\in \mathbb{Z}$, we assume there exists a positive real number $\textup{\texttt{slope}}_m(h)$ such that 
\begin{align} \label{slope_condition}
\textup{\texttt{slope}}_m(h) = \lim_{x\to \infty} \frac{h^m(x)}{x}  = \lim_{x\to -\infty} \frac{h^m(x)}{x}.    
\end{align}
The following proposition gives a sufficient condition for $G^{\textup{S}}_{m,\bbullet}$ to be finite.
It is essentially the condition for a queue to be stable, see Appendix \ref{stab} for its proof. 


\begin{proposition}[Sufficient condition for boundary LPP]\label{stab2}
Assume $h$ satisfies the asymptotic slope condition.
For levels $n>m$, if
$$\textup{\texttt{slope}}_m (h)> \max\{\textup{\texttt{slope}}_{m+1}(h), \dots, \textup{\texttt{slope}}_n(h)\},$$
then the last passage value $G^\textup{S}_{m, (x, n)}(h)$ is finite for every $x\in\Z$.
\end{proposition}

\begin{remark}[Random weights]
Ultimately we will take the vertex weights $(\omega_{(x,m)})_{x\in \mathbb{Z}}$ to be i.i.d.~positive random variables with finite mean.
In this case, the asymptotic slope condition \eqref{slope_condition} is satisfied with $\textup{\texttt{slope}}_m(h) = \mathbb{E}[\omega_{(0,m)}]$.
Moreover, Proposition~\ref{stab2} says that the LPP with south boundary makes sense when the weights on the boundary level have a strictly larger mean than those on higher levels.
Intuitively, this is because a long leftward journey from $(0,m)$ to $(z,m)$ will cost more asymptotically (as $z\to-\infty$) than the gain from optimizing over up-right paths from $(z,m)$ to $(x,n)$.
\end{remark}

Once we know $G^{\textup{S}}_{m,(x,n)}$ is finite, we can define horizontal increments \(I^{\textup{S},m}_\mathbf{z}(h)\) for \(\mathbf{z} \in \mathbb{Z} \times \mathbb{Z}_{\geq m}\), vertical increments  $J^{\textup{S},m}_\mathbf{z}(h)$ for \(\mathbf{z} \in \mathbb{Z} \times \mathbb{Z}_{\geq m+1}\), and the dual weights \(\wc \omega_{\mathbf{z}+\mathbf{e}_1} (h)\) for \(\mathbf{z} \in \mathbb{Z} \times \mathbb{Z}_{\geq m}\):
\begin{equation}\label{def_IJ} 
\begin{aligned}
I^{\textup{S},m}_\mathbf{z}(h) &= G^{\textup{S}}_{m, \mathbf{z}}(h) - G^{\textup{S}}_{m, \mathbf{z}-\mathbf{e}_1}(h),\\
J^{\textup{S},m}_\mathbf{z}(h) &= G^{\textup{S}}_{m, \mathbf{z}}(h) - G^{\textup{S}}_{m, \mathbf{z}-\mathbf{e}_2}(h),\\
\wc \omega_{\mathbf{z}+\mathbf{e}_1} (h)&= I^{\textup{S},m}_{\mathbf{z}+\mathbf e_1}(h)\wedge J^{\textup{S},m}_{\mathbf{z}+\mathbf e_2}(h).
\end{aligned}
\end{equation} 
Note that we index the dual weight $\wc \omega$ by $\mathbf{z} + \mathbf e_1$, rather than the more conventional $\mathbf{z}$. This choice ensures that, when we transform these weights to the $h$-environment, $\wc\omega_{\mathbf{z} + \mathbf e_1}$ replaces the original $I^{\textup{S}, m}_{\mathbf{z} + \mathbf e_1}$ weight associated with the edge $[\![\mathbf{z}, \mathbf{z} + \mathbf e_1]\!]$.

With these definitions, we can now introduce the \emph{upward queueing map}, denoted by \( \sigma^m \), which can be seen as a centered version of the Pitman transform. The queueing map $\sigma^m$ acts on the environment \( h \), producing a new environment \( \sigma^m h \), where the two horizontal layers of weights \( (h^{m}, h^{m+1}) \) are replaced with transformed weights \( ((\sigma^m h)^{m}, (\sigma^m h)^{m+1}) \). 
More precisely, let $(\sigma^{m} h)^{m+1}(0) = (\sigma^{m} h)^{m} (0) =0$. 
If $\textup{\texttt{slope}}_m (h)> \textup{\texttt{slope}}_{m+1} (h)$, then complete the definition of $(\sigma^m h)^{m+1}$ and $(\sigma^m h)^{m}$ by prescribing the increments for all $i\in\Z$:
\begin{align}\label{qmap}
\begin{aligned}
(\sigma^{m} h)^{m+1}(i) - (\sigma^{m} h)^{m+1}(i-1) &= I^{\textup{S},m}_{(i, m+1)}(h),\\
(\sigma^{m} h)^{m}(i) - (\sigma^{m} h)^{m}(i-1) &=  \wc \omega_{(i, m)}.
\end{aligned}
\end{align}
If instead $\textup{\texttt{slope}}_m (h) \leq \textup{\texttt{slope}}_{m+1}(h)$, then let $\sigma^m$ act as the identity map: $\sigma^m h = h$.
In either case, $\sigma^m$ acts as the identity on all other levels $k\in \mathbb{Z}\setminus\{m, m+1\}$: $(\sigma^m h)^k = h^k$.

\subsection{The downward queueing map}

Recall the definition of the upward queueing map in the previous section, and we see $\sigma^m$ acts non-trivially on $h$ if $\textup{\texttt{slope}}_m(h) > \textup{\texttt{slope}}_{m+1}(h)$. In addition, as stated in Lemma~\ref{stab1}, in this case $\sigma^m$ swaps the asymptotic slopes
\begin{equation}\label{swappp}
\textup{\texttt{slope}}_m(\sigma^m h) = \textup{\texttt{slope}}_{m+1}(h)\qquad  \textup{and}  \qquad\textup{\texttt{slope}}_{m+1}(\sigma^m h) = \textup{\texttt{slope}}_m(h).
\end{equation}
This motivates the definition of an ``inverse'' of $\sigma^m$, which we call the \emph{downward queueing map} and denote by $\sigma_m$, replacing the superscript $m$ with the subscript $m$.

To define the downward queueing map, we introduce the LPP model with a north boundary. Fix $n \in \mathbb{Z}$
and for $k \leq n$, the last-passage value with the north boundary is given by  
\begin{equation}\label{GN}
G^{\textup{N}}_{(x,k), n} (h)= \begin{cases} -h^n(x) &\textup{if } k = n\\
\sup_{z\geq x} \{-h^{n}(z) + G_{(x+1, k), (z+1,n-1)(h)}\} &\textup{if } k < n.
\end{cases}
\end{equation}
The interpretation is that the allowed paths start at $(0,n)$, travel horizontally to $(z,n)$ for some $z\ge x$, and then travel to $(x,k)$ along a down-left path.
Note that, comparing this definition to \eqref{GS}, the endpoints of the point-to-point LPP appearing inside the supremum are shifted one unit to the right. This shift is a consequence of the indexing convention for the dual weights $\wc\omega$ used in \eqref{def_IJ}.
Otherwise this definition matches that of south-boundary LPP after rotation by $180^\circ$. 

For points $\mathbf{z}$ along or below level $n$, we define the \( I \) and \( J \) increments, as well as the dual weights $\widehat\omega$ (recall the previous notation is $\wc \omega$), in the same manner as before:
\begin{align*}
I^{\textup{N}, n}_{\mathbf{z}} (h) &= {G}{}^{\textup{N}}_{\mathbf{z} - \mathbf{e}_1, n}(h) - {G}{}^{\textup{N}}_{\mathbf{z},  n}(h),\\
J^{\textup{N},{n}}_{\mathbf{z}}(h) &= {G}{}^{\textup{N}}_{ \mathbf{z} - \mathbf{e}_2, {n}}(h) - {G}{}^{\textup{N}}_{\mathbf{z}, {n}}(h),\\
\widehat \omega_{\mathbf z}(h) &= I^{\textup{N}, {n}}_{\mathbf{z}}(h) \wedge J^{\textup{N},{n}}_{\mathbf{z}}(h).
\end{align*}
The downward queueing map $\sigma_m$ is then defined as follows. To start, let  $(\sigma_m h)^m (0)= (\sigma_m h)^{m+1}(0) = 0$. 
If $\textup{\texttt{slope}}_m(h) < \textup{\texttt{slope}}_{m+1}(h)$, then we complete the definition of $(\sigma_m h)^{m}$ and $(\sigma_m h)^{m+1}$ by prescribing the following increments:
\begin{align*}
(\sigma_m h)^m (i) - (\sigma_m h)^m (i-1) &= I^{\textup{N},{m+1}}_{(i, m)}(h),\\
(\sigma_m h)^{m+1} (i) - (\sigma_m h)^{m+1} (i-1) &=  \widehat \omega_{(i, m+1)}.
\end{align*}
On the other hand, if $\textup{\texttt{slope}}_m(h) \geq \textup{\texttt{slope}}_{m+1}(h)$, we then let $\sigma_m$ act as the identity map: $\sigma_m h = h$. 
In either case, $\sigma_m$ acts as the identity on all other levels $k\in \mathbb{Z}\setminus\{m, m+1\}$: $(\sigma_m h)^k = h^k$.

Because of the equivalence between $G^\textup{S}$ and $G^{\textup{N}}$, $\sigma_m$ is essentially the same map as $\sigma^m$ but with a $180^\circ$ rotation. In addition, the following proposition justifies the notion of ``inverse'' for our queueing maps.

\begin{proposition}[Cancellation of upward and downward queueing maps]\label{q_inverse}
We have
\eeq{ \label{inverse_eq}
h = \begin{cases}
\sigma_m \sigma^m h &\text{if $\textup{\texttt{slope}}_m(h) >\textup{\texttt{slope}}_{m+1}(h)$}\\
\sigma^m \sigma_m h &\text{if $\textup{\texttt{slope}}_m(h) <\textup{\texttt{slope}}_{m+1}(h)$}.
\end{cases}
}
\end{proposition}

\begin{remark}[Invertible if actually updating]
One way to interpret Proposition~\ref{q_inverse} is that whenever the queueing map $\sigma^m$ or $\sigma_m$ acts non-trivially on an environment $h$ (that is, whenever it is not the identity map), its action can be inverted with $\sigma_m$ or $\sigma^m$ respectively. 
\end{remark}

\begin{remark}[Trivial extension]
Although we will not use this, we note that when  $\textup{\texttt{slope}}_m(h)=  \textup{\texttt{slope}}_{m+1}(h)$, the identity $h = \sigma_m \sigma^m h = \sigma^m \sigma_m h$ holds trivially since both $\sigma^m$ and $\sigma_m$ act as the identity map in this case.
\end{remark}

To justify Proposition~\ref{q_inverse}, we exploit the following connection between LPP with a south boundary and LPP with a north boundary.  
To start, let us fix $m< n$ and $\mathbf z \in \mathbb{Z}\times [\![m,n ]\!]$. For an environment $h$ with 
$$\textup{\texttt{slope}}_m(h) > \max\{ \textup{\texttt{slope}}_{m+1}(h), \dots, \textup{\texttt{slope}}_{n}(h)\},$$ 
we can define $G^{\textup{S}}_{m, \mathbf z} (h)$ and construct an environment $\sigma^{n-1} \cdots\sigma^m h$. In particular, by the swapping property \eqref{swappp}, we have 
$$\textup{\texttt{slope}}_n(\sigma^{n-1} \cdots\sigma^m h) > \max\{ \textup{\texttt{slope}}_{m}(\sigma^{n-1} \cdots\sigma^m h), \dots, \textup{\texttt{slope}}_{n-1}(\sigma^{n-1} \cdots\sigma^m h)\}.$$
This means we can run an LPP with a north boundary along level $n$ with respect to this new environment $\sigma^{n-1} \cdots\sigma^m h$. Often, we will refer to the two LPPs: $G^{\textup{S}}_{m, \mathbf{z}} (h)$ and $G^{\textup{N}}_{\mathbf{z},n} (\sigma^{n-1} \cdots\sigma^m h)$
as the dual of each other. 

The following proposition describes the connection between the original model and its dual. In particular, it states that their corresponding increments are equal.

\begin{proposition}[{Increment duality, \cite[Lemma 4.7]{seppalainen20a}}]\label{dual_eq1}
Fix $m< n$. For each $\mathbf{z} \in \mathbb{Z}\times [\![m, n]\!]$, 
$$I^{\textup{S},m}_\mathbf{z}(h)  = I^{\textup{N},n}_{\mathbf{z}}(\sigma^{n-1} \cdots\sigma^m h).$$
And for each $\mathbf{z} \in \mathbb{Z}\times [\![m+1, n]\!]$,
$$J^{\textup{S},m}_\mathbf{z}(h)= J^{\textup{N},n}_{\mathbf{z}}(\sigma^{n-1} \cdots\sigma^m h).$$
\end{proposition}

By rotating the picture by $180^\circ$, we obtain the following equivalent result.

\begin{proposition}[Increment duality rewritten] \label{dual_eq2}
Fix $m< n$. For each $\mathbf{z} \in \mathbb{Z}\times [\![m, n]\!]$, 
$$I^{\textup{N},n}_{\mathbf{z} }(h)  = I^{\textup{S},m}_\mathbf{z}(\sigma_{m} \cdots\sigma_{n-1} h).$$
And for each $\mathbf{z} \in \mathbb{Z}\times [\![m+1, n]\!]$,
$$J^{\textup{N},n}_{\mathbf{z}}(h)= J^{\textup{S},m}_\mathbf{z}(\sigma_{m} \cdots\sigma_{n-1} h).$$
\end{proposition}
Proposition \ref{q_inverse} follows immediately from the two propositions above applied to the $I$ increments in two neighboring horizontal levels.

\subsection{Nested property of the queueing maps}

In this section, we will recall an important concept known as the nested LPP. By the definition of our queueing maps, we have
\begin{align*}
G^\textup{S}_{m,(a,m+1)}(h) - G^\textup{S}_{m,(b,m+1)}(h) &= (\sigma^m h)^{m+1}(a) - (\sigma^m h)^{m+1}(b),\\
G^\textup{N}_{(a,m-1), m}(h) - G^\textup{N}_{(b,m-1), m}(h) &= (\sigma_{m-1} h)^{m-1}(a) - (\sigma_{m-1} h)^{m-1}(b).
\end{align*}
These identities can be generalized from the levels ($m-1, m, m+1$) to ($k < m < n$) as follows.

\begin{proposition}[{Representation of increments, \cite[Proposition 3.1]{ind_buse}}]\label{LPP_queue}
For $k<m<n$,
\begin{align}
G^{\textup{S}}_{{m}, (a, n)}(h) - G^{\textup{S}}_{{m}, (b, n)}(h) &= (\sigma^{n-1}\cdots\sigma^m h)^n(a) - (\sigma^{n-1}\cdots\sigma^m h)^n(b), \label{rep_south} \\
G^{\textup{N}}_{(a,k), {m}}(h) - G^{\textup{N}}_{(b,k), {m}}(h) &= (\sigma_{k}\cdots\sigma_{m-1} h)^k(a)-  (\sigma_{k}\cdots\sigma_{m-1} h)^k(b). \label{rep_north}
\end{align}
\end{proposition}

We now use this result to prove the following.

\begin{proposition}[Nested LPP property]\label{nest}
Fix $k<m<n$.
For any vertices $\mathbf{x}, \mathbf{y}$ above or along level $n$, we have 
\begin{align} \label{nest_south}
G^{\textup{S}}_{m, \mathbf{x}}(h) - G^{\textup{S}}_{m, \mathbf{y}}(h) = G^{\textup{S}}_{n, \mathbf{x}}(\sigma^{n-1}\cdots\sigma^m h) - G^{\textup{S}}_{n, \mathbf{y}}(\sigma^{n-1}\cdots\sigma^m h).
\end{align}
Similarly, for any vertices $\mathbf{a}, \mathbf{b}$ below or along level $k$, we have 
$$G^{\textup{N}}_{\mathbf{a}, m}(h) - G^{\textup{N}}_{\mathbf{b}, m}(h) = G^{\textup{N}}_{\mathbf{a}, k}(\sigma_{k}\cdots\sigma_{m-1} h) - G^{\textup{N}}_{\mathbf{b}, k}(\sigma_{k}\cdots\sigma_{m-1} h).$$
\end{proposition}
\begin{proof}
We will prove the result with the south boundary, and the north-boundary LPP follows similarly. 
To start, let us rewrite the difference as follows:
$$G^{\textup{S}}_{m, \mathbf{x}}(h) - G^{\textup{S}}_{m, \mathbf{y}}(h)= \big[G^{\textup{S}}_{m, \mathbf{x}}(h) -G^{\textup{S}}_{m, (0,n)}(h)\big]- \big[G^{\textup{S}}_{m, \mathbf{y}}(h) - G^{\textup{S}}_{m, (0,n)}(h)\big].$$
Then, to complete the proof, it suffices to show that 
\begin{equation}\label{nest_eq}
\big[G^{\textup{S}}_{m, \mathbf{x}}(h) -G^{\textup{S}}_{m, (0,n)}(h)\big] =  G^{\textup{S}}_{n, \mathbf{x}}(\sigma^{n-1}\cdots\sigma^m h),
\end{equation}
since \eqref{nest_eq} will also hold when $\mathbf{x}$ is replaced with $\mathbf{y}$. 

Note that when $\mathbf{x}$ belongs to level $n$, the equality \eqref{nest_eq} follows from Proposition \ref{LPP_queue} and the fact that $(\sigma^{n-1}\cdots\sigma^m h)^n(0) = 0$.
Next, we consider the case where $\mathbf{x}$ lies strictly above level $n$.
To simplify the notation, let $x_1 = \mathbf{x} \cdot \mathbf{e}_1$.
We have
\begin{align*}
 G^{\textup{S}}_{m, \mathbf{x}}(h) -G^{\textup{S}}_{m, (0,n)}(h)
& \stackrefp{rep_south}{=} \sup_{a\leq x_1}\Big\{G^{\textup{S}}_{m, (a, n)}(h) -G^{\textup{S}}_{m, (0,n)}(h)  + G_{(a, n+1), \mathbf{x}}(h)\Big\} \\
& \stackref{rep_south}{=} \sup_{a\leq x_1}\Big\{ (\sigma^{n-1}\cdots\sigma^m h)^n(a) + G_{(a, n+1), \mathbf{x}}(\sigma^{n-1}\cdots\sigma^m h)\Big\} \\
& \stackrefp{rep_south}{=} G^{\textup{S}}_{n, \mathbf{x}}(\sigma^{n-1}\cdots\sigma^m h).
\end{align*}
We have thus established \eqref{nest_eq}, which completes the proof.
\end{proof}

\subsection{Preservation properties of the queueing maps}

The following result states that the queueing maps preserve the sum of weights across consecutive levels. 

\begin{proposition}[Conservation of interior mass]\label{sumeq}
For any $k\leq m <  n$ and $x, y \in \mathbb{Z}$, we have
\begin{align} \label{con_int_mass}
\sum_{i=k}^n\big[h^i(x) - h^i(y)\big] 
= \sum_{i=k}^n \big[(\sigma^m h)^i(x) -  (\sigma^m h)^i(y)\big] 
= \sum_{i=k}^n \big[(\sigma_{m} h)^i(x) -  (\sigma_{m} h)^i(y)\big].
\end{align}
\end{proposition}

\begin{proof}
We will prove the identity for the upward queueing map $\sigma^m$; the argument for the downward queueing map $\sigma_m$ is analogous.
Without loss of generality, assume $y<x$ (otherwise multiply all terms in \eqref{con_int_mass} by $-1$).

We begin by recalling the well-known ``corner-flipping'' relation
\begin{align}\label{flip}
I^{\textup{S},m}_{\mathbf x} = \omega_{\mathbf x} + (I^{\textup{S},m}_{\mathbf x - \mathbf e_2} - J^{\textup{S},m}_{\mathbf x - \mathbf e_1})^+
\qquad \text{and} \qquad
J^{\textup{S},m}_{\mathbf x} = \omega_{\mathbf x} + (I^{\textup{S},m}_{\mathbf x - \mathbf e_2} - J^{\textup{S},m}_{\mathbf x - \mathbf e_1})^-,
\end{align}
which follows from
\begin{align*}
I^{\textup{S},m}_{\mathbf x}
= G^{\textup{S}}_{m,\mathbf x} - G^{\textup{S}}_{x, \mathbf x - \mathbf e_1} 
&= \omega_{\mathbf x} + \max\Big\{G^{\textup{S}}_{m, \mathbf x - \mathbf e_1},G^{\textup{S}}_{m, \mathbf x - \mathbf e_2}\Big\} - G^{\textup{S}}_{m, \mathbf x - \mathbf e_1}\\
&= \omega_{\mathbf x} + \max\Big\{0, G^{\textup{S}}_{m,\mathbf x - \mathbf e_2} - G^{\textup{S}}_{m, \mathbf x - \mathbf e_1}\Big\} = \omega_{\mathbf x} + (I^{\textup{S},m}_{\mathbf x - \mathbf e_2} - J^{\textup{S},m}_{\mathbf x - \mathbf e_1})^+,
\end{align*}
and a similar computation yields the identity for $J^{\textup{S},m}_{\mathbf x}$.

Next, we note the general fact that
\begin{equation}\label{2equal}
\wc \omega_{\mathbf x - \mathbf e_2} + I^{\textup{S},m}_{\mathbf x} 
\stackref{def_IJ,flip}{=} \Big[I^{\textup{S},m}_{\mathbf x - \mathbf e_2} \wedge J^{\textup{S},m}_{\mathbf x - \mathbf e_1}\Big]  + \Big[\omega_{\mathbf x} + (I^{\textup{S},m}_{\mathbf x - \mathbf e_2} -J^{\textup{S},m}_{\mathbf x - \mathbf e_1})^+\Big] =\omega_{\mathbf x} + I^{\textup{S},m}_{\mathbf x - \mathbf e_2}.
\end{equation}
We now have
\begin{align*}
&h^m(x) - h^m(y) + h^{m+1}(x) - h^{m+1}(y) \\
&\stackref{defh}{=}\sum_{j=y+1}^x\big[\omega_{(j,m)} + \omega_{(j, m+1)}\big] 
=\sum_{j=y+1}^x \big[\omega_{(j,m+1)} + I^{\textup{S},m}_{(j, m)}\big] 
\stackref{2equal}{=} \sum_{j=y+1}^x \big[\wc{\omega}_{(j,m)} + I^{\textup{S},m}_{(j, m+1)}\big] \\
&\stackref{qmap}{=} (\sigma^m h)^m(x) -  (\sigma^m h)^m(y)+  (\sigma^m h)^{m+1}(x) -  (\sigma^{m} h)^{m+1}(y).
\end{align*}
Finally, although the summation in the proposition ranges over levels from $k$ to $n$, the operator $\sigma^m$ affects only $h^m$ and $h^{m+1}$. This observation completes the proof of the proposition.
\end{proof}

The next proposition shows that the queueing map preserves last-passage values. 
This was previously observed independently in \cite{fan_seppalainen20} for discrete LPP and \cite{KPZ_DL} for Brownian LPP. In \cite{fan_seppalainen20}, the isometry was used to help identify the joint distribution of the Busemann process in the exponential case. 
In \cite{KPZ_DL}, the isometry was used to help construct the directed landscape, which is the (conjecturally) universal scaling limit of ($1+1$)-dimensional KPZ models. 
Recalling the notation $G_{\mathbf{x}, \mathbf{y}}(h)$ introduced in \eqref{defGh}, we have the following proposition.

\begin{proposition}[{Isometric property, \cite[Lemma 4.3]{fan_seppalainen20}}]\label{isometry}
For any $k\leq m <  n$ and $a, b \in \mathbb{Z}$,
$$G_{(a,k), (b,n)}(h) = G_{(a,k), (b,n)}(\sigma^m h)=G_{(a,k), (b,n)}(\sigma_{m} h).$$
As a consequence, for $k< m <n$, 
\eeq{ \label{isometry_bdy}
G^\textup{S}_{k, (b,n)}(h) = G^\textup{S}_{k, (b,n)}(\sigma^m h) = G^\textup{S}_{k, (b,n)}(\sigma_m h).
}
\end{proposition}

\subsection{Braid relation of the queueing maps}

We start this section by giving an equivalent definition for the upward queueing map $\sigma^m$ from \eqref{qmap}. This will be used to prove the braid relation stated further below.

\begin{lemma}[Queuing map as a Pitman transform]\label{pitdef}
Suppose $\textup{\texttt{slope}}_m (h)> \textup{\texttt{slope}}_{m+1} (h)$ and recall $\sigma^m$  defined in \eqref{qmap}.
We then have the following identities:
\begin{align}
(\sigma^{m} h)^{m+1} (x) &= h^{m+1}(x) + \sup_{z\leq x} \{h^{m}(z) - h^{m+1}(z-1)\} - \sup_{z\leq 0} \{h^{m}(z) - h^{m+1}(z-1)\},  \label{altd1}\\
(\sigma^{m} h)^{m} (x) &= h^{m}(x) - \sup_{z\leq x} \{h^{m}(z) - h^{m+1}(z-1)\} +  \sup_{z\leq 0} \{h^{m}(z) - h^{m+1}(z-1)\}.\label{altd2}
\end{align}
\end{lemma}

\begin{proof}
To prove \eqref{altd1}, we write
\begin{align*}
(\sigma^{m} h)^{m+1}(x)
&\stackref{qmap}{=} G^\textup{S}_{m, (x, m+1)}(h) - G^\textup{S}_{m, (0, m+1)} (h) \\
&\stackref{GS}{=} \Big[h^{m+1}(x) + \sup_{z\leq x} \{h^{m}(z) - h^{m+1}(z-1)\}\Big]
-\Big[0 +  \sup_{z\leq 0} \{h^{m}(z) - h^{m+1}(z-1)\}\Big].
\end{align*}
To prove \eqref{altd2}, note that
$h^m(x) + h^{m+1}(x) = (\sigma^m h)^m(x) + (\sigma^m h)^{m+1}(x)$ by Proposition \ref{sumeq}, so 
\begin{equation}\label{alt_def1}(\sigma^m h)^m(x) = h^m(x) + h^{m+1}(x) - (\sigma^m h)^{m+1}(x).\end{equation}
Inserting the rewriting of $(\sigma^m h)^{m+1}(x)$ from \eqref{altd1} into \eqref{alt_def1}, we obtain \eqref{altd2}.
\end{proof}

The braid relation for the Pitman transform was first established in \cite{braid1}. Our proof is adapted from the argument of \cite[Lemma 4.4]{braid2}. For the reader's convenience, we include the details since our setting differs from theirs: the (original) weights $\omega_\mathbf{z}$ in our model are attached to vertices rather than edges, and our Pitman transform is defined in a centered form.

\begin{proposition}[Braid relation]\label{braid}
The following properties hold:
\begin{enumerate}[label=\textup{(\roman*)}]
\item \label{braiditem1} For $i, j \in \mathbb{Z}$ such that $|i-j|\geq 2$, we have 
$$\sigma^i \sigma^jh = \sigma^j \sigma^ih \qquad \textup{and} \qquad \sigma_i \sigma_jh = \sigma_j \sigma_ih.$$

\item \label{braiditem2} For each $i \in \mathbb{Z}$, we have
$$\sigma^i \sigma^{i+1}\sigma^i h = \sigma^{i+1} \sigma^{i}\sigma^{i+1} h \qquad \textup{and}\qquad \sigma_i \sigma_{i+1}\sigma_i h = \sigma_{i+1} \sigma_{i}\sigma_{i+1} h.$$
\end{enumerate}
\end{proposition}

\begin{proof}
We will work with $\sigma^i$, and the result for $\sigma_i$ follows by the equivalence of their definitions. 
Statement \ref{braiditem1} is clear because $\sigma^i$ and $\sigma^j$ act on disjoint levels. 
So we turn our attention to \ref{braiditem2}. 
Without loss of generality, assume $i=1$ and consider $h = (h^1, h^2, h^3)$. In addition, define 
$$f = \sigma^1 \sigma^{2}\sigma^1 h \qquad \textup{and} \qquad g = \sigma^2 \sigma^{1}\sigma^2 h.$$
To start, Proposition \ref{sumeq} gives
$$f^1(x) + f^2(x) + f^3(x) = g^1(x) + g^2(x) + g^3(x) \quad \text{for each } x\in \mathbb{Z}.$$
Therefore, it suffices to show 
\begin{align}
f^1 = (\sigma^1 \sigma^{2}\sigma^1 h)^1 &= (\sigma^{2} \sigma^{1}\sigma^{2} h)^1 = g^1,\label{lvlb}\\
f^3 = (\sigma^1 \sigma^{2}\sigma^1 h)^3 &= (\sigma^{2} \sigma^{1}\sigma^{2} h)^3 = g^3.\label{lvlt}
\end{align}

First consider the case
$\textup{\texttt{slope}}_{1}(h) > \textup{\texttt{slope}}_{2}(h) >\textup{\texttt{slope}}_{3}(h)$.
In this case, all queueing maps act non-trivially. 
Looking at \eqref{lvlt}, we will show 
$$f^3 = (\sigma^2 \sigma^1 h)^3 = g^3.$$ 
In the display above, the equality $f^3 = (\sigma^2 \sigma^1 h)^3$ is clear because $(\sigma^2 \sigma^1 h)^3$ remains unchanged under the map $\sigma^1$. To show $(\sigma^2 \sigma^1 h)^3 = g^3$, we will use the definition in Lemma \ref{pitdef}. We start by rewriting $(\sigma^2 \sigma^1 h)^3(x)$ as follows:
\begin{align*}
& (\sigma^2 \sigma^1 h)^3(x)\\
&= h^3(x) + \sup_{z_1\leq x}\Big\{ h^2(z_1) + \sup_{z_2 \leq z_1}\{h^1(z_2) - h^2( z_2-1)\} - \sup_{z_3 \leq 0}\{h^1(z_3) - h^2(z_3-1) \} - h^3(z_1-1)\Big\}\\
& \qquad \quad -  \sup_{z_1\leq 0}\Big\{ h^2(z_1) + \sup_{z_2 \leq z_1}\{h^1(z_2) - h^2( z_2-1)\} - \sup_{z_3 \leq 0}\{h^1(z_3) - h^2(z_3-1) \} - h^3(z_1-1)\Big\}\\
& = h^3(x) + \sup_{z_1\leq x}\Big\{ h^2(z_1) + \sup_{z_2 \leq z_1}\{h^1(z_2) - h^2( z_2-1)\}  - h^3(z_1-1)\Big\}\\
& \qquad \qquad -  \sup_{z_1\leq 0}\Big\{ h^2(z_1) + \sup_{z_2 \leq z_1}\{h^1(z_2) - h^2( z_2-1)\}  - h^3(z_1-1)\Big\}\\
& = \sup_{z_2\leq z_1\leq x}\Big\{ h^1(z_2)  + [h^2(z_1) - h^2(z_2-1)] + [h^3(x)- h^3(z_1-1)]\Big\}\\
& \qquad \qquad - \sup_{z_2\leq z_1\leq 0}\Big\{ h^1(z_2)  + [h^2(z_1) - h^2(z_2-1)] + [h^3(0)- h^3(z_1-1)]\Big\}\\
& = \sup_{z_2\leq x}\Big\{ h^1(z_2)  + G_{(z_2, 2), (x, 3)}(h)\Big\} - \sup_{z_2\leq 0}\Big\{ h^1(z_2)  + G_{(z_2, 2), (0, 3)}(h)\Big\}.
\end{align*}
Now, from the last line of calculation above, by Proposition \ref{isometry}, it holds that 
\begin{align*}
& \sup_{z_2\leq x}\Big\{ h^1(z_2)  + G_{(z_2, 2), (x, 3)}(h)\Big\} - \sup_{z_2\leq 0}\Big\{ h^1(z_2)  + G_{(z_2, 2), (0, 3)}(h)\Big\}\\
& = \sup_{z_2\leq x}\Big\{ \sigma^2 h^1(z_2)  + G_{(z_2, 2), (x, 3)}(\sigma^2h)\Big\} - \sup_{z_2\leq 0}\Big\{ \sigma^2h^1(z_2)  + G_{(z_2, 2), (0, 3)}(\sigma^2 h)\Big\}.
\end{align*}
In words, this states that we can replace the $h$ environment in the expression $(\sigma^2 \sigma^1 h)^3$ with $\sigma^2 h$. This gives us the desired equality 
$$(\sigma^2 \sigma^1 h)^3 = (\sigma^2 \sigma^1 \sigma^2 h)^3= g^3.$$

Next, looking at \eqref{lvlb}, we will show 
$$f^1 = (\sigma^1 \sigma^2 h)^1 = g^1.$$ 
In the display above, the equality $(\sigma^1 \sigma^2 h)^1 = g^1$ is clear because $(\sigma^1 \sigma^2 h)^1$ remains unchanged under the map $\sigma^2$, so we just need to prove $f^1 = (\sigma^1\sigma^2h)^1$. 
For the last equality in the calculation below, we use Propositions \ref{sumeq} and \ref{isometry}:
\begin{align*}
& \; (\sigma^1 \sigma^2 h)^1(x)\\
&= h^1(x) - \sup_{z_1\leq x}\{  h^1(z_1) - (h^2(z_1-1) - \sup_{z_2 \leq z_1-1}\{h^2(z_2) - h^3( z_2-1)\} + \sup_{z_3 \leq 0}\{h^2(z_3) - h^3(z_3-1) \})\}\\
& \qquad \quad +  \sup_{z_1\leq 0}\{ h^1(z_1) - (h^2(z_1-1) - \sup_{z_2 \leq z_1-1}\{h^2(z_2) - h^3( z_2-1)\} + \sup_{z_3 \leq 0}\{h^2(z_3) - h^3(z_3-1) \})\}\\
& = h^1(x) - \sup_{z_1\leq x}\Big\{ h^1(z_1) -h^2(z_1-1) + \sup_{z_2 \leq z_1-1}\{h^2(z_2) - h^3( z_2-1)\} \Big\}\\
& \qquad \qquad +  \sup_{z_1\leq 0}\Big\{ h^1(z_1) - h^2(z_1-1) + \sup_{z_2 \leq z_1-1}\{h^2(z_2) - h^3( z_2-1)\} \Big\}\\
& = - \sup_{z_2< z_1\leq x}\Big\{ - h^3(z_2-1)  - ([h^2(z_1-1) - h^2(z_2)] + [h^1(x) - h^1(z_1) ])\Big\}\\
& \qquad \qquad + \sup_{z_2< z_1\leq 0}\Big\{ - h^3(z_2-1)  - ([h^2(z_1-1) - h^2(z_2)] + [h^1(0) - h^1(z_1) ])\Big\}\\
& = - \sup_{z_2\leq z_1\leq x}\Big\{ - h^3(z_2-2)  - ([h^2(z_1-1) - h^2(z_2-1)] + [h^1(x) - h^1(z_1) ])\Big\}\\
& \qquad \qquad + \sup_{z_2\leq z_1\leq 0}\Big\{ - h^3(z_2-2)  - [h^2(z_1-1) - h^2(z_2-1)] + [h^1(0) - h^1(z_1)])\Big\}\\
& = - \sup_{z_2\leq z_1\leq x}\Big\{ - h^3(z_2-2) - (h^2(x) - h^2(z_2-1) + h^1(x) - h^1(z_2-1) - \\
& \qquad \qquad\qquad \qquad \qquad \qquad \qquad \qquad  ([h^2(x) - h^2(z_1-1)] - [h^1(z_2-1)- h^1(z_1)]))\Big\}\\
& \qquad \qquad + \sup_{z_2\leq z_1\leq 0}\Big\{ - h^3(z_2-2) - (h^2(0) - h^2(z_2-1) + h^1(0) - h^1(z_2-1) - \\
& \qquad \qquad\qquad \qquad \qquad \qquad \qquad \qquad  ([h^2(0) - h^2(z_1-1)] - [h^1(z_2-1)- h^1(z_1)]))\Big\}\\
& = \inf_{z_2\leq z_1\leq x}\Big\{  h^3(z_2-2) + h^2(x) - h^2(z_2-1) + h^1(x) - h^1(z_2-1) - \\
& \qquad \qquad\qquad \qquad \qquad \qquad \qquad \qquad  ([h^2(x) - h^2(z_1-1)] - [h^1(z_2-1)- h^1(z_1)])\Big\}\\
& \qquad \qquad - \inf_{z_2\leq z_1\leq 0}\Big\{  h^3(z_2-2) + h^2(0) - h^2(z_2-1) + h^1(0) - h^1(z_2-1) - \\
& \qquad \qquad\qquad \qquad \qquad \qquad \qquad \qquad  ([h^2(0) - h^2(z_1-1)] - [h^1(z_2-1)- h^1(z_1)])\Big\}\\
& = \inf_{z_2\leq x}\Big\{  h^3(z_2-2) + h^2(x) - h^2(z_2-1) + h^1(x) - h^1(z_2-1) - G_{(z_2, 1), (x,2)}(h)\Big\}\\
& \qquad \qquad - \inf_{z_2\leq 0}\Big\{  h^3(z_2-2) + h^2(0) - h^2(z_2-1) + h^1(0) - h^1(z_2-1) -  G_{(z_2, 1), (0,2)}(h)\Big\}\\
& = \inf_{z_2\leq x}\Big\{  \sigma^1 h^3(z_2-2) + \sigma^1 h^2(x) - \sigma^1 h^2(z_2-1) + \sigma^1 h^1(x) - \sigma^1 h^1(z_2-1) - G_{(z_2, 1), (x,2)}(\sigma^1 h)\Big\}\\
& \qquad \qquad - \inf_{z_2\leq 0}\Big\{  \sigma^1 h^3(z_2-2) + \sigma^1 h^2(0) - \sigma^1 h^2(z_2-1) + \sigma^1 h^1(0) -\sigma^1 h^1(z_2-1) \\
& \qquad \qquad \qquad\qquad \qquad \qquad \qquad \qquad \qquad \qquad \qquad\qquad   \qquad \qquad \qquad -  G_{(z_2, 1), (0,2)}(\sigma^1 h)\Big\}.
\end{align*}
Again, our last equality above states that one can replace the environment $h$ in $(\sigma^1 \sigma^2  h)^1$ with $\sigma^1 h$, which gives us the desired equality 
$$(\sigma^1 \sigma^2  h)^1 = (\sigma^1 \sigma^2 \sigma^1 h)^1 = f^1.$$ 
With this, we have finished proving \ref{braiditem2} for the case when $\textup{\texttt{slope}}_{1}(h) > \textup{\texttt{slope}}_{2}(h) >\textup{\texttt{slope}}_{3}(h)$.

If $\textup{\texttt{slope}}_{1}(h) \leq \textup{\texttt{slope}}_{2}(h)$, then 
$$ \sigma^1\sigma^2\sigma^1 h = \sigma^1\sigma^2 h  = \sigma^2\sigma^1\sigma^2 h.$$
Similarly, if $ \textup{\texttt{slope}}_{2}(h) \leq \textup{\texttt{slope}}_{3}(h)$, then
$$ \sigma^1\sigma^2\sigma^1 h = \sigma^2\sigma^1 h  = \sigma^2\sigma^1\sigma^2 h.$$
This completes the proof.
\end{proof}

\subsection{Monotonicity of the queueing maps}\label{subsec_monotonicities}
This section recalls a straightforward but important monotonicity result.

\begin{proposition}[{Monotonicity of increments with respect to environment, \cite[Lemma B.1]{balazs_busani_seppalainen20}}]\label{inc_bdry}
Consider two environments $h$ and $\wt h$ that are identical except at level $m$, where instead
$$h^m(i) - h^m(i-1)  \leq \wt h^m(i) - \wt h^m(i-1) \quad \text{for all $i \in \mathbb{Z}$}.$$
For any $k \in \mathbb{Z}_{>m}$ and $i\in \mathbb{Z}$, we have 
$$I^{\textup{S}, m}_{(i,k)}(\wt h) \geq I^{\textup{S},  m}_{(i,k)}(h) \qquad \text{and} \qquad J^{\textup{S}, m}_{(i,k)}(\wt h) \leq J^{\textup{S},  m}_{(i,k)}(h).$$
In addition, for any $k \in \mathbb{Z}_{<m}$ and $i\in \mathbb{Z}$, we have 
$$I^{\textup{N}, m}_{(i,k)}(\wt h) \geq I^{\textup{N},  m}_{(i,k)}(h) \qquad \text{and} \qquad J^{\textup{N}, m}_{(i,k)}(\wt h) \leq J^{\textup{N},  m}_{(i,k)}(h).$$
\end{proposition}

\subsection{Burke property} \label{subsec_burke}

In this section, we recall the well-known Burke property of the $M/M/1$ queue.

\begin{proposition}[{Burke property, \cite[Proposition 3.5]{ind_buse}}]\label{burk}
For $0< \lambda < \rho$, consider two sequences of independent random variables, $\{\omega_{(i,m)}\}_{i\in \mathbb{Z}}$ where $\omega_{(i,m)} \sim \Exp(\lambda)$, and $\{\omega_{(i,m+1)}\}_{i\in \mathbb{Z}}$ where $\omega_{(i,m+1)} \sim \Exp(\rho)$. Using these vertex weights, define $h^m$ and $h^{m+1}$ via \eqref{defh}.
Then the following equality of joint distributions holds:
$$(h^m, h^{m+1}) \stackrel{\textup{d}}{=} ((\sigma^m h)^{m+1}, (\sigma^m h)^{m}).$$
\end{proposition}

In other words, when $0<\lambda < \rho$, the queueing map $\sigma^m$ preserves independence and swaps the exponential rates of levels $m$ and $m+1$.

\section{Results for LPP with boundary}\label{proof1}

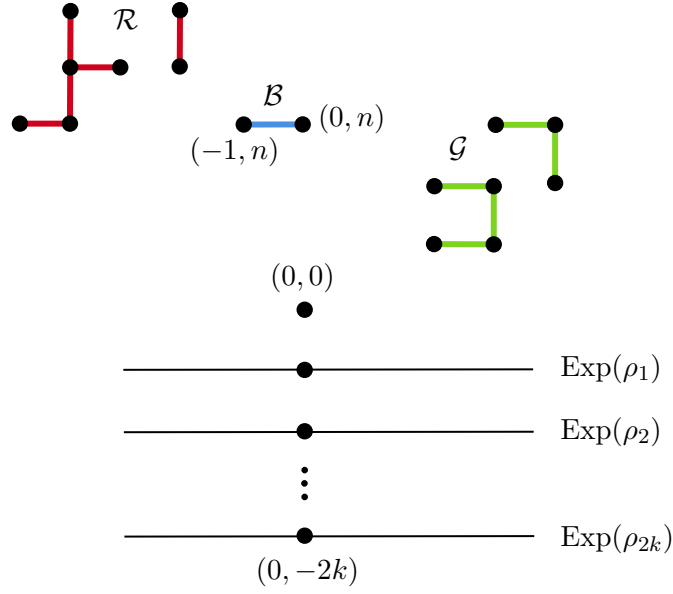
\begin{figure}[t]
\begin{center}

\tikzset{every picture/.style={line width=0.75pt}} 

\begin{tikzpicture}[x=0.75pt,y=0.75pt,yscale=-1,xscale=1]

\draw    (89.77,230.16) -- (295.44,229.92) ;
\draw    (90.05,261.47) -- (295.72,261.23) ;
\draw    (90.32,313.85) -- (296,313.61) ;
\draw [color={rgb, 255:red, 74; green, 144; blue, 226 }  ,draw opacity=1 ][line width=2.25]    (150.2,107.06) -- (179.8,107) ;
\draw [color={rgb, 255:red, 208; green, 2; blue, 27 }  ,draw opacity=1 ][line width=2.25]    (63.28,50.11) -- (63.13,78.36) ;
\draw [color={rgb, 255:red, 126; green, 211; blue, 33 }  ,draw opacity=1 ][line width=2.25]    (245.73,167.05) -- (277.1,167.2) ;
\draw  [fill={rgb, 255:red, 0; green, 0; blue, 0 }  ,fill opacity=1 ] (177.2,200) .. controls (177.2,198.01) and (178.81,196.4) .. (180.8,196.4) .. controls (182.79,196.4) and (184.4,198.01) .. (184.4,200) .. controls (184.4,201.99) and (182.79,203.6) .. (180.8,203.6) .. controls (178.81,203.6) and (177.2,201.99) .. (177.2,200) -- cycle ;
\draw  [fill={rgb, 255:red, 0; green, 0; blue, 0 }  ,fill opacity=1 ] (146.6,107.06) .. controls (146.6,105.07) and (148.21,103.46) .. (150.2,103.46) .. controls (152.19,103.46) and (153.8,105.07) .. (153.8,107.06) .. controls (153.8,109.05) and (152.19,110.66) .. (150.2,110.66) .. controls (148.21,110.66) and (146.6,109.05) .. (146.6,107.06) -- cycle ;
\draw  [fill={rgb, 255:red, 0; green, 0; blue, 0 }  ,fill opacity=1 ] (242.13,167.05) .. controls (242.13,165.06) and (243.75,163.45) .. (245.73,163.45) .. controls (247.72,163.45) and (249.33,165.06) .. (249.33,167.05) .. controls (249.33,169.04) and (247.72,170.65) .. (245.73,170.65) .. controls (243.75,170.65) and (242.13,169.04) .. (242.13,167.05) -- cycle ;
\draw  [fill={rgb, 255:red, 0; green, 0; blue, 0 }  ,fill opacity=1 ] (177.2,230.2) .. controls (177.2,228.21) and (178.81,226.6) .. (180.8,226.6) .. controls (182.79,226.6) and (184.4,228.21) .. (184.4,230.2) .. controls (184.4,232.19) and (182.79,233.8) .. (180.8,233.8) .. controls (178.81,233.8) and (177.2,232.19) .. (177.2,230.2) -- cycle ;
\draw  [fill={rgb, 255:red, 0; green, 0; blue, 0 }  ,fill opacity=1 ] (177.2,261) .. controls (177.2,259.01) and (178.81,257.4) .. (180.8,257.4) .. controls (182.79,257.4) and (184.4,259.01) .. (184.4,261) .. controls (184.4,262.99) and (182.79,264.6) .. (180.8,264.6) .. controls (178.81,264.6) and (177.2,262.99) .. (177.2,261) -- cycle ;
\draw  [fill={rgb, 255:red, 0; green, 0; blue, 0 }  ,fill opacity=1 ] (177.2,313.4) .. controls (177.2,311.41) and (178.81,309.8) .. (180.8,309.8) .. controls (182.79,309.8) and (184.4,311.41) .. (184.4,313.4) .. controls (184.4,315.39) and (182.79,317) .. (180.8,317) .. controls (178.81,317) and (177.2,315.39) .. (177.2,313.4) -- cycle ;
\draw  [fill={rgb, 255:red, 0; green, 0; blue, 0 }  ,fill opacity=1 ] (180.79,279.33) .. controls (181.43,279.32) and (181.94,279.84) .. (181.95,280.47) .. controls (181.95,281.11) and (181.43,281.62) .. (180.8,281.63) .. controls (180.16,281.63) and (179.65,281.11) .. (179.65,280.48) .. controls (179.64,279.84) and (180.16,279.33) .. (180.79,279.33) -- cycle ;
\draw  [fill={rgb, 255:red, 0; green, 0; blue, 0 }  ,fill opacity=1 ] (180.78,292.81) .. controls (181.41,292.81) and (181.93,293.33) .. (181.93,293.96) .. controls (181.93,294.6) and (181.42,295.11) .. (180.78,295.11) .. controls (180.15,295.12) and (179.63,294.6) .. (179.63,293.97) .. controls (179.63,293.33) and (180.14,292.82) .. (180.78,292.81) -- cycle ;
\draw  [fill={rgb, 255:red, 0; green, 0; blue, 0 }  ,fill opacity=1 ] (180.68,286) .. controls (181.31,286) and (181.83,286.52) .. (181.83,287.15) .. controls (181.83,287.79) and (181.32,288.3) .. (180.68,288.3) .. controls (180.05,288.31) and (179.53,287.79) .. (179.53,287.16) .. controls (179.53,286.52) and (180.04,286.01) .. (180.68,286) -- cycle ;
\draw  [fill={rgb, 255:red, 0; green, 0; blue, 0 }  ,fill opacity=1 ] (176.2,107) .. controls (176.2,105.01) and (177.81,103.4) .. (179.8,103.4) .. controls (181.79,103.4) and (183.4,105.01) .. (183.4,107) .. controls (183.4,108.99) and (181.79,110.6) .. (179.8,110.6) .. controls (177.81,110.6) and (176.2,108.99) .. (176.2,107) -- cycle ;
\draw [color={rgb, 255:red, 208; green, 2; blue, 27 }  ,draw opacity=1 ][line width=2.25]    (63.13,78.36) -- (88.24,78.36) ;
\draw  [fill={rgb, 255:red, 0; green, 0; blue, 0 }  ,fill opacity=1 ] (84.64,78.36) .. controls (84.64,76.37) and (86.26,74.76) .. (88.24,74.76) .. controls (90.23,74.76) and (91.84,76.37) .. (91.84,78.36) .. controls (91.84,80.34) and (90.23,81.96) .. (88.24,81.96) .. controls (86.26,81.96) and (84.64,80.34) .. (84.64,78.36) -- cycle ;
\draw [color={rgb, 255:red, 126; green, 211; blue, 33 }  ,draw opacity=1 ][line width=2.25]    (245.9,138) -- (275.68,138) ;
\draw  [fill={rgb, 255:red, 0; green, 0; blue, 0 }  ,fill opacity=1 ] (242.3,138) .. controls (242.3,136.01) and (243.91,134.4) .. (245.9,134.4) .. controls (247.89,134.4) and (249.5,136.01) .. (249.5,138) .. controls (249.5,139.99) and (247.89,141.6) .. (245.9,141.6) .. controls (243.91,141.6) and (242.3,139.99) .. (242.3,138) -- cycle ;
\draw [color={rgb, 255:red, 126; green, 211; blue, 33 }  ,draw opacity=1 ][line width=2.25]    (275.68,138) -- (275.68,167.2) ;
\draw  [fill={rgb, 255:red, 0; green, 0; blue, 0 }  ,fill opacity=1 ] (272.08,138) .. controls (272.08,136.01) and (273.69,134.4) .. (275.68,134.4) .. controls (277.67,134.4) and (279.28,136.01) .. (279.28,138) .. controls (279.28,139.99) and (277.67,141.6) .. (275.68,141.6) .. controls (273.69,141.6) and (272.08,139.99) .. (272.08,138) -- cycle ;
\draw  [fill={rgb, 255:red, 0; green, 0; blue, 0 }  ,fill opacity=1 ] (272.08,167.2) .. controls (272.08,165.21) and (273.69,163.6) .. (275.68,163.6) .. controls (277.67,163.6) and (279.28,165.21) .. (279.28,167.2) .. controls (279.28,169.19) and (277.67,170.8) .. (275.68,170.8) .. controls (273.69,170.8) and (272.08,169.19) .. (272.08,167.2) -- cycle ;
\draw  [fill={rgb, 255:red, 0; green, 0; blue, 0 }  ,fill opacity=1 ] (59.68,50.11) .. controls (59.68,48.12) and (61.29,46.51) .. (63.28,46.51) .. controls (65.27,46.51) and (66.88,48.12) .. (66.88,50.11) .. controls (66.88,52.09) and (65.27,53.71) .. (63.28,53.71) .. controls (61.29,53.71) and (59.68,52.09) .. (59.68,50.11) -- cycle ;
\draw [color={rgb, 255:red, 208; green, 2; blue, 27 }  ,draw opacity=1 ][line width=2.25]    (118.22,49.66) -- (118.08,77.91) ;
\draw  [fill={rgb, 255:red, 0; green, 0; blue, 0 }  ,fill opacity=1 ] (114.48,77.91) .. controls (114.48,75.92) and (116.09,74.31) .. (118.08,74.31) .. controls (120.07,74.31) and (121.68,75.92) .. (121.68,77.91) .. controls (121.68,79.89) and (120.07,81.51) .. (118.08,81.51) .. controls (116.09,81.51) and (114.48,79.89) .. (114.48,77.91) -- cycle ;
\draw  [fill={rgb, 255:red, 0; green, 0; blue, 0 }  ,fill opacity=1 ] (114.62,49.66) .. controls (114.62,47.67) and (116.24,46.06) .. (118.22,46.06) .. controls (120.21,46.06) and (121.82,47.67) .. (121.82,49.66) .. controls (121.82,51.65) and (120.21,53.26) .. (118.22,53.26) .. controls (116.24,53.26) and (114.62,51.65) .. (114.62,49.66) -- cycle ;
\draw [color={rgb, 255:red, 208; green, 2; blue, 27 }  ,draw opacity=1 ][line width=2.25]    (63.13,78.36) -- (62.99,106.6) ;
\draw  [fill={rgb, 255:red, 0; green, 0; blue, 0 }  ,fill opacity=1 ] (59.53,78.36) .. controls (59.53,76.37) and (61.15,74.76) .. (63.13,74.76) .. controls (65.12,74.76) and (66.73,76.37) .. (66.73,78.36) .. controls (66.73,80.34) and (65.12,81.96) .. (63.13,81.96) .. controls (61.15,81.96) and (59.53,80.34) .. (59.53,78.36) -- cycle ;
\draw [color={rgb, 255:red, 208; green, 2; blue, 27 }  ,draw opacity=1 ][line width=2.25]    (37.88,106.6) -- (62.99,106.6) ;
\draw  [fill={rgb, 255:red, 0; green, 0; blue, 0 }  ,fill opacity=1 ] (34.28,106.6) .. controls (34.28,104.62) and (35.89,103) .. (37.88,103) .. controls (39.87,103) and (41.48,104.62) .. (41.48,106.6) .. controls (41.48,108.59) and (39.87,110.2) .. (37.88,110.2) .. controls (35.89,110.2) and (34.28,108.59) .. (34.28,106.6) -- cycle ;
\draw  [fill={rgb, 255:red, 0; green, 0; blue, 0 }  ,fill opacity=1 ] (59.39,106.6) .. controls (59.39,104.62) and (61,103) .. (62.99,103) .. controls (64.98,103) and (66.59,104.62) .. (66.59,106.6) .. controls (66.59,108.59) and (64.98,110.2) .. (62.99,110.2) .. controls (61,110.2) and (59.39,108.59) .. (59.39,106.6) -- cycle ;
\draw [color={rgb, 255:red, 126; green, 211; blue, 33 }  ,draw opacity=1 ][line width=2.25]    (276.7,107.2) -- (306.48,107.2) ;
\draw  [fill={rgb, 255:red, 0; green, 0; blue, 0 }  ,fill opacity=1 ] (273.1,107.2) .. controls (273.1,105.21) and (274.71,103.6) .. (276.7,103.6) .. controls (278.69,103.6) and (280.3,105.21) .. (280.3,107.2) .. controls (280.3,109.19) and (278.69,110.8) .. (276.7,110.8) .. controls (274.71,110.8) and (273.1,109.19) .. (273.1,107.2) -- cycle ;
\draw [color={rgb, 255:red, 126; green, 211; blue, 33 }  ,draw opacity=1 ][line width=2.25]    (306.48,107.2) -- (306.48,136.4) ;
\draw  [fill={rgb, 255:red, 0; green, 0; blue, 0 }  ,fill opacity=1 ] (302.88,107.2) .. controls (302.88,105.21) and (304.49,103.6) .. (306.48,103.6) .. controls (308.47,103.6) and (310.08,105.21) .. (310.08,107.2) .. controls (310.08,109.19) and (308.47,110.8) .. (306.48,110.8) .. controls (304.49,110.8) and (302.88,109.19) .. (302.88,107.2) -- cycle ;
\draw  [fill={rgb, 255:red, 0; green, 0; blue, 0 }  ,fill opacity=1 ] (302.88,136.4) .. controls (302.88,134.41) and (304.49,132.8) .. (306.48,132.8) .. controls (308.47,132.8) and (310.08,134.41) .. (310.08,136.4) .. controls (310.08,138.39) and (308.47,140) .. (306.48,140) .. controls (304.49,140) and (302.88,138.39) .. (302.88,136.4) -- cycle ;

\draw (307.6,306.2) node [anchor=north west][inner sep=0.75pt]    {$\Exp(\rho_{2k})$};
\draw (307.6,252.4) node [anchor=north west][inner sep=0.75pt]    {$\Exp( \rho_2)$};
\draw (307.6,220.8) node [anchor=north west][inner sep=0.75pt]    {$\Exp(\rho_1)$};
\draw (162,175) node [anchor=north west][inner sep=0.75pt]    {$( 0,0)$};
\draw (154.8,322.8) node [anchor=north west][inner sep=0.75pt]    {$( 0,-2k)$};
\draw (158.4,85.86) node [anchor=north west][inner sep=0.75pt]    {$\mathcal  B$};
\draw (186,95.3) node [anchor=north west][inner sep=0.75pt]    {$( 0,n)$};
\draw (83.4,47.8) node [anchor=north west][inner sep=0.75pt]    {$\mathcal  R$};
\draw (251.7,111.8) node [anchor=north west][inner sep=0.75pt]    {$\mathcal G$};
\draw (121.44,111.76) node [anchor=north west][inner sep=0.75pt]    {$(-1,n)$};

\end{tikzpicture}

\captionsetup{width=.8\linewidth}
\caption{
The setup of the LPP with boundary in Proposition \ref{bdry_cov}.}\label{fig1}
\end{center}

\end{figure}
To prove our main theorems, we will use the distributional identity between the Busemann process and the increments of an LPP model with boundary conditions \cite{fan_seppalainen20}, which will be restated precisely as Proposition \ref{B_bdry} in Section \ref{proof_main}. In this section, we begin by introducing the LPP model with a south boundary and prove an important negative association result, as stated in Proposition \ref{bdry_cov}.

The following setup for the LPP with south boundary is illustrated in Figure \ref{fig1}. Let $k \in \mathbb{Z}_{>0}$, and fix parameters \( 1> \rho_1>  \rho_2  > \dots > \rho_{2k} > 0 \).
The boundary is located at level $-2k$, so we just specify the environment $h$ (via \eqref{defh}) at levels $-2k$ and above.
On the negative levels, we let the weights $\{\omega_{(x,-\ell)}\}_{x\in\Z}$ be $\Exp(\rho_\ell)$ random variables, for each $\ell\in\{1,\ldots,2k\}$.
On the non-negative levels, we let the weights $\{\omega_{(x,\ell)}\}_{x\in\Z,\ell\ge0}$ be $\Exp(1)$ random variables.
Furthermore, we assume all weights are independent.

Next, let $\mathcal{B}$ denote the unit edge $[\![(-1,n), (0,n)]\!]$, where $n$ is a fixed large integer. Then, for each value of $j$ in $\{1, 2, \dots , 2k\}$, let  $\mathcal{R}^j = (\mathcal{R}^j_i)_{i=1}^{2k}$ be a sequence of $2k$ unit edges contained in $(-1, n) + \mathbb{Z}_{\leq 0} \times \mathbb{Z}_{\geq 0}$. 
In addition, let $\mathcal{G}^j = (\mathcal{G}^j_i)_{i=1}^{2k}$ be a sequence of $2k$ unit edges contained in $(0,n) + \mathbb{Z}_{\geq 0} \times \mathbb{Z}_{\leq 0}$. 
We assume $n$ is sufficiently large that all the chosen edges in $\{\mathcal{R}_i^j, \mathcal{G}_i^j\}_{i,j=1}^{2k}$ lie strictly above the horizontal line $y=0$.

Next, we introduce increments associated with these unit edges.
Define
\eeq{ \label{G_blue}
G(\mathcal{B})= G^\textup{S}_{-k, (-1,n)}(h) - G^\textup{S}_{-k, (0,n)}(h).
}
For $\mathcal{G}_i^j = [\![\mathbf c, \mathbf d ]\!]$ with $\mathbf d - \mathbf c \in \{\mathbf e_1, -\mathbf e_2 \}$, define 
\eq{ 
G(\mathcal{G}^j_i) = G^\textup{S}_{-j, \mathbf c}(h) - G^\textup{S}_{-j, \mathbf d}(h).
}
Similarly, for edge $\mathcal{R}^j_i = [\![\mathbf e, \mathbf f]\!]$ with $\mathbf f - \mathbf e \in \{\mathbf e_1, -\mathbf e_2 \}$, define
\eq{ 
G(\mathcal{R}^j_i) = G^\textup{S}_{-j, \mathbf e}(h) - G^\textup{S}_{-j, \mathbf f}(h).
}
Our goal in this section is to prove the following proposition.

\begin{proposition}[Negative association or independence in LPP with boundary] \label{bdry_cov}
For any coordinate-wise monotone functions $f$ and $g$ that are monotone in the same direction (that is, either both are increasing or both are decreasing),
it holds that 
\begin{align} \label{bdry_cov1}
\Cov\Big(f\big(G(\mathcal{B})\big), g\big( (G(\mathcal{R}^j_i))_{i,j=1}^{2k} ,(G(\mathcal{G}^j_i))_{i,j=1}^{2k}\big) \Big) \leq 0.
\end{align}
In addition, if we restrict the range of the index $j$, then we have
\begin{align} \label{bdry_cov2}
\Cov\Big(f\big(G(\mathcal{B})\big), g\big( (G(\mathcal{G}^j_i))_{i=1,\dots,2k}^{j=1,\dots,k} , (G(\mathcal{R}^j_i))_{i=1,\dots,2k}^{j=k,\dots,2k} \big) \Big) = 0.
\end{align}
\end{proposition}

\begin{proof}
We start by rewriting all of the LPP increments of $h$ in terms of a new environment $h^*$, which is defined below in \eqref{hstar}.
To simplify the notation, for $m < n - 1$, let 
\eq{ 
\pi^{m,n-1} = \sigma^{n-1} \cdots \sigma^{m+1} \sigma^m \qquad \textup{ and } \qquad \pi_{m,n-1} = \sigma_{m} \cdots \sigma_{n-2} \sigma_{n-1}. 
}
With this notation, define the new environment $h^*$ as 
\begin{equation}\label{hstar}
h^* =\pi^{-k, n-1}\cdots \pi^{-2, n-1}\pi^{-1, n-1} h.
\end{equation}
 By the Burke property (Proposition \ref{burk}), it will be helpful to think of each queueing map as simply swapping the rates of the exponential random variables along each horizontal level. 
With this in mind, we introduce the notation $[a]_b$ to indicate that the exponential random variables along level $b$ have rate $a$. 
Then, the rates in the environment $h$ on levels $n,n-1,\ldots,-2k$ are given by 
\[
\Big([1]_n, [1]_{n-1}, \dots, [1]_{1}, [1]_{0}, [\rho_{1}]_{-1}, [\rho_{2}]_{-2}, \dots, [\rho_{2k}]_{-2k}\Big).
\]
When the first (rightmost) queueing map $\pi^{-1, n-1}$ is applied to $h$ in \eqref{hstar}, because $\rho_1<1$, by the Burke property (Proposition \ref{burk}), the queueing map $\pi^{-1, n-1}$ moves the parameter $\rho_{1}$, originally at level $-1$, up to level $n$ through successive pairwise swaps. Thus, we obtain a new environment $\pi^{-1, n-1}h$ with rates
\[
\Big([\rho_1]_n, [1]_{n-1}, \dots, [1]_{0}, [1]_{-1}, [\rho_2]_{-2}, \dots, [\rho_{2k}]_{-2k}\Big).
\]
Looking back to \eqref{hstar}, the next queueing map $\pi^{-2, n-1}$ is then applied to the environment $\pi^{-1, n-1}h$. Again, because $\rho_2 < \min\{\rho_1, 1\}$, by the Burke property (Proposition \ref{burk}), the new environment $\pi^{-2, n-1}\pi^{-1, n-1}h$ has rates
\[
\Big([\rho_2]_n,[\rho_1]_{n-1}, [1]_{n-2}, \dots, [1]_{-1}, [1]_{-2}, [\rho_3]_{-3}, \dots, [\rho_{2k}]_{-2k}\Big).
\]
This is again analogous to moving the rate $\rho_2$, originally at level $-2$, up to level $n$ through successive pairwise swaps.
The process continues in the same manner, and in the end, we obtain the environment $h^*$, whose rates on levels $n,n-1,\ldots,-2k$ are given by 
\begin{equation}\label{rate_h*}
\Big([\rho_k]_n, [\rho_{k-1}]_{n-1}, \dots, [\rho_{1}]_{n-k+1}, [1]_{n-k}, \dots, [1]_{-k}, [\rho_{k+1}]_{-k-1}, \dots, [\rho_{2k}]_{-2k}\Big).
\end{equation}

The motivation for this definition is the following: the value of $G(\mathcal{B})$ becomes a weight in this new environment $h^*$. More precisely, it holds that  
\eeq{
G(\mathcal{B}) 
&\stackrefpp{G_blue}{isometry_bdy}{=} G^{\textup{S}}_{{-k}, (-1,n)}(h) - G^{\textup{S}}_{{-k}, (0,n)}(h) \\
&\stackref{isometry_bdy}{=}G^{\textup{S}}_{{-k}, (-1,n)}(\pi^{-k, n-1}\cdots \pi^{-2, n-1}\pi^{-1, n-1} h) \\
& \qquad \qquad \qquad\qquad - G^{\textup{S}}_{{-k}, (0,n)}(\pi^{-k, n-1}\cdots \pi^{-2, n-1}\pi^{-1, n-1} h) \\
&\stackrefpp{rep_south}{isometry_bdy}{=} (\pi^{-k, n-1}\cdots \pi^{-2, n-1}\pi^{-1, n-1} h)^n(-1)  \\
&\stackrefp{isometry_bdy}{=} (h^*)^n(-1).\label{reB1}
}

Next, we rewrite $G(\mathcal{G}_i^j)$. The steps are the same for all $i \in\{1, \dots, 2k\}$, but differ depending on the value of $j$.
To start, we consider the case where $j \in\{k+1, \dots, 2k\}$. 
Suppose the endpoints of $\mathcal{G}^{j}_i$ are $\{\mathbf{c}, \mathbf d\}$ with $\mathbf d - \mathbf c \in \{\mathbf e_1, -\mathbf e_2 \}$. 
Suppose vertex $\mathbf d$ lies on level $t$, where $1\leq t \leq n$. 
To proceed with the rewriting, we start by expressing $h$ in terms of $h^*$. This is done in the calculation below, and the resulting formula is given in \eqref{braid_h}.

Due to how we obtain the specific rates of $h^*$ described in \eqref{rate_h*}, all the queueing maps in $h^*$ are non-trivial (i.e.\ not identity maps). Therefore, the cancellation described in Proposition \ref{q_inverse} can be applied here and we obtain
\begin{equation}\label{re_h}
h =\pi_{-1, n-1} \pi_{-2, n-1} \dots \pi_{-k+1, n-1}\pi_{-k, n-1}h^*.
\end{equation}

Next, we will use the braid relation, Proposition~\ref{braid}, to delay the mixing of \((h^*)^n\) with the environment below level $n$ due to the first (rightmost) queueing map \(\pi_{-k,n-1}\) in \eqref{re_h}. 
To do this, first note that for any $p< q \leq r \leq s$, by Proposition \ref{braid}, it holds that
\eeq{
 \pi_{q, r}  \pi_{p, s} & =  \pi_{q,r-1} \pi_{p, r-2} (\sigma_{r} \sigma_{r-1}\sigma_{r})\pi_{r+1, s} \label{br1}\\
& = \pi_{q,r-1} \pi_{p, r-2} ({\sigma_{r-1}} \textcolor{red}{\sigma_{r}}\textcolor{blue}{\sigma_{r-1}})\textcolor{red}{\pi_{r+1,s}}\\
& =  \pi_{q,r-2} \pi_{p, r-3} (\sigma_{r-1} \sigma_{r-2}{\sigma_{r-1}})\textcolor{red}{\pi_{r, s}}\textcolor{blue}{\sigma_{r-1}} \\
& =   \pi_{q,r-2} \pi_{p, r-3} ({\sigma_{r-2}} \textcolor{red}{\sigma_{r-1}}\textcolor{blue}{\sigma_{r-2}})\textcolor{red}{\pi_{r, s}}\textcolor{blue}{\sigma_{r-1}}\\
& =   \pi_{q,r-3} \pi_{p, r-4} (\sigma_{r-2}\sigma_{r-3}{\sigma_{r-2}}) \textcolor{red}{\pi_{r-1, s}}\textcolor{blue}{\pi_{r-2,r-1}} \\
&\hspace{1.2ex}\vdots\\
& = \pi_{p, q-2}(\sigma_{q} \sigma_{q-1} \sigma_q)\textcolor{red}{\pi_{q+1, s}}\textcolor{blue}{\pi_{q,r-1}} \\
& = \pi_{p, q-2}(\sigma_{q-1} \textcolor{red}{\sigma_{q}}\textcolor{blue}{\sigma_{q-1}})\textcolor{red}{\pi_{q+1, s}}\textcolor{blue}{\pi_{q,r-1}}\\
& = \pi_{p, s}\pi_{q-1,r-1}.
}
Applying the above identity repeatedly to the right-hand side of \eqref{re_h}, we obtain 
\begin{align*}
\textup{RHS of \eqref{re_h}} &=\pi_{-1, n-1}\dots\pi_{-k+2, n-1}(\pi_{-k+1, n-1}\pi_{-k, n-1})h^*\\
& = \pi_{-1, n-1}\dots(\pi_{-k+2, n-1}\pi_{-k, n-1})\pi_{-k, n-2}h^*\\
&\hspace{1.2ex}\vdots\\
& =  (\pi_{-1, n-1} \pi_{-k, n-1}) \dots \pi_{-k+1, n-2}\pi_{-k, n-2}h^*\\
& =  \pi_{-k, n-1}\pi_{-2, n-2}\dots \pi_{-k+1, n-2}\pi_{-k, n-2}h^*.
\end{align*}
With the calculation above, equation \eqref{re_h} becomes
\begin{equation}\label{braid_h}
h =  \pi_{-k, n-1}\pi_{-2, n-2}\dots \pi_{-k+1, n-2}\pi_{-k, n-2}h^*.
\end{equation}
We note that the queueing maps $\pi_{-2, n-2}\cdots \pi_{-k+1, n-2}\pi_{-k, n-2}$ appearing on the right-hand side of \eqref{braid_h} do not scramble the weights $(h^*)^n$ with the environment below $y = n$.
The mixing only happens when the last (leftmost) queueing map $\pi_{-k, n-1}$ is applied. 

Using the formula \eqref{braid_h}, we continue with the rewriting. Since $j \geq k+1$, Proposition \ref{isometry} gives us the first equality below. In addition, note $\pi_{-t, n-1}\pi_{-2, n-2}\dots \pi_{-k+1, n-2}\pi_{-k, n-2}h^*$ has rates 
$$\Big([1]_n, \dots,[1]_{t+1},  [\rho_k]_{t},  [1]_{t-1}\dots,  [1]_{-1}, [\rho_{1}]_{-2}, \dots, [\rho_{k-1}]_{-k}, [\rho_{k+1}]_{-k-1}, \dots  [\rho_{2k}]_{-2k}\Big).$$
Therefore, the queueing map $\pi_{-k, t-1}$ acts non-trivially on  $\pi_{-t, n-1}\pi_{-2, n-2}\dots \pi_{-k+1, n-2}\pi_{-k, n-2}h^*$ since $\rho_k< \min \{1, \rho_1, \dots, \rho_{k-1}\}$. Then, Proposition \ref{q_inverse} can be applied, and it gives us the last equality below:
\eeq{
G(\mathcal{G}^j_i) 
& \stackref{isometry_bdy}{=}  G^{\textup{S}}_{-j, \mathbf c}(\pi^{-k, t-1}h) - G^{\textup{S}}_{-j, \mathbf d }(\pi^{-k, t-1}h) \\
&\stackrefpp{braid_h}{isometry_bdy}{=} G^{\textup{S}}_{-j, \mathbf c}(\pi^{-k, t-1} \pi_{-k, n-1}\pi_{-2, n-2}\dots \pi_{-k+1, n-2}\pi_{-k, n-2}h^*) \\
& \qquad\qquad - G^{\textup{S}}_{-j, \mathbf d }(\pi^{-k, t-1} \pi_{-k, n-1}\pi_{-2, n-2}\dots \pi_{-k+1, n-2}\pi_{-k, n-2}h^*) \\
&\stackrefp{isometry_bdy}{=} G^{\textup{S}}_{-j, \mathbf c}(\pi^{-k, t-1} \pi_{-k, t-1}\pi_{t, n-1}\pi_{-2, n-2}\dots \pi_{-k+1, n-2}\pi_{-k, n-2}h^*) \\
& \qquad\qquad - G^{\textup{S}}_{-j, \mathbf d }(\pi^{-k, t-1} \pi_{-k, t-1}\pi_{t, n-1}\pi_{-2, n-2}\dots \pi_{-k+1, n-2}\pi_{-k, n-2}h^*) \\
& \stackrefpp{inverse_eq}{isometry_bdy}{=}  G^{\textup{S}}_{-j, \mathbf c}(\pi_{t, n-1}\pi_{-2, n-2}\dots \pi_{-k+1, n-2}\pi_{-k, n-2}h^*) \\
& \qquad\qquad - G^{\textup{S}}_{-j, \mathbf d }(\pi_{t, n-1}\pi_{-2, n-2}\dots \pi_{-k+1, n-2}\pi_{-k, n-2}h^*).\label{reG1}
}
This finishes the rewriting of $\{G(\mathcal{G}_i^j)\}_{i=1}^{2k}$ for $j = k+1, \dots, 2k$.

Next, let us rewrite $G(\mathcal{G}_i^j)$ for $j =1, \dots, k-1$. Again, suppose the endpoints of $\mathcal{G}^{j}_i$ are $[\![\mathbf{c}, \mathbf d]\!]$ with $\mathbf d - \mathbf c \in \{\mathbf e_1, -\mathbf e_2 \}$. 
To start, we note that the rates in the environment 
$\pi_{-2, n-2}\dots \pi_{-k+1, n-2}\pi_{-k, n-2}h^*$ is given by 
$$\Big([\rho_k]_n, [1]_{n-1}, \dots,[1]_{-1},  [\rho_{1}]_{-2}, [\rho_2]_{-3}, \dots, [\rho_{k-1}]_{-k}, [\rho_{k+1}]_{-k-1}, \dots, [\rho_{2k}]_{-2k}\Big).$$
In particular, for $j = 1, \dots, k-1$, the queueing map $\pi^{-j-1, n-2}$ acts non-trivially on this since $\rho_j < \min\{1, \rho_1, \dots, \rho_{j-1} \}$. Therefore, Proposition \ref{q_inverse} can be applied to obtain the first equality below in \eqref{switch1}. In addition, the rates of $\pi_{-k, n-1}\pi^{-j-1, n-2}\pi_{-2, n-2}\dots \pi_{-k+1, n-2}\pi_{-k, n-2}h^*$ is given by 
$$\Big([\rho_j]_n, [1]_{n-1}, \dots,[1]_{-1},  [\rho_{1}]_{-2}, \dots, [\rho_{j-1}]_{-j}, [\rho_{j+1}]_{-j-1}\dots, [\rho_{2k}]_{-2k}\Big).$$
Since $\rho_j < \min\{1, \rho_1, \dots, \rho_{j-1}\}$, $\pi_{-j, n-1}$ acts non-trivially on this environment, and by Proposition \ref{q_inverse}, we also obtain the last equality below:
\eeq{
&\pi^{-j, n-1}\pi_{-k, n-1}\pi_{-2, n-2}\dots \pi_{-k+1, n-2}\pi_{-k, n-2}h^* \label{switch1}\\
&\stackref{inverse_eq}{=}\pi^{-j, n-1}\textcolor{red}{\pi_{-k, n-1}}(\textcolor{blue}{\pi_{-j-1, n-2}}\pi^{-j-1, n-2})\pi_{-2, n-2}\dots \pi_{-k+1, n-2}\pi_{-k, n-2}h^*\\
& \stackref{br1}{=}(\pi^{-j, n-1}\textcolor{blue}{\pi_{-j, n-1}})\textcolor{red}{\pi_{-k, n-1}}\pi^{-j-1, n-2}\pi_{-2, n-2}\dots \pi_{-k+1, n-2}\pi_{-k, n-2}h^*\\
&\stackref{inverse_eq}{=}\pi_{-k, n-1}\pi^{-j-1, n-2}\pi_{-2, n-2}\dots \pi_{-k+1, n-2}\pi_{-k, n-2}h^*.
}
In the calculation below, the first equality follows from Proposition~\ref{dual_eq1}:
\eeq{
G(\mathcal{G}^j_i) 
&\stackrefp{switch1}{=}  G^{\textup{N}}_{ \mathbf c, n}(\pi^{-j, n-1}h) - G^{\textup{N}}_{ \mathbf d, n}(\pi^{-j, n-1}h)\\
& \stackrefpp{braid_h}{switch1}{=} G^{\textup{N}}_{ \mathbf c, n}(\pi^{-j, n-1}\pi_{-k, n-1}\pi_{-2, n-2}\dots \pi_{-k+1, n-2}\pi_{-k, n-2}h^*) \\
&\qquad \qquad - G^{\textup{N}}_{ \mathbf d, n}(\pi^{-j, n-1}\pi_{-k, n-1}\pi_{-2, n-2}\dots \pi_{-k+1, n-2}\pi_{-k, n-2}h^*)\\
& \stackref{switch1}{=} G^{\textup{N}}_{ \mathbf c, n}(\pi_{-k, n-1}\pi^{-j-1, n-2}\pi_{-2, n-2}\dots \pi_{-k+1, n-2}\pi_{-k, n-2}h^*) \\
&\qquad \qquad - G^{\textup{N}}_{ \mathbf d, n}(\pi_{-k, n-1}\pi^{-j-1, n-2}\pi_{-2, n-2}\dots \pi_{-k+1, n-2}\pi_{-k, n-2}h^*)\label{reG2}.
}
The above is the rewriting of $\{G(\mathcal{G}_i^j)\}_{i=1}^{2k}$ for $j = 1, \dots, k-1$.

Finally, when $j = k$, we can apply Proposition \ref{q_inverse} similar to \eqref{switch1} above and obtain
\eeq{
G(\mathcal{G}^k_i) 
&\stackrefp{inverse_eq}{=}  G^{\textup{N}}_{ \mathbf c, n}(\pi^{-k, n-1}h) - G^{\textup{N}}_{ \mathbf d, n}(\pi^{-k, n-1}h)\\
& \stackref{braid_h}{=} G^{\textup{N}}_{ \mathbf c, n}(\pi^{-k, n-1}\pi_{-k, n-1}\pi_{-2, n-2}\dots \pi_{-k+1, n-2}\pi_{-k, n-2}h^*) \\
&\qquad \qquad- G^{\textup{N}}_{ \mathbf d, n}(\pi^{-k, n-1}\pi_{-k, n-1}\pi_{-2, n-2}\dots \pi_{-k+1, n-2}\pi_{-k, n-2}h^*)\\
& \stackref{inverse_eq}{=} G^{\textup{N}}_{ \mathbf c, n}(\pi_{-2, n-2}\dots \pi_{-k+1, n-2}\pi_{-k, n-2}h^*) \\
&\qquad \qquad - G^{\textup{N}}_{ \mathbf d, n}(\pi_{-2, n-2}\dots \pi_{-k+1, n-2}\pi_{-k, n-2}h^*)\label{reG3}.
}
Together with \eqref{reG1} and \eqref{reG2}, we have finished the rewriting of $\{G(\mathcal{G}_i^j) \}_{i=1}^{2k}$ for all $j \in \{1, \dots, 2k\}$. 

Next, we will rewrite $G(\mathcal{R}_i^j)$. 
Again, let us denote the unit edge of $\mathcal{R}_i^j$ as $[\![\mathbf e, \mathbf f]\!]$, where $\mathbf f - \mathbf e \in \{\mathbf e_1, -\mathbf e_2 \}$.
Note that by Proposition \ref{isometry}, for $j \in\{k+1, \dots, 2k\}$, it holds that 
\begin{equation}\label{reR1}
G(\mathcal{R}_i^j) = G^\textup{S}_{-j, \mathbf e}(h) - G^\textup{S}_{-j, \mathbf f}(h) \stackrel{\textup{Prop.~\ref{isometry}}}{=} G^\textup{S}_{-j, \mathbf e}(h^*) - G^\textup{S}_{-j, \mathbf f}(h^*).
\end{equation}
For $j = k$,  Propositions \ref{isometry} and \ref{nest} together give the second equality below:
\begin{equation}\label{reR2}
G(\mathcal{R}_i^k) = G^\textup{S}_{-k, \mathbf e}(h) - G^\textup{S}_{-k, \mathbf f}(h) = G^\textup{S}_{n, \mathbf e}(h^*) - G^\textup{S}_{n, \mathbf f}(h^*).
\end{equation}
This finishes the rewriting for $j = k, \dots, 2k$, and next, we will do the rewriting for $j = 1, \dots, k-1$. 
Proposition \ref{nest} gives us the first equality below, and the last equality follows from the fact that the LPP $G^\textup{S}_{n, \bbullet}$ only uses weights along and above level $n$:
\eeq{
G(\mathcal{R}_i^j) 
& \stackref{nest_south}{=} G^{\textup{S}}_{{n}, \mathbf e}(\pi^{-j, n-1} h) - G^{\textup{S}}_{{n}, \mathbf f}(\pi^{-j, n-1} h)\\
& \stackrefpp{braid_h}{nest_south}{=} G^{\textup{S}}_{n, \mathbf e}(\pi^{-j, n-1}\pi_{-k, n-1}\pi_{-2, n-2}\dots \pi_{-k+1, n-2}\pi_{-k, n-2}h^*) \\
&\qquad \qquad - G^{\textup{S}}_{n, \mathbf f}(\pi^{-j, n-1}\pi_{-k, n-1}\pi_{-2, n-2}\dots \pi_{-k+1, n-2}\pi_{-k, n-2}h^*) \\
& \stackrefpp{switch1}{nest_south}{=} G^{\textup{S}}_{n, \mathbf e}(\pi_{-k, n-1}\pi^{-j-1, n-2}\pi_{-2, n-2}\dots \pi_{-k+1, n-2}\pi_{-k, n-2}h^*)  \\
&\qquad \qquad  - G^{\textup{S}}_{n, \mathbf f}(\pi_{-k, n-1}\pi^{-j-1, n-2}\pi_{-2, n-2}\dots \pi_{-k+1, n-2}\pi_{-k, n-2}h^*) \\
&\stackrefp{nest_south}{=} G^{\textup{S}}_{n, \mathbf e}(\sigma_{n-1}\pi^{-j-1, n-2}\pi_{-2, n-2}\dots \pi_{-k+1, n-2}\pi_{-k, n-2}h^*)  \\
&\qquad \qquad  - G^{\textup{S}}_{n, \mathbf f}(\sigma_{n-1}\pi^{-j-1, n-2}\pi_{-2, n-2}\dots \pi_{-k+1, n-2}\pi_{-k, n-2}h^*).\label{reR3}
}
With this, we have completed our task of rewriting $G(\mathcal{B}), G(\mathcal{G}_i^j)$ and $G(\mathcal{R}_i^j)$ with respect to $h^*$, which appear in \eqref{reB1}, \eqref{reG1}, \eqref{reG2}, \eqref{reG3}, \eqref{reR1}, \eqref{reR2}, and \eqref{reR3}.

Next, let $\omega^\star = \{\omega^\star_\mathbf{z}\}$ denote the collection of independent exponential random variables whose rates match the rates of $h^*$. Specifically, the rates for levels $n,n-1,\ldots,-2k$ are
\[
\Big([\rho_k]_n, [\rho_{k-1}]_{n-1}, \dots, [\rho_{1}]_{n-k+1}, [1]_{n-k}, \dots, [1]_{-k}, [\rho_{k+1}]_{-k-1}, \dots, [\rho_{2k}]_{-2k}\Big),
\]
and the rates are equal to $1$ everywhere else.
Suppose that the environment $h^\star$ is defined from $\omega^\star$ via the mapping $\eqref{defh}$.
Since the rates of $\omega^\star$ and $h^*$ match, we have
\begin{equation}\label{defhstar}
h^* \stackrel{\textup{d}}{=} h^\star.
\end{equation}
Therefore, as functions of the two environments, it holds that 
\begin{align*}
&\Big(G(\mathcal{B}), (G(\mathcal{G}_i^j))_{i,j=1}^{2k}, (G(\mathcal{R}_i^j))_{i,j=1}^{2k}\Big) \textup{ computed from $h^*$} \\
& \qquad \qquad \stackrel{\textup{d}}{=} \Big(G(\mathcal{B}), (G(\mathcal{G}_i^j))_{i,j=1}^{2k}, (G(\mathcal{R}_i^j))_{i,j=1}^{2k}\Big) \textup{ computed from $h^\star$}.
\end{align*}
Therefore, from now on, we shall always view $G(\mathcal{B}), G(\mathcal{G}_i^j)$ and $G(\mathcal{R}_i^j)$ as computed directly from the environment $h^\star$ (instead of $h^*$) using the formulas \eqref{reB1}, \eqref{reG1}, \eqref{reG2}, \eqref{reG3}, \eqref{reR1}, \eqref{reR2}, and \eqref{reR3}. 

Next, we will show that if we decrease the value of a single weight $G(\mathcal{B}) = (h^\star)^n(-1) = -\omega^\star_{(0,n)}$ in the environment $\omega^\star$, this action does not decrease the values of $G(\mathcal{R}^j_i)$ and $G(\mathcal{G}^j_i)$ for all $i,j = 1, \dots, 2k$. We will split the argument into two parts to highlight that, in some cases, changing the value of the single weight $G(\mathcal{B}) = (h^\star)^n(-1)= -\omega^\star_{(0,n)}$ has no effect on $G(\mathcal{R}_i^j)$ and $G(\mathcal{G}_i^j)$ for certain values of $j$. These results are recorded as the two propositions below.

\begin{lemma}[Monotonicity observation]\label{prop_dec1}
If the value of $G(\mathcal{B}) = -\omega^\star_{(0,n)}$ is decreased in the environment $\omega^\star$, this does not decrease the values of $\{G(\mathcal{G}^j_i)\}_{i=1}^{2k}$ for $j = k+1, \dots, 2k$ and $\{G(\mathcal{R}^j_i)\}_{i=1}^{2k}$ for $j = 1, \dots, k-1$.
\end{lemma}

\begin{lemma}[Independence observation]\label{prop_dec2}
If the value of $G(\mathcal{B}) = -\omega^\star_{(0,n)}$ is changed in the environment $\omega^\star$, this has no effect on the values of $\{G(\mathcal{G}^j_i)\}_{i=1}^{2k}$ for $j = 1, \dots, k$ and $\{G(\mathcal{R}^j_i)\}_{i=1}^{2k}$ for $j = k, \dots, 2k$.
\end{lemma}

We postpone the proofs of these lemmas and first finish the proof of Proposition \ref{bdry_cov}. 
Let $\mathcal{F}$ denote the $\sigma$-algebra generated by the all weights in $\omega^\star$ except $\omega^\star_{(0,n)}$, which is equal to $-G(\mathcal{B})$. 
Recall the assumption that functions $f$ and $g$ are monotone in the same direction. Viewing
$
f(G(\mathcal{B}))$ and $
g\big((G(\mathcal{R}_i^j))_{i,j=1}^{2k}, (G(\mathcal{G}_i^j))_{i,j=1}^{2k}\big)
$
inside the conditional expectation $\mathbb{E}[\cdots|\mathcal{F}]$, Lemmas \ref{prop_dec1} and \ref{prop_dec2} together imply that $f(G(\mathcal{B}))$ is a non-decreasing (respectively, non-increasing) function of $G(\mathcal{B})$, whereas
$
g\big((G(\mathcal{R}_i^j))_{i,j=1}^{2k}, (G(\mathcal{G}_i^j))_{i,j=1}^{2k}\big)
$
is a non-increasing (respectively, non-decreasing) function of $G(\mathcal{B})$. In this case, the Harris--FKG inequality gives us the inequality below:
\begin{align}
\begin{split}  \label{harris}
&\mathbb{E}\Big[f(G(\mathcal{B}))g\big((G(\mathcal{R}_i^j))_{i,j=1}^{2k}, (G(\mathcal{G}_i^j))_{i,j=1}^{2k}\big)\Big]\\
&=\mathbb{E}\Big[\mathbb{E}\Big[f(G(\mathcal{B}))g\big((G(\mathcal{R}_i^j))_{i,j=1}^{2k}, (G(\mathcal{G}_i^j))_{i,j=1}^{2k}\big)\Big|\mathcal{F} \Big]\Big]\\
&  \leq  \mathbb{E}\Big[\mathbb{E}\Big[f(G(\mathcal{B}))\Big|\mathcal{F}\Big]\cdot \mathbb{E}\Big[g\big((G(\mathcal{R}_i^j))_{i,j=1}^{2k}, (G(\mathcal{G}_i^j))_{i,j=1}^{2k}\big)\Big|\mathcal{F} \Big]\Big]\\
& = \mathbb{E}\Big[\mathbb{E}\Big[f(G(\mathcal{B})) \Big]\cdot\mathbb{E}\Big[g\big((G(\mathcal{R}_i^j))_{i,j=1}^{2k}, (G(\mathcal{G}_i^j))_{i,j=1}^{2k}\big)\Big|\mathcal{F} \Big]\Big]\\
& = \mathbb{E}\Big[f(G(\mathcal{B})) \Big]\cdot \mathbb{E}\Big[g\big((G(\mathcal{R}_i^j))_{i,j=1}^{2k}, (G(\mathcal{G}_i^j))_{i,j=1}^{2k}\big)\Big].
\end{split}
\end{align}
This proves \eqref{bdry_cov1}.

To obtain \eqref{bdry_cov2}, we apply Lemma \ref{prop_dec2}. 
Since changing $G(\mathcal{B})$ does not affect the value of $g\big( (G(\mathcal{G}^j_i))_{i=1,\dots,2k}^{j=1,\dots,k} , (G(\mathcal{R}^j_i))_{i=1,\dots,2k}^{j=k,\dots,2k}\big)$, one can apply the Harris-FKG inequality in both directions, and the inequality in \eqref{harris} becomes an equality in this setting. Thus, we obtain that the covariance is equal to zero.
\end{proof}

In the next two subsections, we prove Lemma \ref{prop_dec1} and Lemma \ref{prop_dec2}.

\subsection{Proof of Lemma \ref{prop_dec1}}

Recall environment $\omega^\star$, or equivalently $h^\star$, defined before \eqref{defhstar}. Let us define a new environment   $\omega^\star_{\textup{dec}}$, or equivalently $h^\star_{\textup{dec}}$, where $\omega^\star_{\mathbf{z}} = \omega^\star_{\textup{dec}, \mathbf{z}}$ except that at the vertex $(0, n)$, 
$$- \omega^\star_{(0,n)} \geq - \omega^\star_{\textup{dec}, (0,n)}.$$
To simplify the notation, define
\begin{equation}\label{def_htilde}
\begin{aligned}
\wt h &= \pi_{-2, n-2}\cdots \pi_{-k+1, n-2}\pi_{-k, n-2}h^\star,\\
\wt h_\textup{dec} &= \pi_{-2, n-2}\cdots \pi_{-k+1, n-2}\pi_{-k, n-2}h^\star_\textup{dec}.
\end{aligned}
\end{equation}
Note that $\wt h$ is introduced only as a simplified notation, and the underlying configuration $h^\star$ remains implicit in the background throughout.

Let us consider $G(\mathcal{G}^j_i)$ for $j = k+1, \dots, 2k$, which has been rewritten in \eqref{reG1}. 
With the new notation $\wt h$, the formula \eqref{reG1} can be rewritten as
$$
G(\mathcal{G}^j_i) =
G^{\textup{S}}_{-j, \mathbf c}(\pi_{t, n-1}\wt h)
-
G^{\textup{S}}_{-j, \mathbf d}(\pi_{t, n-1}\wt h),
$$
and our goal is to show 
\begin{equation}\label{ggoal1}
G^{\textup{S}}_{-j, \mathbf c}(\pi_{t, n-1}\wt h_\textup{dec})
-
G^{\textup{S}}_{-j, \mathbf d}(\pi_{t, n-1}\wt h_\textup{dec}) \geq G^{\textup{S}}_{-j, \mathbf c}(\pi_{t, n-1}\wt h)
-
G^{\textup{S}}_{-j, \mathbf d}(\pi_{t, n-1}\wt h).
\end{equation}

To do this, we will split the argument into two cases, depending on whether $\mathbf d - \mathbf c$ is $\mathbf e_1$ or $-\mathbf e_2$. 
Let us look at the case where $\mathbf d - \mathbf c = \mathbf e_1$. 
First note that in the environment $\pi_{t, n-1}\wt h$, decreasing the value of $-\omega^\star_{(0,n)}$ does not change the portion of the environment below level $t$:
\eq{ 
(\pi_{t, n-1}\wt h)^r = (\pi_{t, n-1}\wt h_\textup{dec})^r \qquad \textup{for each }r\leq t-1.
}
In addition, note that for each $a<b$, by Proposition \ref{LPP_queue}, 
\begin{align*}
(\pi_{t, n-1}\wt h)^t(b)  -(\pi_{t, n-1}\wt h)^t(a) 
\stackref{rep_north}{=} G^\textup{N}_{(a, t), n}(\wt h) - G^\textup{N}_{(b,t), n}(\wt h).
\end{align*}
Note that by decreasing $-\omega^\star_{(0,n)}$, Proposition~\ref{inc_bdry} implies that the horizontal increments along the semi-infinite horizontal line $(-\infty, 0]\times \{t\}$, i.e.,
\begin{align*}
I^{\textup{N}, n}_{(i,t)} (\wt h) = G^\textup{N}_{(i-1, t), n}(\wt h)  - G^\textup{N}_{(i, t), n}(\wt h) \quad \text{for each $i \in \mathbb{Z}_{\leq0}$},
\end{align*}
will increase or remain unchanged; in other words, the following holds:
\begin{align}
I^{\textup{N}, n}_{(i,t)} (\wt h)  &\leq I^{\textup{N}, n}_{(i,t)} (\wt h_\textup{dec}) \quad \textup{for each $i \in \mathbb{Z}_{\leq0}$,} \label{inc_I1}\\
I^{\textup{N}, n}_{(i,t)} (\wt h)  & = I^{\textup{N}, n}_{(i,t)} (\wt h_\textup{dec}) \quad \textup{for each $i \in \mathbb{Z}_{>0}$}.\label{inc_I2}
\end{align}
We note that the weights in \eqref{inc_I1} are seen by $G^{\textup{S}}_{-j, (0,t)}(\pi_{t, n-1}\wt h)$ and  $G^{\textup{S}}_{-j, (0,t)}(\pi_{t, n-1}\wt h_\textup{dec})$ along the semi-infinite horizontal line $(-\infty, 0]\times \{t\}$. In particular, the inequality in \eqref{inc_I1} implies
$$G^{\textup{S}}_{-j, (0,t)}(\pi_{t, n-1}\wt h) \leq G^{\textup{S}}_{-j, (0,t)}(\pi_{t, n-1}\wt h_\textup{dec}).$$
On the other hand, because $(\pi_{t, n-1}\wt h)^r = (\pi_{t, n-1}\wt h_\textup{dec})^r$ for $r \leq t-1$, the following equality holds
$$G^{\textup{S}}_{-j, (0,t-1)}(\pi_{t, n-1}\wt h) = G^{\textup{S}}_{-j, (0,t-1)}(\pi_{t, n-1}\wt h_\textup{dec}).$$
Combining the inequality and equality from the two math displays above, it implies that the $J$-increment across the vertical edge $[\![(0,t-1), (0,t) ]\!]$ must increase:
\begin{equation}\label{J_inc}
J^{\textup{S}, -j}_{(0,t)}(\pi_{t, n-1}\wt h_{\textup{dec}}) \geq J^{\textup{S}, -j}_{(0,t)}(\pi_{t, n-1}\wt h).
\end{equation}
Now, using the fact that $(\pi_{t, n-1}\wt h)^r = (\pi_{t, n-1}\wt h_\textup{dec})^r$ for $r \leq t-1$ again, we obtain 
\begin{equation}\label{same_I1}
I^{\textup{S}, -j}_{(i,t-1)}(\pi_{t, n-1}\wt h_\textup{dec}) = I^{\textup{S}, -j}_{(i,t-1)}(\pi_{t, n-1}\wt h) \quad \textup{for each $i\in \mathbb{Z}$}.
\end{equation}
By the ``corner-flipping'' induction from \eqref{flip} (where the $\omega$'s are coming from \eqref{inc_I2}, the $J$'s from \eqref{J_inc}, and the $I$'s from \eqref{same_I1}; see Figure \ref{corner_G}), we must have 
$$I^{\textup{S}, -j}_{(i,t)}(\pi_{t, n-1}\wt h_{\textup{dec}}) \geq I^{\textup{S}, -j}_{(i,t)}(\pi_{t, n-1}\wt h) \quad \textup{for each } i \in \mathbb{Z}_{>0}.$$
This is exactly \eqref{ggoal1} when $\mathbf d - \mathbf c = \mathbf{e}_1$.

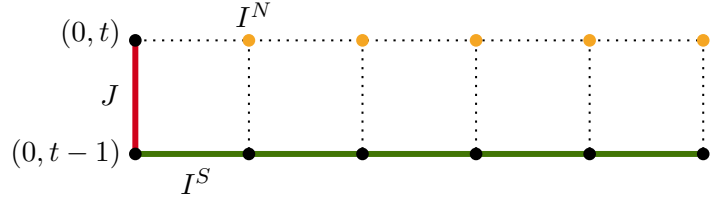
\begin{figure}[t]
\begin{center}

\tikzset{every picture/.style={line width=0.75pt}} 

\begin{tikzpicture}[x=0.75pt,y=0.75pt,yscale=-1,xscale=1]

\draw  [draw opacity=0][dash pattern={on 0.84pt off 2.51pt}] (66.67,79.67) -- (352.17,79.67) -- (352.17,137.33) -- (66.67,137.33) -- cycle ; \draw  [color={rgb, 255:red, 0; green, 0; blue, 0 }  ,draw opacity=1 ][dash pattern={on 0.84pt off 2.51pt}] (66.67,79.67) -- (66.67,137.33)(123.67,79.67) -- (123.67,137.33)(180.67,79.67) -- (180.67,137.33)(237.67,79.67) -- (237.67,137.33)(294.67,79.67) -- (294.67,137.33)(351.67,79.67) -- (351.67,137.33) ; \draw  [color={rgb, 255:red, 0; green, 0; blue, 0 }  ,draw opacity=1 ][dash pattern={on 0.84pt off 2.51pt}] (66.67,79.67) -- (352.17,79.67)(66.67,136.67) -- (352.17,136.67) ; \draw  [color={rgb, 255:red, 0; green, 0; blue, 0 }  ,draw opacity=1 ][dash pattern={on 0.84pt off 2.51pt}]  ;
\draw [color={rgb, 255:red, 208; green, 2; blue, 27 }  ,draw opacity=1 ][line width=2.25]    (66.67,79.67) -- (66.67,136.67) ;
\draw  [fill={rgb, 255:red, 0; green, 0; blue, 0 }  ,fill opacity=1 ] (64,79.67) .. controls (64,78.19) and (65.19,77) .. (66.67,77) .. controls (68.14,77) and (69.33,78.19) .. (69.33,79.67) .. controls (69.33,81.14) and (68.14,82.33) .. (66.67,82.33) .. controls (65.19,82.33) and (64,81.14) .. (64,79.67) -- cycle ;
\draw [color={rgb, 255:red, 65; green, 117; blue, 5 }  ,draw opacity=1 ][line width=2.25]    (66.67,136.67) -- (351.67,136.67) ;
\draw  [fill={rgb, 255:red, 0; green, 0; blue, 0 }  ,fill opacity=1 ] (64,136.67) .. controls (64,135.19) and (65.19,134) .. (66.67,134) .. controls (68.14,134) and (69.33,135.19) .. (69.33,136.67) .. controls (69.33,138.14) and (68.14,139.33) .. (66.67,139.33) .. controls (65.19,139.33) and (64,138.14) .. (64,136.67) -- cycle ;
\draw  [color={rgb, 255:red, 245; green, 166; blue, 35 }  ,draw opacity=1 ][fill={rgb, 255:red, 245; green, 166; blue, 35 }  ,fill opacity=1 ] (121,79.67) .. controls (121,78.19) and (122.19,77) .. (123.67,77) .. controls (125.14,77) and (126.33,78.19) .. (126.33,79.67) .. controls (126.33,81.14) and (125.14,82.33) .. (123.67,82.33) .. controls (122.19,82.33) and (121,81.14) .. (121,79.67) -- cycle ;
\draw  [fill={rgb, 255:red, 0; green, 0; blue, 0 }  ,fill opacity=1 ] (121,136.67) .. controls (121,135.19) and (122.19,134) .. (123.67,134) .. controls (125.14,134) and (126.33,135.19) .. (126.33,136.67) .. controls (126.33,138.14) and (125.14,139.33) .. (123.67,139.33) .. controls (122.19,139.33) and (121,138.14) .. (121,136.67) -- cycle ;
\draw  [color={rgb, 255:red, 245; green, 166; blue, 35 }  ,draw opacity=1 ][fill={rgb, 255:red, 245; green, 166; blue, 35 }  ,fill opacity=1 ] (178,79.67) .. controls (178,78.19) and (179.19,77) .. (180.67,77) .. controls (182.14,77) and (183.33,78.19) .. (183.33,79.67) .. controls (183.33,81.14) and (182.14,82.33) .. (180.67,82.33) .. controls (179.19,82.33) and (178,81.14) .. (178,79.67) -- cycle ;
\draw  [fill={rgb, 255:red, 0; green, 0; blue, 0 }  ,fill opacity=1 ] (178,136.67) .. controls (178,135.19) and (179.19,134) .. (180.67,134) .. controls (182.14,134) and (183.33,135.19) .. (183.33,136.67) .. controls (183.33,138.14) and (182.14,139.33) .. (180.67,139.33) .. controls (179.19,139.33) and (178,138.14) .. (178,136.67) -- cycle ;
\draw  [color={rgb, 255:red, 245; green, 166; blue, 35 }  ,draw opacity=1 ][fill={rgb, 255:red, 245; green, 166; blue, 35 }  ,fill opacity=1 ] (235,79.67) .. controls (235,78.19) and (236.19,77) .. (237.67,77) .. controls (239.14,77) and (240.33,78.19) .. (240.33,79.67) .. controls (240.33,81.14) and (239.14,82.33) .. (237.67,82.33) .. controls (236.19,82.33) and (235,81.14) .. (235,79.67) -- cycle ;
\draw  [fill={rgb, 255:red, 0; green, 0; blue, 0 }  ,fill opacity=1 ] (235,136.67) .. controls (235,135.19) and (236.19,134) .. (237.67,134) .. controls (239.14,134) and (240.33,135.19) .. (240.33,136.67) .. controls (240.33,138.14) and (239.14,139.33) .. (237.67,139.33) .. controls (236.19,139.33) and (235,138.14) .. (235,136.67) -- cycle ;
\draw  [fill={rgb, 255:red, 0; green, 0; blue, 0 }  ,fill opacity=1 ] (292,136.67) .. controls (292,135.19) and (293.19,134) .. (294.67,134) .. controls (296.14,134) and (297.33,135.19) .. (297.33,136.67) .. controls (297.33,138.14) and (296.14,139.33) .. (294.67,139.33) .. controls (293.19,139.33) and (292,138.14) .. (292,136.67) -- cycle ;
\draw  [color={rgb, 255:red, 245; green, 166; blue, 35 }  ,draw opacity=1 ][fill={rgb, 255:red, 245; green, 166; blue, 35 }  ,fill opacity=1 ] (292,79.67) .. controls (292,78.19) and (293.19,77) .. (294.67,77) .. controls (296.14,77) and (297.33,78.19) .. (297.33,79.67) .. controls (297.33,81.14) and (296.14,82.33) .. (294.67,82.33) .. controls (293.19,82.33) and (292,81.14) .. (292,79.67) -- cycle ;
\draw  [color={rgb, 255:red, 245; green, 166; blue, 35 }  ,draw opacity=1 ][fill={rgb, 255:red, 245; green, 166; blue, 35 }  ,fill opacity=1 ] (349,79.67) .. controls (349,78.19) and (350.19,77) .. (351.67,77) .. controls (353.14,77) and (354.33,78.19) .. (354.33,79.67) .. controls (354.33,81.14) and (353.14,82.33) .. (351.67,82.33) .. controls (350.19,82.33) and (349,81.14) .. (349,79.67) -- cycle ;
\draw  [fill={rgb, 255:red, 0; green, 0; blue, 0 }  ,fill opacity=1 ] (349,136.67) .. controls (349,135.19) and (350.19,134) .. (351.67,134) .. controls (353.14,134) and (354.33,135.19) .. (354.33,136.67) .. controls (354.33,138.14) and (353.14,139.33) .. (351.67,139.33) .. controls (350.19,139.33) and (349,138.14) .. (349,136.67) -- cycle ;

\draw (47.5,100.57) node [anchor=north west][inner sep=0.75pt]   {$J$};

\draw (27.5,67.57) node [anchor=north west][inner sep=0.75pt]    {$(0,t)$};
\draw (02.5,127.57) node [anchor=north west][inner sep=0.75pt]   {$(0, t-1)$};

\draw (115.5,58.9) node [anchor=north west][inner sep=0.75pt]    {$I^{N}$};
\draw (87.83,142.57) node [anchor=north west][inner sep=0.75pt]   {$I^{S}$};

\end{tikzpicture}

\captionsetup{width=.8\linewidth}
\caption{Weights used in ``corner-flipping'' induction.
The $I^N$ weights are considered in \eqref{inc_I2}, the $J$ weight in \eqref{J_inc}, and the $I^S$ weights in \eqref{same_I1}.}
\label{corner_G}
\end{center}
\end{figure}

Now, let us look at the case when $\mathbf d - \mathbf c = -\mathbf{e}_2$. To simplify notation, let $s\geq 0$ be fixed so that
$$\mathbf{c} = (s, t+1)\qquad \textup{and} \qquad \mathbf{d} = (s, t).$$
By replacing the horizontal level $t$ in the argument for \eqref{inc_I1} and \eqref{inc_I2} with level $t+1$, the same argument used there yields
\begin{align}
I^{\textup{N}, n}_{(i,t+1)} (\wt h)  &\leq I^{\textup{N}, n}_{(i,t+1)} (\wt h_\textup{dec}) \quad \textup{for each $i \in \mathbb{Z}_{\leq0}$},\label{inc_I11}\\
I^{\textup{N}, n}_{(i,t+1)} (\wt h)  & = I^{\textup{N}, n}_{(i,t+1)} (\wt h_\textup{dec}) \quad \textup{for each $i \in \mathbb{Z}_{>0}$}.\label{inc_I22}
\end{align}

Next, let us write $G^{{\textup{N}}, < s}_{\bbullet, n}$ to denote north-boundary LPP in which the allowed paths are those exiting the boundary strictly to the right of $(-s, n)$.
Similarly, let $G^{{\textup{N}}, \geq s}_{\bbullet, n}$ denote the north-boundary LPP with allowed paths exiting the boundary at or to the left of $(-s, n)$. 
In the rewriting below, the third equality follows from decomposing the LPP depending on where the geodesic jumps up from level $t-1$ to level $t$, the next three equalities follow directly from the definitions and are illustrated in Figure \ref{2lines}, and the last equality follows from the fact that $(\pi_{t+1, n-1}\wt h)^r = \wt h^r$ for $r\leq t-1$:
\begin{align}
&G^{\textup{S}}_{-j, (s, t+1)}(\pi_{t, n-1}\wt h) \nonumber \stackrel{\textup{Prop.~\ref{isometry}}}{=}G^{\textup{S}}_{-j, (s, t+1)}(\pi^{t, t+1}\pi_{t, n-1}\wt h) \nonumber \stackref{inverse_eq}{=}G^{\textup{S}}_{-j, (s, t+1)}(\pi_{t+1, n-1}\wt h) \nonumber\\
 &=\sup_{m\leq s}\Big\{G^{\textup{S}}_{-j, (m, t-1)}(\pi_{t+1, n-1}\wt h) + \textcolor{red}{\sup_{q: m \leq q \leq s} \Big\{(\pi_{t+1, n-1}\wt h)^t(q) -(\pi_{t+1, n-1}\wt h)^t(m-1) }\label{redterm}\\
 & \qquad \qquad  \qquad \qquad \textcolor{red}{+ (\pi_{t+1, n-1}\wt h)^{t+1}(s) -(\pi_{t+1, n-1}\wt h)^{t+1}(q-1)  \Big\} } \Big\}\nonumber\\
  &=\sup_{m\leq s}\Big\{G^{\textup{S}}_{-j, (m, t-1)}(\pi_{t+1, n-1}\wt h) + \textcolor{green!45!black}{\sup_{q: m \leq q \leq s} \Big\{(\pi_{t+1, n-1}\wt h)^t(q) -(\pi_{t+1, n-1}\wt h)^t(m-1) }\nonumber\\
 & \qquad \qquad  \qquad \qquad \textcolor{green!45!black}{ -(\pi_{t+1, n-1}\wt h)^{t+1}(q-1)\Big\}}\Big\} + \textcolor{blue}{(\pi_{t+1, n-1}\wt h)^{t+1}(s)}  \nonumber\\
   &=\sup_{m\leq s}\Big\{G^{\textup{S}}_{-j, (m, t-1)}(\pi_{t+1, n-1}\wt h) + \textcolor{green!45!black}{\sup_{q': m-1 \leq q' \leq s-1} \Big\{(\pi_{t+1, n-1}\wt h)^t(q'+1) -(\pi_{t+1, n-1}\wt h)^t(m-1) }\nonumber\\
 & \qquad \qquad  \qquad \qquad \textcolor{green!45!black}{ -(\pi_{t+1, n-1}\wt h)^{t+1}(q')\Big\}} \Big\}+ \textcolor{blue}{(\pi_{t+1, n-1}\wt h)^{t+1}(s)}  \nonumber\\
&=\sup_{m\leq s}\Big\{G^{\textup{S}}_{-j, (m, t-1)}(\pi_{t+1, n-1}\wt h) + \textcolor{green!45!black}{G^{\textup{N}, <s}_{(m-1, t), t+1}(\pi_{t+1, n-1}\wt h) }\Big\} + \textcolor{blue}{(\pi_{t+1, n-1}\wt h)^{t+1}(s) }.\label{re_t+1}
\end{align}

\begin{figure}[t]
\begin{center}

\tikzset{every picture/.style={line width=0.75pt}} 

\begin{tikzpicture}[x=0.75pt,y=0.75pt,yscale=-1,xscale=1]

\draw [color={rgb, 255:red, 126; green, 211; blue, 33 }  ,draw opacity=1 ][line width=3]    (208.21,123.03) -- (302.39,123.03) ;
\draw [color={rgb, 255:red, 126; green, 211; blue, 33 }  ,draw opacity=1 ] [dash pattern={on 0.84pt off 2.51pt}]  (260.29,80) -- (301.03,124.01) ;
\draw [color={rgb, 255:red, 74; green, 144; blue, 226 }  ,draw opacity=1 ][line width=3]    (57.02,77.15) -- (338.24,77.15) ;
\draw [color={rgb, 255:red, 126; green, 211; blue, 33 }  ,draw opacity=1 ][line width=3]    (57.98,81.05) -- (262.5,81.05) ;
\draw [dashed, color={rgb, 255:red, 208; green, 2; blue, 5 }  ,draw opacity=1 ][line width=3]    (208.82,126.68) -- (302.62,126.71) ;
\draw [dashed, color={rgb, 255:red, 208; green, 2; blue, 5 }  ,draw opacity=1 ][line width=3]    (262.5,81.05) -- (338.43,81.05) ;
\draw  [fill={rgb, 255:red, 0; green, 0; blue, 0 }  ,fill opacity=1 ] (202.71,125.07) .. controls (202.71,122.67) and (204.66,120.72) .. (207.07,120.72) .. controls (209.48,120.72) and (211.43,122.67) .. (211.43,125.07) .. controls (211.43,127.48) and (209.48,129.43) .. (207.07,129.43) .. controls (204.66,129.43) and (202.71,127.48) .. (202.71,125.07) -- cycle ;
\draw [color={rgb, 255:red, 208; green, 2; blue, 27 }  ,draw opacity=1 ] [dash pattern={on 0.84pt off 2.51pt}]  (263.83,80.77) -- (304.58,124.77) ;
\draw  [fill={rgb, 255:red, 0; green, 0; blue, 0 }  ,fill opacity=1 ] (300.22,124.77) .. controls (300.22,122.37) and (302.17,120.42) .. (304.58,120.42) .. controls (306.98,120.42) and (308.93,122.37) .. (308.93,124.77) .. controls (308.93,127.18) and (306.98,129.13) .. (304.58,129.13) .. controls (302.17,129.13) and (300.22,127.18) .. (300.22,124.77) -- cycle ;
\draw  [fill={rgb, 255:red, 0; green, 0; blue, 0 }  ,fill opacity=1 ] (257.86,79.36) .. controls (257.86,76.95) and (259.81,75) .. (262.21,75) .. controls (264.62,75) and (266.57,76.95) .. (266.57,79.36) .. controls (266.57,81.76) and (264.62,83.71) .. (262.21,83.71) .. controls (259.81,83.71) and (257.86,81.76) .. (257.86,79.36) -- cycle ;
\draw  [fill={rgb, 255:red, 0; green, 0; blue, 0 }  ,fill opacity=1 ] (334.43,79.64) .. controls (334.43,77.24) and (336.38,75.29) .. (338.78,75.29) .. controls (341.19,75.29) and (343.14,77.24) .. (343.14,79.64) .. controls (343.14,82.05) and (341.19,84) .. (338.78,84) .. controls (336.38,84) and (334.43,82.05) .. (334.43,79.64) -- cycle ;
\draw  [fill={rgb, 255:red, 0; green, 0; blue, 0 }  ,fill opacity=1 ] (52.43,79.36) .. controls (52.43,76.95) and (54.38,75) .. (56.78,75) .. controls (59.19,75) and (61.14,76.95) .. (61.14,79.36) .. controls (61.14,81.76) and (59.19,83.71) .. (56.78,83.71) .. controls (54.38,83.71) and (52.43,81.76) .. (52.43,79.36) -- cycle ;
\draw    (77.6,73.6) -- (88.6,84.6) ;
\draw    (99.6,73.6) -- (110.6,84.6) ;
\draw    (143.6,73.6) -- (154.6,84.6) ;
\draw    (121.6,73.6) -- (132.6,84.6) ;
\draw    (165.6,73.6) -- (176.6,84.6) ;
\draw    (187.6,73.6) -- (198.6,84.6) ;
\draw    (209.6,73.6) -- (220.6,84.6) ;
\draw    (231.6,73.6) -- (242.6,84.6) ;
\draw    (112.4,40.8) -- (136.08,67.89) ;
\draw [shift={(137.4,69.4)}, rotate = 228.84] [color={rgb, 255:red, 0; green, 0; blue, 0 }  ][line width=0.75]    (10.93,-3.29) .. controls (6.95,-1.4) and (3.31,-0.3) .. (0,0) .. controls (3.31,0.3) and (6.95,1.4) .. (10.93,3.29)   ;

\draw (27.1,92.14) node [anchor=north west][inner sep=0.75pt]    {$( 0,t+1)$};
\draw (344.4,70.2) node [anchor=north west][inner sep=0.75pt]    {$( s,t+1)$};
\draw (168.75,131.77) node [anchor=north west][inner sep=0.75pt]    {$( m,t)$};
\draw (305.5,133.59) node [anchor=north west][inner sep=0.75pt]    {$( q,t)$};
\draw (221.33,51.32) node [anchor=north west][inner sep=0.75pt]    {$( q-1,t+1)$};
\draw (10,20.6) node [anchor=north west][inner sep=0.75pt]    {green and blue cancels to zero here};

\end{tikzpicture}

\captionsetup{width=.8\linewidth}
\caption{Rewriting the two-line LPP (shown in dashed red) from \eqref{redterm} using a sum of LPP with north boundary (shown in solid green, interior lines) and the environment along level $t+1$ (shown in solid blue, top line). This figure essentially says ``green segments $+$ blue segment $=$ red segments.''}\label{2lines}
\end{center}
\end{figure}
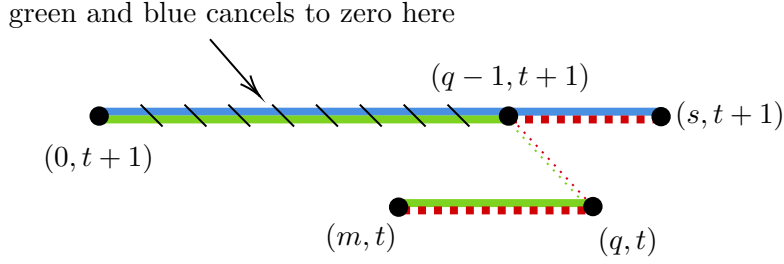

On the other hand, we will rewrite $G^{\textup{S}}_{-j, (s, t)}(\pi_{t, n-1}\wt h)$ as follows. The first equality below follows by decomposing the LPP value along level $t$, and the second equality follows from Proposition \ref{LPP_queue},
\eeq{
&G^{\textup{S}}_{-j, (s, t)}(\pi_{t, n-1}\wt h) \\
&\stackrefp{rep_north}{=}\sup_{m\leq s}\Big\{G^{\textup{S}}_{-j, (m, t-1)}(\pi_{t, n-1}\wt h) + \Big[(\pi_{t, n-1}\wt h)^t(m-1) - (\pi_{t, n-1}\wt h)^t(s)\Big]\Big\}\\
&\stackref{rep_north}{=}\sup_{m\leq s}\Big\{G^{\textup{S}}_{-j, (m, t-1)}(\pi_{t, n-1}\wt h) + \Big[G^{\textup{N}}_{(m-1, t), t+1}(\pi_{t+1, n-1}\wt h)  - G^{\textup{N}}_{(s, t), t+1}(\pi_{t+1, n-1}\wt h)\Big]\Big\}\\
&\stackrefp{rep_north}{=} \sup_{m\leq s}\Big\{G^{\textup{S}}_{-j, (m, t-1)}(\wt h) + G^{\textup{N}}_{(m-1, t), t+1}(\pi_{t+1, n-1}\wt h) \Big\}- G^{\textup{N}}_{(s, t), t+1}(\pi_{t+1, n-1}\wt h).\label{re_t}
}
In particular, in the two rewritings above, we note that the following quantities are the same for both $\wt h$ and $\wt h_\textup{dec}$, because by definition these quantities do not use any of the weights that are changed:
\begin{align}
G^{\textup{S}}_{-j, (m, t-1)}(\wt h) &= G^{\textup{S}}_{-j, (m, t-1)}(\wt h_\textup{dec}),\nonumber\\ 
(\pi_{t+1, n-1}\wt h)^{t+1}(s) &= (\pi_{t+1, n-1}\wt h_{\textup{dec}})^{t+1}(s) \quad \textup{ for $s\geq 0$},\label{same_dec}\\
G^{\textup{N}}_{(s, t), t+1}(\pi_{t+1, n-1}\wt h) &= G^{\textup{N}}_{(s, t), t+1}(\pi_{t+1, n-1}\wt h_\textup{dec}) \quad \textup{ for $s\geq 0$}.\nonumber
\end{align}
In particular, we note that all these terms are the same between $\wt h$ and $\wt h_{\textup{dec}}$ when we try to establish \eqref{ggoal1}.

Using the rewritings of $G^{\textup{S}}_{-j, (s, t+1)}(\pi_{t, n-1}\wt h)$ and
$G^{\textup{S}}_{-j, (s, t)}(\pi_{t, n-1}\wt h)$ in \eqref{re_t+1} and
\eqref{re_t}, together with the observation in \eqref{same_dec}, we see that
to prove \eqref{ggoal1}, it suffices to show
\eeq{
&\sup_{p\leq s} \Big\{r_p + G^{\textup{N}, <s}_{(p-1, t), t+1}(\pi_{t+1, n-1}\wt h_\textup{dec})\Big\} -  \sup_{m \leq s} \Big\{ r_m + G^{\textup{N}, <s}_{(m-1, t), t+1}(\pi_{t+1, n-1}\wt h) \Big\}\label{ver_goal}\\
&\geq \sup_{p\leq s} \Big\{ r_p + G^{\textup{N}}_{(p-1, t), t+1}(\pi_{t+1, n-1}\wt h_\textup{dec})\Big\} - \sup_{m\leq s} \Big\{ r_m + G^{\textup{N}}_{(m-1, t), t+1}(\pi_{t+1, n-1}\wt h)\Big\},
}
where $$r_m = G^{\textup{S}}_{-j, (m, t-1)}(\wt h) = G^{\textup{S}}_{-j, (m, t-1)}(\wt h_\textup{dec}).$$
The intuition behind this inequality \eqref{ver_goal} is that if we restrict the exit time of the north boundary, i.e., $G^{\textup{N}, < s}$, then it becomes more likely to collect the increased weights from \eqref{inc_I11} along the semi-infinite horizontal line $(-\infty, 0]\times \{t+1\}$ compared to $G^{\textup{N}}$. The rigorous argument is given below. 

We will argue by cases. In the environment $\wt h$, at least one of the two cases must happen:
\begin{align}
&\sup_{m \leq s} \Big\{ r_m + G^{\textup{N}}_{(m-1, t), t+1}(\pi_{t+1, n-1}\wt h) \Big\}= \sup_{j\leq s} \Big\{ r_j + G^{\textup{N}, <s}_{(j-1, t), t+1}(\pi_{t+1, n-1}\wt h)\Big\},\label{case1<}\\
& \sup_{m \leq s} \Big\{ r_m + G^{\textup{N}}_{(m-1, t), t+1}(\pi_{t+1, n-1}\wt h) \Big\}= \sup_{j\leq s} \Big\{ r_j + G^{\textup{N}, \geq s}_{(j-1, t), t+1}(\pi_{t+1, n-1}\wt h)\Big\}.\label{case1>}
\end{align}
Similarly, in the environment $\wt h_{\textup{dec}}$, at least one of the two cases must happen:
\begin{align}
&\sup_{m \leq s} \Big\{ r_m + G^{\textup{N}}_{(m-1, t), t+1}(\pi_{t+1, n-1}\wt h_\textup{dec}) \Big\}= \sup_{j\leq s} \Big\{ r_j + G^{\textup{N}, <s}_{(j-1, t), t+1}(\pi_{t+1, n-1}\wt h_\textup{dec})\Big\},\label{case2<}\\
& \sup_{m \leq s} \Big\{ r_m + G^{\textup{N}}_{(m-1, t), t+1}(\pi_{t+1, n-1}\wt h_\textup{dec}) \Big\}= \sup_{j\leq s} \Big\{ r_j + G^{\textup{N}, \geq s}_{(j-1, t), t+1}(\pi_{t+1, n-1}\wt h_\textup{dec})\Big\}.\label{case2>}
\end{align}
There are four different combinations, and we will look at them separately:
\begin{enumerate}[(a)]
\item \label{caseA} \eqref{case1<} + \eqref{case2<}. In this case, \eqref{ver_goal} holds as an equality. 
\item \eqref{case1>} + \eqref{case2>}. In this case, we have 
\begin{align*}
&\sup_{p\leq s} \Big\{ r_p + G^{\textup{N}}_{(p-1, t), t+1}(\pi_{t+1, n-1}\wt h_\textup{dec})\Big\} - \sup_{m\leq s} \Big\{ r_m + G^{\textup{N}}_{(m-1, t), t+1}(\pi_{t+1, n-1}\wt h)\Big\}\\
& = \sup_{p\leq s} \Big\{ r_p + G^{\textup{N}, \geq s}_{(p-1, t), t+1}(\pi_{t+1, n-1}\wt h_\textup{dec})\Big\} - \sup_{m\leq s} \Big\{ r_m + G^{\textup{N}, \geq s}_{(m-1, t), t+1}(\pi_{t+1, n-1}\wt h)\Big\}
= 0,
\end{align*}
where the last equality with zero is because the condition ``$\geq s$'' means that the LPPs inside both suprema only pick up boundary weights among \eqref{inc_I22}, but none of the changed weights in \eqref{inc_I11}. In this case, \eqref{ver_goal} holds because the right-hand side is zero.

\item \eqref{case1>} + \eqref{case2<}. In this case, \eqref{case1>} implies 
\begin{align}\sup_{m \leq s} \Big\{ r_m + G^{\textup{N}}_{(m-1, t), t+1}(\pi_{t+1, n-1}\wt h) \Big\}  \geq  \sup_{m\leq s} \Big\{ r_m + G^{\textup{N}, <s}_{(m-1, t), t+1}(\pi_{t+1, n-1}\wt h)\Big\}.\label{case3_ineq}
\end{align}
Then our desired inequality \eqref{ver_goal} follows from subtracting the inequality \eqref{case3_ineq} from the equality \eqref{case2<}. 

\item  \eqref{case1<} + \eqref{case2>}. In this case, it suffices to assume that \eqref{case1<} holds while \eqref{case1>} fails, since otherwise, case \ref{caseA} applies. Here, we will argue that this case is impossible. To see the contradiction, note that the first and last terms below are the same, while there is a strict inequality between them:
\begin{align*}\sup_{m \leq s} \Big\{ r_m + G^{\textup{N}}_{(m-1, t), t+1}(\pi_{t+1, n-1}\wt h) \Big\}  
& \stackrel{\mbox{\footnotesize\eqref{case1>} fails}}{>}  \sup_{m\leq s} \Big\{ r_m + G^{\textup{N}, \geq s}_{(m-1, t), t+1}(\pi_{t+1, n-1}\wt h)\Big\}\\
& \stackrel{\hphantom{\mbox{\footnotesize\eqref{case1>} fails}}}{=} \sup_{m\leq s} \Big\{ r_m + G^{\textup{N}, \geq s}_{(m-1, t), t+1}(\pi_{t+1, n-1}\wt h_\textup{dec})\Big\}\\
&  \stackrel{\parbox{\widthof{\footnotesize\eqref{case1>} fails}}{\centering\footnotesize\eqref{case2>}}}{=} \sup_{m\leq s} \Big\{ r_m + G^{\textup{N}}_{(m-1, t), t+1}(\pi_{t+1, n-1}\wt h_\textup{dec})\Big\}\\
& \stackrel{\hphantom{\mbox{\footnotesize\eqref{case1>} fails}}}{\ge} \sup_{m\leq s} \Big\{ r_m + G^{\textup{N}}_{(m-1, t), t+1}(\pi_{t+1, n-1}\wt h)\Big\},
\end{align*}
where the first equality holds because none of the changed weights among \eqref{inc_I11} are used in the LPP term. 
\end{enumerate}

This completes the proof of \eqref{ver_goal}. Therefore, we have shown that every vertical increment is non-decreasing when the environment is changed from $\wt h$ to $\wt h_{\mathrm{dec}}$.
We have thus established that decreasing $-\omega^\star_{(0,n)}$ makes
$\{G(\mathcal{G}^j_i)\}_{i=1}^{2k}$ non-decreasing for every
$j\in\{ k+1, \dots, 2k\}$.
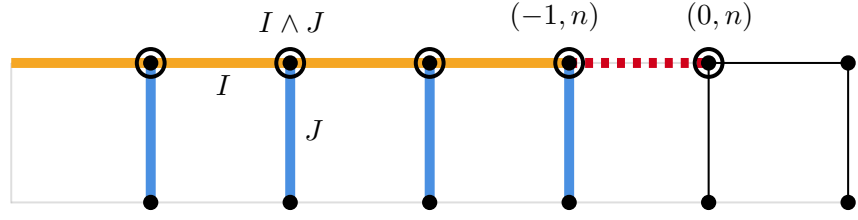
\begin{figure}[t]
\begin{center}

\tikzset{every picture/.style={line width=0.75pt}} 

\begin{tikzpicture}[x=0.75pt,y=0.75pt,yscale=-1,xscale=1]

\draw  [draw opacity=0] (6.4,106.5) -- (426.48,106.5) -- (426.48,176.95) -- (6.4,176.95) -- cycle ; \draw  [color={rgb, 255:red, 155; green, 155; blue, 155 }  ,draw opacity=0.33 ] (6.4,106.5) -- (6.4,176.95)(76.4,106.5) -- (76.4,176.95)(146.4,106.5) -- (146.4,176.95)(216.4,106.5) -- (216.4,176.95)(286.4,106.5) -- (286.4,176.95)(356.4,106.5) -- (356.4,176.95)(426.4,106.5) -- (426.4,176.95) ; \draw  [color={rgb, 255:red, 155; green, 155; blue, 155 }  ,draw opacity=0.33 ] (6.4,106.5) -- (426.48,106.5)(6.4,176.5) -- (426.48,176.5) ; \draw  [color={rgb, 255:red, 155; green, 155; blue, 155 }  ,draw opacity=0.33 ]  ;
\draw [color={rgb, 255:red, 245; green, 166; blue, 35 }  ,draw opacity=1 ][line width=3.75]    (6.4,106.5) -- (286.4,106.5) ;
\draw  [fill={rgb, 255:red, 0; green, 0; blue, 0 }  ,fill opacity=1 ] (423.07,106.5) .. controls (423.07,104.66) and (424.56,103.17) .. (426.4,103.17) .. controls (428.24,103.17) and (429.72,104.66) .. (429.72,106.5) .. controls (429.72,108.34) and (428.24,109.83) .. (426.4,109.83) .. controls (424.56,109.83) and (423.07,108.34) .. (423.07,106.5) -- cycle ;
\draw  [fill={rgb, 255:red, 0; green, 0; blue, 0 }  ,fill opacity=1 ] (423.07,176.5) .. controls (423.07,174.66) and (424.56,173.17) .. (426.4,173.17) .. controls (428.24,173.17) and (429.72,174.66) .. (429.72,176.5) .. controls (429.72,178.34) and (428.24,179.83) .. (426.4,179.83) .. controls (424.56,179.83) and (423.07,178.34) .. (423.07,176.5) -- cycle ;
\draw  [fill={rgb, 255:red, 0; green, 0; blue, 0 }  ,fill opacity=1 ] (353.07,176.5) .. controls (353.07,174.66) and (354.56,173.17) .. (356.4,173.17) .. controls (358.24,173.17) and (359.72,174.66) .. (359.72,176.5) .. controls (359.72,178.34) and (358.24,179.83) .. (356.4,179.83) .. controls (354.56,179.83) and (353.07,178.34) .. (353.07,176.5) -- cycle ;
\draw [color={rgb, 255:red, 74; green, 144; blue, 226 }  ,draw opacity=1 ][line width=3.75]    (76.4,106.5) -- (76.4,176.5) ;
\draw [color={rgb, 255:red, 74; green, 144; blue, 226 }  ,draw opacity=1 ][line width=3.75]    (146.4,106.5) -- (146.4,176.5) ;
\draw [color={rgb, 255:red, 74; green, 144; blue, 226 }  ,draw opacity=1 ][line width=3.75]    (216.4,106.5) -- (216.4,176.5) ;
\draw [color={rgb, 255:red, 74; green, 144; blue, 226 }  ,draw opacity=1 ][line width=3.75]    (286.4,106.5) -- (286.4,176.5) ;
\draw [line width=0.75]    (356.4,106.5) -- (356.4,176.5) ;
\draw [line width=0.75]    (426.4,106.5) -- (426.4,176.5) ;
\draw [dashed, color={rgb, 255:red, 208; green, 2; blue, 27 }  ,draw opacity=1 ][line width=3.75]    (286.4,106.5) -- (356.4,106.5) ;
\draw  [fill={rgb, 255:red, 0; green, 0; blue, 0 }  ,fill opacity=1 ] (283.07,106.5) .. controls (283.07,104.66) and (284.56,103.17) .. (286.4,103.17) .. controls (288.24,103.17) and (289.72,104.66) .. (289.72,106.5) .. controls (289.72,108.34) and (288.24,109.83) .. (286.4,109.83) .. controls (284.56,109.83) and (283.07,108.34) .. (283.07,106.5) -- cycle ;
\draw  [fill={rgb, 255:red, 0; green, 0; blue, 0 }  ,fill opacity=1 ] (353.07,106.5) .. controls (353.07,104.66) and (354.56,103.17) .. (356.4,103.17) .. controls (358.24,103.17) and (359.72,104.66) .. (359.72,106.5) .. controls (359.72,108.34) and (358.24,109.83) .. (356.4,109.83) .. controls (354.56,109.83) and (353.07,108.34) .. (353.07,106.5) -- cycle ;
\draw  [fill={rgb, 255:red, 0; green, 0; blue, 0 }  ,fill opacity=1 ] (213.07,106.5) .. controls (213.07,104.66) and (214.56,103.17) .. (216.4,103.17) .. controls (218.24,103.17) and (219.72,104.66) .. (219.72,106.5) .. controls (219.72,108.34) and (218.24,109.83) .. (216.4,109.83) .. controls (214.56,109.83) and (213.07,108.34) .. (213.07,106.5) -- cycle ;
\draw  [fill={rgb, 255:red, 0; green, 0; blue, 0 }  ,fill opacity=1 ] (213.07,176.5) .. controls (213.07,174.66) and (214.56,173.17) .. (216.4,173.17) .. controls (218.24,173.17) and (219.72,174.66) .. (219.72,176.5) .. controls (219.72,178.34) and (218.24,179.83) .. (216.4,179.83) .. controls (214.56,179.83) and (213.07,178.34) .. (213.07,176.5) -- cycle ;
\draw  [fill={rgb, 255:red, 0; green, 0; blue, 0 }  ,fill opacity=1 ] (283.07,176.5) .. controls (283.07,174.66) and (284.56,173.17) .. (286.4,173.17) .. controls (288.24,173.17) and (289.72,174.66) .. (289.72,176.5) .. controls (289.72,178.34) and (288.24,179.83) .. (286.4,179.83) .. controls (284.56,179.83) and (283.07,178.34) .. (283.07,176.5) -- cycle ;
\draw  [fill={rgb, 255:red, 0; green, 0; blue, 0 }  ,fill opacity=1 ] (143.07,176.5) .. controls (143.07,174.66) and (144.56,173.17) .. (146.4,173.17) .. controls (148.24,173.17) and (149.72,174.66) .. (149.72,176.5) .. controls (149.72,178.34) and (148.24,179.83) .. (146.4,179.83) .. controls (144.56,179.83) and (143.07,178.34) .. (143.07,176.5) -- cycle ;
\draw  [fill={rgb, 255:red, 0; green, 0; blue, 0 }  ,fill opacity=1 ] (143.07,106.5) .. controls (143.07,104.66) and (144.56,103.17) .. (146.4,103.17) .. controls (148.24,103.17) and (149.72,104.66) .. (149.72,106.5) .. controls (149.72,108.34) and (148.24,109.83) .. (146.4,109.83) .. controls (144.56,109.83) and (143.07,108.34) .. (143.07,106.5) -- cycle ;
\draw  [color={rgb, 255:red, 0; green, 0; blue, 0 }  ,draw opacity=1 ][fill={rgb, 255:red, 0; green, 0; blue, 0 }  ,fill opacity=1 ] (73.07,106.5) .. controls (73.07,104.66) and (74.56,103.17) .. (76.4,103.17) .. controls (78.24,103.17) and (79.72,104.66) .. (79.72,106.5) .. controls (79.72,108.34) and (78.24,109.83) .. (76.4,109.83) .. controls (74.56,109.83) and (73.07,108.34) .. (73.07,106.5) -- cycle ;
\draw  [fill={rgb, 255:red, 0; green, 0; blue, 0 }  ,fill opacity=1 ] (73.07,176.5) .. controls (73.07,174.66) and (74.56,173.17) .. (76.4,173.17) .. controls (78.24,173.17) and (79.72,174.66) .. (79.72,176.5) .. controls (79.72,178.34) and (78.24,179.83) .. (76.4,179.83) .. controls (74.56,179.83) and (73.07,178.34) .. (73.07,176.5) -- cycle ;
\draw  [line width=1.5]  (279.42,106.5) .. controls (279.42,102.65) and (282.55,99.52) .. (286.4,99.52) .. controls (290.25,99.52) and (293.37,102.65) .. (293.37,106.5) .. controls (293.37,110.35) and (290.25,113.48) .. (286.4,113.48) .. controls (282.55,113.48) and (279.42,110.35) .. (279.42,106.5) -- cycle ;
\draw  [line width=1.5]  (349.42,106.5) .. controls (349.42,102.65) and (352.55,99.52) .. (356.4,99.52) .. controls (360.25,99.52) and (363.37,102.65) .. (363.37,106.5) .. controls (363.37,110.35) and (360.25,113.48) .. (356.4,113.48) .. controls (352.55,113.48) and (349.42,110.35) .. (349.42,106.5) -- cycle ;
\draw  [line width=1.5]  (209.42,106.5) .. controls (209.42,102.65) and (212.55,99.52) .. (216.4,99.52) .. controls (220.25,99.52) and (223.37,102.65) .. (223.37,106.5) .. controls (223.37,110.35) and (220.25,113.48) .. (216.4,113.48) .. controls (212.55,113.48) and (209.42,110.35) .. (209.42,106.5) -- cycle ;
\draw  [line width=1.5]  (139.42,106.5) .. controls (139.42,102.65) and (142.55,99.52) .. (146.4,99.52) .. controls (150.25,99.52) and (153.37,102.65) .. (153.37,106.5) .. controls (153.37,110.35) and (150.25,113.48) .. (146.4,113.48) .. controls (142.55,113.48) and (139.42,110.35) .. (139.42,106.5) -- cycle ;
\draw  [line width=1.5]  (69.42,106.5) .. controls (69.42,102.65) and (72.55,99.52) .. (76.4,99.52) .. controls (80.25,99.52) and (83.37,102.65) .. (83.37,106.5) .. controls (83.37,110.35) and (80.25,113.48) .. (76.4,113.48) .. controls (72.55,113.48) and (69.42,110.35) .. (69.42,106.5) -- cycle ;
\draw [line width=0.75]    (356.4,106.5) -- (426.4,106.5) ;

\draw (129,79.4) node [anchor=north west][inner sep=0.75pt]    {$I\wedge J$};
\draw (152,133.9) node [anchor=north west][inner sep=0.75pt]    {$J$};
\draw (107.5,110.9) node [anchor=north west][inner sep=0.75pt]    {$I$};
\draw (343.5,74.4) node [anchor=north west][inner sep=0.75pt]    {$( 0,n)$};
\draw (255,74.4) node [anchor=north west][inner sep=0.75pt]    {$( -1,n)$};

\end{tikzpicture}

\captionsetup{width=.8\linewidth}
\caption{If the weight on the red dashed edge increases, then the blue vertical increments corresponding to the LPP with north boundary decrease, while the orange horizontal increments in remain unchanged. Therefore, the dual weights, defined by $I \wedge J$, also decrease.}\label{R_2lines}
\end{center}
\end{figure}

Finally, let us look at $\{\mathcal{R}_i^j\}_{i=1}^{2k}$ where $j\in\{ 1, \dots, k-1\}$, with the rewriting in \eqref{reR3}. Let us denote the unit edge of $\mathcal{R}_i^j$ as $[\![\mathbf e, \mathbf f]\!]$, where $\mathbf f - \mathbf e \in \{\mathbf e_1, -\mathbf e_2 \}$.
Adapting the notation $\wt h$ from \eqref{def_htilde}, we have 
$$G(\mathcal{R}_i^j) = G^\textup{S}_{n, \mathbf e}( \sigma_{n-1}\pi^{-j-1, n-2} \wt h) - G^\textup{S}_{n, \mathbf f}( \sigma_{n-1}\pi^{-j-1, n-2} \wt h).$$

The next part of the argument is illustrated in Figure \ref{R_2lines}. Let us note that the boundary of $G^\textup{S}_{n, \bbullet}(\sigma_{n-1}\pi^{-j-1, n-2} \wt h)$, namely $(\sigma_{n-1}\pi^{-j-1, n-2} \wt h)^n$, is computed from the dual weights along level $n$ of the LPP with north boundary $G^\textup{N}_{\bbullet, n}(\pi^{-j-1, n-2} \wt h)$.
By Proposition \ref{inc_bdry}, if we decrease the value of $-\omega^\star_{(0,n)} = (h^\star)^n(-1)$, which corresponds to an increase of the boundary value because of the additional negative sign for defining $G^\textup{N}$ in \eqref{GN}, the vertical increments $\{J^{\textup{N}, n}_{(i, n)}(\pi^{-j-1, n-2} \wt h)\}_{i \leq -1}$ (shown in blue in Figure \ref{R_2lines}) will become smaller. Note that the horizontal increments $\{I^{\textup{N}, n}_{(i, n)}(\pi^{-j-1, n-2} \wt h)\}_{i \leq -1}$ (shown in orange in Figure \ref{R_2lines}) are not affected by the decrease in $-\omega^\star_{(0,n)}$. Therefore, the dual weights
$$\Big\{I^{\textup{N}, n}_{(i, n)}(\pi^{-j-1, n-2} \wt h) \wedge J^{\textup{N}, n}_{(i, n)}(\pi^{-j-1, n-2} \wt h)\Big\}_{i \leq -1}$$
will become smaller. By definition, for $i \leq -1$, 
$$(\sigma_{n-1}\pi^{-j-1, n-2}\wt h)^n(i) - (\sigma_{n-1}\pi^{-j-1, n-2}\wt h)^n(i-1)=  I^{\textup{N}, n}_{(i, n)}(\pi^{-j-1, n-2}\wt h) \wedge J^{\textup{N}, n}_{(i, n)}(\pi^{-j-1, n-2}\wt h),$$
therefore we have shown 
\eeq{
&(\sigma_{n-1}\pi^{-j-1, n-2}\wt h)^n(i) - (\sigma_{n-1}\pi^{-j-1, n-2}\wt h)^n(i-1)\label{dual_dec}\\
&\geq (\sigma_{n-1}\pi^{-j-1, n-2}\wt h_\textup{dec})^n(i) - (\sigma_{n-1}\pi^{-j-1, n-2}\wt h_\textup{dec})^n(i-1) \quad \text{for $i \leq -1$}.
}

Now, because both $\mathbf e$ and $\mathbf f$ are along or to the left of the vertical line $x = -1$, this means that, by definition, the weight $(\sigma_{n-1}\pi^{-j-1, n-2}\wt h)^n(-1)$ must appear in the passage values $G^\textup{S}_{n, \mathbf e}( \sigma_{n-1}\pi^{-j-1, n-2} \wt h)$ and $G^\textup{S}_{n, \mathbf f}( \sigma_{n-1}\pi^{-j-1, n-2} \wt h)$, because both of their geodesics must take a $-\mathbf{e}_1$ step at the start, going from $(0,n)$ to $(-1, n)$. This means no matter what the value $(\sigma_{n-1}\pi^{-j-1, n-2}\wt h)^n(-1)$ is, it does not affect the difference
$$G^\textup{S}_{n, \mathbf e}( \sigma_{n-1}\pi^{-j-1, n-2} \wt h) - G^\textup{S}_{n, \mathbf f}( \sigma_{n-1}\pi^{-j-1, n-2} \wt h)$$
due to its cancellation. Therefore, $(\sigma_{n-1}\pi^{-j-1, n-2}\wt h)^n(-1)$ and $(\sigma_{n-1}\pi^{-j-1, n-2}\wt h_\textup{dec})^n(-1)$ may be simply treated as ``zero'' when looking at the difference of the last-passage values. 

This observation together with \eqref{dual_dec} means Proposition \ref{inc_bdry} can be applied here, and we obtain
\begin{align*}
&G^\textup{S}_{n, \mathbf e}( \sigma_{n-1}\pi^{-j-1, n-2} \wt h) - G^\textup{S}_{n, \mathbf f}( \sigma_{n-1}\pi^{-j-1, n-2} \wt h) \\
&\leq G^\textup{S}_{n, \mathbf e}( \sigma_{n-1}\pi^{-j-1, n-2} \wt h_\textup{dec}) - G^\textup{S}_{n, \mathbf f}( \sigma_{n-1}\pi^{-j-1, n-2} \wt h_\textup{dec}),
\end{align*}
which completes the argument that the increments $G(\mathcal{R}_i^j)$ are non-decreasing for $j \in\{ 1, \dots, k-1\}$. With this, we have finished the proof of Lemma \ref{prop_dec1}.

\subsection{Proof of Lemma \ref{prop_dec2}}

Again, let us use the short notation $\wt h$ as in \eqref{def_htilde}. 
First, let us look at $G(\mathcal{G}^j_i)$ where $j \in\{ 1, \dots, k-1\}$. 
In this case, our rewriting from \eqref{reG2} is 
$$
G(\mathcal{G}^j_i) =
G^{\textup{N}}_{ \mathbf c, n}(\pi_{-k, n-1}\pi^{-j-1, n-2}\wt h)  - G^{\textup{N}}_{ \mathbf d, n}(\pi_{-k, n-1}\pi^{-j-1, n-2}\wt h).
$$
Then, note that by definition, if we look at the weights to the right of the vertical line $x=0$, the two environments are the same:
$$\pi_{-k, n-1}\pi^{-j-1, n-2}\wt h = \pi_{-k, n-1}\pi^{-j-1, n-2}\wt h_\textup{dec} \quad \textup{to the right of $x=0$}.$$
In addition, for the LPPs with north boundaries $G^{\textup{N}}_{ \mathbf c, n}$ and $G^{\textup{N}}_{ \mathbf d, n}$, because both points $\mathbf{c}$ and $\mathbf d$ are at or to the right of $x=0$, by defintion, the will only use weights at or to the right of $x= 0$. This means that for $\mathbf y \in \{\mathbf c, \mathbf d\}$, 
$$G^{\textup{N}}_{ \mathbf y, n}(\pi_{-k, n-1}\pi^{-j-1, n-2}\wt h)  = G^{\textup{N}}_{ \mathbf y, n}(\pi_{-k, n-1}\pi^{-j-1, n-2}\wt h_\textup{dec}).$$

Next, we look at $G(\mathcal{G}^j_i)$ where $j = k$. In this case, the rewriting is given in \eqref{reG3}, and it follows from the same argument as above; that is, for $\mathbf y \in \{\mathbf c, \mathbf d\}$,
$$G^{\textup{N}}_{ \mathbf y, n}(\wt h)  = G^{\textup{N}}_{ \mathbf y, n}(\wt h_\textup{dec}),$$
because both LPPs only use weights to the right of $x= 0$, which are the same for $\wt h$ and $\wt h_\textup{dec}$.

Next, we move on to $G(\mathcal{R}^j_i)$ for $j \in\{k+1, \dots, 2k\}$, where the rewriting is given in \eqref{reR1}. The increments remain unchanged because the geodesics of the two LPPs in \eqref{reR1} cannot pass through the edge $[\![(-1,n),(0,n)]\!]$, whose weight was changed from $\wt h$ to $\wt h_{\textup{dec}}$.

Lastly, we consider $G(\mathcal{R}^j_i)$ for $j = k$, where the rewriting is given in \eqref{reR2}. This increment also remains unchanged because both geodesics of the LPPs appearing in \eqref{reR2} must pass through the edge $[\![(-1,n),(0,n)]\!]$, whose weight has been modified. Consequently, when taking their difference, the contribution from this edge cancels out. This completes the proof of Lemma \ref{prop_dec2}.

\section{Proofs of main results}\label{proof_main}

In this section, we use Proposition \ref{bdry_cov} to prove results stated in Section~\ref{sec:main}. 
In order to do this, we will first state the equality in distribution between the Busemann process and the increments of LPP with boundary.
To make this connection, we rotate the original model (with independent~$\Exp(1)$ weights) by $180^\circ$. 
In particular, we replace up-right paths with down-left paths, and we replace the limiting direction $\boldsymbol{\xi}[\rho]$ from \eqref{char1} with $-\boldsymbol{\xi}[\rho]$, where $\rho \in (0,1)$.
In this setting, we define the Busemann function as follows. Let $\mathbf{v}_n$ be a sequence of vectors such that $|\mathbf{v}_n|_1 \to \infty$ and $\mathbf{v}_n/|\mathbf{v}_n|_1 \to - \boldsymbol{\xi}[\rho]$. 
The (rotated) Busemann function with direction parameter $\rho$ is the following almost sure limit  
\begin{equation}\label{rota_buse}
\wt B_{\mathbf{x}, \mathbf{y}}^\rho = \lim_{n\to \infty}  G_{\mathbf{v}_n, \mathbf x}(\wt h) - G_{\mathbf{v}_n, \mathbf y}(\wt h),
\end{equation}
where $\wt h$ is an environment of independent~$\Exp(1)$ weights.

With this setup, Theorem 3.2 and Lemma 3.3 of \cite{fan_seppalainen20} give us the following distributional equality, which we will state below. 

Fix any positive integer $k$, and recall the LPP model with boundary defined at the beginning of Section \ref{proof1}. Fix parameters $\{\rho_j\}_{j=1}^{2k}$ where $1> \rho_1>  \rho_2  > \dots > \rho_{2k} > 0$. For each parameter $\rho_j$, we fix $2k$ pairs of lattice points $\{\mathbf{a}_i^j, \mathbf{b}_i^j\} \subset \mathbb{Z} \times \mathbb{Z}_{\geq 0}$ for $i = 1, \dots, 2k$.  

\begin{proposition}[{Joint distribution of Busemann functions, \cite[Theorem 3.2 and Lemma 3.3]{fan_seppalainen20}}] \label{B_bdry}
Given parameters $\{\rho_j\}_{j=1}^{2k}$ where $1> \rho_1>  \rho_2  > \dots > \rho_{2k} > 0$, let $h$ be environment from the beginning of Section \ref{proof1}.
For any lattice points $\big((\mathbf{a}_i^j, \mathbf{b}_i^j)\big)_{i,j=1}^{2k}$ in $\Z\times\Z_{\ge0}$, we have the following distributional equality of random vectors:
$$\big(\wt B^{\rho_j}_{\mathbf{a}_i^j, \mathbf{b}_i^j}\big)_{i,j=1}^{2k} \stackrel{\textup{d}}{=} \big( G^\textup{S}_{-j,\mathbf{a}_i^j}(h) - G^\textup{S}_{-j,\mathbf{b}_i^j}(h)\big)_{i,j=1}^{2k}.$$
\end{proposition}

\subsection{Proofs of Theorem \ref{thm_1} and Theorem \ref{imp_ind}}

Let $\mathbf{a}, \mathbf{b}$ be two nearest-neighbor lattice points such that $\mathbf{b} - \mathbf{a} \in \{\mathbf{e}_1, -\mathbf{e_2}\}$. 
This will be the new unit edge $\mathcal{B} = [\![\mathbf a,\mathbf b]\!]$. Define the collections $\{\mathcal{R}_i^j\}_{i,j=1}^{2k}$ and $\{\mathcal{G}_i^j\}_{i,j=1}^{2k}$ analogously to the beginning of Section \ref{proof1} with respect to this new edge $\mathcal{B}$.
Letting $B^\rho_k(\mathcal{B}), B^{\rho_j}(\mathcal{R}_i^j), B^{\rho_j}(\mathcal{G}_i^j)$ denote the Busemann functions (defined in \eqref{rota_buse}) across these edges with respect to their ordered endpoints, we have the following result.

\begin{proposition}[Negative association or independence with unit edges]\label{prop_1edge}
For any coordinate-wise monotone functions $f$ and $g$ that are monotone in the same direction, it holds that 
\begin{equation}\label{b_cov}
\Cov\Big(f\big(B^{\rho_k}(\mathcal{B})\big), g\big( (B^{\rho_j}(\mathcal{R}^j_i))_{i,j=1}^{2k} ,(B^{\rho_j}(\mathcal{G}^j_i))_{i,j=1}^{2k}\big) \Big) \leq 0.
\end{equation}
In addition, if we restrict the range of the index $j$, then we have
\begin{equation}\label{b_ind}
\Cov\Big(f\big(B^{\rho_k}(\mathcal{B})\big), g\big( (B^{\rho_j}(\mathcal{R}^j_i))_{i=1, \dots, 2k}^{j=1,\dots, k} ,(B^{\rho_j}(\mathcal{G}^j_i))_{i=1 \dots 2k}^{j=k,\dots, 2k}\big) \Big) = 0.
\end{equation}
\end{proposition}

\begin{proof}
First, note that because of the distributional equality stated in Proposition \ref{B_bdry}, when $\mathcal{B} = (\mathbf a, \mathbf b)$ with $\mathbf{b} - \mathbf{a} = \mathbf{e}_1$, the results follow immediately from Proposition \ref{bdry_cov}.

For the case when $\mathbf{b} - \mathbf{a} = -\mathbf{e}_2$, we will use the fact that our i.i.d.~environment is symmetric with respect to the diagonal $x=y$. 
Fix $\mathbf{x}, \mathbf{y}$ in $\mathbb{Z}^2$, and let $\mathbf{x}_{\textup{r}}$ and $\mathbf{y}_{\textup{r}}$ denote the reflection of $\mathbf{x}$ and $\mathbf{y}$ across the diagonal line, respectively. Then, jointly for all $\mathbf{x}$ and $\mathbf {y}$, 
$$\wt B^\rho_{\mathbf{x}, \mathbf{y}} \stackrel{\textup{d}}{=} \wt B^{1-\rho}_{\mathbf{x}_{\textup{r}}, \mathbf{y}_{\textup{r}}} = -\wt B^{1-\rho}_{\mathbf{y}_{\textup{r}}, \mathbf{x}_{\textup{r}}}.$$
Note that the second equality is motivated by the fact that if $\mathbf{x}$ precedes $\mathbf{y}$ along a down-right path, then, after the reflection, $\mathbf{y}_{\textup{r}}$ precedes $\mathbf{x}_{\textup{r}}$ along a down-right path. 

Now suppose $\mathcal{B}=(\mathbf{a},\mathbf{b})$ with $\mathbf{b}-\mathbf{a}=-\mathbf{e}_2$, and all other unit edges $\mathcal{R}_i^j$ and $\mathcal{G}_i^j$ are fixed. Let $\overline{\mathcal{B}}$, $\overline{\mathcal{R}}_i^j$, and $\overline{\mathcal{G}}_i^j$ denote the collections of unit edges obtained by first reflecting the original edges across the diagonal and then reversing the order of their endpoints. Then, it holds that 
$$\Big(B^{\rho_k}(\mathcal{B}), (B^{\rho_j}({\mathcal{R}}_i^j))_{i,j=1}^{2k}, (B^{\rho_j}({\mathcal{G}}_i^j))_{i,j=1}^{2k}\Big) \stackrel{\textup{d}}{=}\Big(-B^{1-\rho_k}(\overline{\mathcal{B}}), (-B^{1-\rho_j}(\overline{\mathcal{R}}_i^j))_{i,j=1}^{2k}, (-B^{1-\rho_j}(\overline{\mathcal{G}}_i^j))_{i,j=1}^{2k}\Big).$$
Note that $\overline{\mathcal{B}}$ is now a horizontal edge. In addition, for any coordinate-wise monotone function $h(t_1, \dots, t_n)$, putting a negative sign for each of its coordinates, $h(-t_1, \dots, -t_n)$, preserves its monotonicity. Therefore, the first case for the horizontal edge that we considered can be applied here. Therefore, \eqref{b_cov} holds as shown below:
\begin{align*}
&\Cov\Big(f\big(B^{\rho_k}(\mathcal{B})\big), g\big( (B^{\rho_j}(\mathcal{R}^j_i))_{i,j=1}^{2k} ,(B^{\rho_j}(\mathcal{G}^j_i))_{i,j=1}^{2k}\big) \Big)\\
&= \Cov\Big(f\big(-B^{1-\rho_k}(\overline{\mathcal{B}})\big), g\big( (-B^{1-\rho_j}(\overline{\mathcal{R}}^j_i))_{i,j=1}^{2k} ,(-B^{1-\rho_j}(\overline{\mathcal{G}}^j_i))_{i,j=1}^{2k}\big) \Big) \leq 0.
\end{align*}

Next, we prove the zero-covariance result  in \eqref{b_ind} when  $\mathbf{b} - \mathbf{a} = -\mathbf{e}_2$. To begin, note that in order to apply the result for the horizontal-edge case with $\overline{\mathcal{B}}$, the Busemann increments with exponential rates greater than or equal to the rate of the Busemann increment across $\overline{\mathcal{B}}$ must lie above $\overline{\mathcal{B}}$, while the Busemann increments with rates less than or equal to that rate must lie below $\overline{\mathcal{B}}$. This condition is satisfied by the second display equation below because, after reflection, the rates change from $\rho_j$ to $1-\rho_j$, while the increments above $\overline{\mathcal{B}}$ are reflected below $\overline{\mathcal{B}}$, and these two changes exactly compensate for each other. Therefore,
\begin{align*}
&\Cov\Big(f\big(B^{\rho_k}(\mathcal{B})\big), g\big( (B^{\rho_j}(\mathcal{R}^j_i))_{i=1 \dots 2k}^{j=1 \dots k} ,(B^{\rho_j}(\mathcal{G}^j_i))_{i=1 \dots 2k}^{j=k \dots 2k}\big)\Big)\\
&=\Cov\Big(f\big(-B^{1-\rho_k}(\overline{\mathcal{B}})\big), g\big( (-B^{1-\rho_j}(\overline{\mathcal{R}}^j_i))_{i=1 \dots 2k}^{j=1 \dots k} ,(-B^{1-\rho_j}(\overline{\mathcal{G}}^j_i))_{i=1 \dots 2k}^{j=k \dots 2k}\big)\Big) = 0.
\end{align*}
This completes the proof of the proposition.
\end{proof}

We now use Proposition \ref{prop_1edge} to prove Theorems \ref{thm_1} and \ref{imp_ind}. 

\begin{proof}[Proof of Theorem \ref{thm_1}]
Suppose $(B_i)_{i=1}^n$ and $(R_i)_{i=1}^k$ are defined as in Theorem \ref{thm_1}. 
First, if each of the $R_i$ is a Busemann increment across a unit edge $[\![\mathbf{a}_i, \mathbf{b}_i]\!]$ with $\mathbf{b}_i - \mathbf{a}_i \in \{\mathbf{e}_1, -\mathbf{e}_2\}$, then we can apply Proposition \ref{prop_1edge} iteratively, which allows us to obtain the following inequalities:
\begin{align*}
\mathbb{E}\Big[g\big(R_1, \dots, R_k\big) \prod_{i=1}^n f_i(B_i)\Big] &\leq \mathbb{E}\Big[g\big(R_1, \dots, R_k\big) \prod_{i=2}^{n} f_i(B_i)\Big] \mathbb{E}\big[f_1(B_{1})\big]\\
& \leq \mathbb{E}\Big[g\big(R_1, \dots, R_k\big) \prod_{i=3}^{n} f_i(B_i)\Big]  
\prod_{j=1}^2 \mathbb{E}\big[f_j(B_j)\big]\\
&\hspace{1.2ex}\vdots\\
&\le \mathbb{E}\big[g(R_1, \dots, R_k)\big]
\prod_{j=1}^n \mathbb{E}\big[f_j(B_j)\big].
\end{align*}

Now, consider when the Busemann function $R_i$ has two general endpoints along a down-right path. Because Busemann functions are additive, this means each $R_i$ can be written as the sum of Busemann increments across unit edges $[\![\mathbf{a}_i, \mathbf{b}_i]\!]$ with $\mathbf{b}_i - \mathbf{a}_i \in \{\mathbf{e}_1, -\mathbf{e}_2\}$. More precisely, by additivity of the Busemann functions, there is a collection of Busemann increments across unit edges, denoted by $\{\wt R_j\}_{j=1}^\ell$, such that each $R_i$ is a non-negative linear combination of $\{\wt R_j\}_{j=1}^\ell$, i.e.,
$$R_i = \sum_{j=1}^\ell a_j^i \wt R_j \quad \text{for each $i\in\{ 1, \dots, k\}$},$$
where all $a_j^i \geq 0$.
Then, it holds that 
$$g(R_1, \dots, R_k) = g\Big(\sum_{j=1}^\ell a_j^1 \wt R_j,\dots,  \sum_{j=1}^\ell a_j^k \wt R_j\Big) = \wt g(\wt R_{1}, \wt R_{2}, \dots, \wt R_{\ell}),$$ 
and here $\wt g$ is still monotone because every $a_j^i$ is non-negative. Thus, Proposition \ref{bdry_cov} can also be applied when the $R_i$'s are Busemann functions with general endpoints along down-right paths. This completes the proof of Theorem \ref{thm_1}.
\end{proof}

\begin{proof}[Proof of Theorem \ref{imp_ind}]
First, by the additivity of the Busemann function, Proposition~\ref{prop_1edge} implies that $B^\lambda_{\mathbf{a}, \mathbf{b}}$ is jointly independent of all $B^\rho_{\mathbf{x},\mathbf{y}}$ for which $(\rho, \mathbf{x}, \mathbf{y})$ satisfies \eqref{ind_order}, with the additional requirement that $\mathbf{x}$ and $\mathbf{y}$ lie along a down-right path.

When $\mathbf{x}$ and $\mathbf{y}$ do not lie along a down-right path, we may define a new vertex
$$\mathbf{z} = \Big(\max\{\mathbf{x}\cdot \mathbf{e}_1, \mathbf{y}\cdot \mathbf{e}_1\}, \min\{\mathbf{x}\cdot \mathbf{e}_2,\mathbf{y}\cdot \mathbf{e}_2\}\Big)$$
and write
\[
B^\rho_{\mathbf{x}, \mathbf{y}} = B^\rho_{\mathbf{x}, \mathbf{z}} - B^\rho_{\mathbf{y}, \mathbf{z}}.
\]
Note that both $(\rho, \mathbf{x}, \mathbf{z})$ and $(\rho, \mathbf{y}, \mathbf{z})$ satisfy \eqref{ind_order}, and the pairs $\mathbf{x}, \mathbf{z}$ and $\mathbf{y}, \mathbf{z}$ each lie along a down-right path. Therefore, from the observation at the start of the proof, we know that $B^\lambda_{\mathbf{a}, \mathbf{b}}$ is jointly independent with all $B^\rho_{\mathbf{x}, \mathbf{z}}$ and  $B^\rho_{\mathbf{y}, \mathbf{z}}$. Thus, $B^\lambda_{\mathbf{a}, \mathbf{b}}$ is jointly independent of all $B^\rho_{\mathbf{x}, \mathbf{y}}$ for which $(\rho, \mathbf{x}, \mathbf{y})$ satisfies \eqref{ind_order}.
\end{proof}

\subsection{Proofs of the corollaries}

\begin{proof}[Proof of Corollary \ref{mgf}]
We will just prove \eqref{mgf_a}, since the argument for \eqref{mgf_b} is the same.
For any constant $L$, iterative applications of Theorem \ref{thm_1} yield
\begin{align*}
\mathbb{E}\Big[\prod_{i=1}^n \big[L\wedge \exp\big( \alpha_i B^{\rho_i}_{\mathbf y_i, \mathbf y_{i+1}}\big)\big]\Big]  
&\leq \mathbb{E}\Big[\prod_{i=2}^{n} \big[L\wedge \exp\big( \alpha_i B^{\rho_i}_{\mathbf y_i, \mathbf y_{i+1}}\big)\big]\Big]\cdot \mathbb{E}\big[L\wedge \exp\big( \alpha_1 B^{\rho_1}_{\mathbf y_1, \mathbf y_{2}}\big)\big]\\
&\leq \mathbb{E}\Big[\prod_{i=3}^{n}\big[L\wedge \exp\big( \alpha_i B^{\rho_i}_{\mathbf y_i, \mathbf y_{i+1}}\big)\big]\Big]\cdot \prod_{j=1}^2\mathbb{E}\big[L\wedge \exp\big( \alpha_{j} B^{\rho_j}_{\mathbf y_j, \mathbf y_{j+1}}\big)\big]\\
&\hspace{1.2ex}\vdots\\
& \leq \prod_{j=1}^n \mathbb{E}\big[L\wedge\exp\big(\alpha_j B^{\rho_j}_{\mathbf y_j, \mathbf y_{j+1}}\big)\big].
\end{align*}
Send $L\to \infty$ and apply the monotone convergence theorem to obtain the desired result. 
\end{proof} 

\begin{proof}[Proof of Corollary \ref{thm_diffusive}]
Recall from \eqref{mdist} that $B^{\rho_i}_{\mathbf{y}_i, \mathbf{y}_{i+1}}$ is either an exponential random variable or the negative of an exponential random variable; denote its rate by $\lambda_i\in\{\rho_i,1-\rho_i\}$.
Let us write $X_i = B^{\rho_i}_{\mathbf{y}_i, \mathbf{y}_{i+1}} - \mathbb{E}[B^{\rho_i}_{\mathbf{y}_i, \mathbf{y}_{i+1}}]$.
If $B_i$ is positive, then whenever $|\alpha|\le\lambda_i/2$ we have
\eq{
\log\mathbb E[e^{\alpha X_i}]
= \log\Big[\frac{\lambda_i}{\lambda_i-\alpha}e^{-\alpha/\lambda_i}\Big]
= \log\Big[\frac{1}{1-\alpha/\lambda_i}\Big]-\frac{\alpha}{\lambda_i}
= \sum_{n=2}^\infty \frac{(\alpha/\lambda_i)^n}{n}
\le (\alpha/\lambda_i)^2.
}
Similarly, if $B_i$ is negative, then whenever $|\alpha|\le\lambda_i/2$ we have
\eq{
\log\mathbb E[e^{\alpha X_i}]
= \log\Big[\frac{\lambda_i}{\lambda_i+\alpha}e^{\alpha/\lambda_i}\Big]
= \log\Big[\frac{1}{1+\alpha/\lambda_i}\Big]+\frac{\alpha}{\lambda_i}
= \sum_{n=2}^\infty \frac{(-\alpha/\lambda_i)^n}{n}
\le (\alpha/\lambda_i)^2.
}
By assumption we have $\lambda_i \ge \varepsilon$ for all $i$.
Therefore, these inequalities show that
\eeq{ \label{3noc}
\log\mathbb E[e^{\alpha X_i}] \le \alpha^2/\varepsilon^2 \quad \text{for every $i$, whenever $|\alpha|\le\varepsilon/2$.}
}
Now we convert this estimate to tail bounds. 
For any $\alpha\in(0,\varepsilon/2]$, we have
\begin{align*}
\mathbb{P}\Big(\sum_{i=1}^n X_i \geq t \sqrt{n}\Big)  
= \mathbb{P}(e^{\alpha \sum_{i=1}^n X_i} \geq e^{\alpha t\sqrt{n}})  
\leq  \frac{ \mathbb{E}[e^{\alpha \sum_{i=1}^n X_i}] }{e^{\alpha t\sqrt{n}}}
\stackref{mgf_a}{\leq} \frac{ \prod_{i=1}^n \mathbb{E}[e^{\alpha X_i}] }{e^{\alpha t\sqrt{n}}}
\stackref{3noc}{\leq}  e^{n\alpha^2/\varepsilon^2 - \alpha t\sqrt{n}}.
\end{align*}
If $\varepsilon t\le\sqrt{n}$, then we can take $\alpha = \frac{\varepsilon^2t}{2\sqrt{n}}$ to obtain
\eq{
\mathbb{P}\Big(\sum_{i=1}^n X_i \geq t \sqrt{n}\Big)
\le e^{-\varepsilon^2t^2/4}
= e^{-\min\{\varepsilon^2t^2,\varepsilon t\sqrt{n}\}/4}.
}
If instead $\varepsilon t>\sqrt{n}$, then we take $\alpha=\varepsilon/2$ to obtain
\eq{
\mathbb{P}\Big(\sum_{i=1}^n X_i \geq t \sqrt{n}\Big)
\le \exp\Big(\frac{n}{4}-\frac{\varepsilon t\sqrt{n}}{2}\Big)
< \exp\Big(\frac{\varepsilon t\sqrt{n}}{4}-\frac{\varepsilon t\sqrt{n}}{2}\Big)
= e^{-\varepsilon t\sqrt{n}/4}
= e^{-\min\{\varepsilon^2t^2,\varepsilon t\sqrt{n}\}/4}.
}
The two cases together yield \eqref{diffusive_a}.
To obtain \eqref{diffusive_b}, but begin with
\eq{
\mathbb{P}\Big(\sum_{i=1}^n X_i \leq -t \sqrt{n}\Big)  
= \mathbb{P}(e^{-\alpha \sum_{i=1}^n X_i} \geq e^{\alpha t\sqrt{n}})  
\leq  \frac{ \mathbb{E}[e^{-\alpha \sum_{i=1}^n X_i}] }{e^{\alpha t\sqrt{n}}},
}
and then use \eqref{mgf_b} instead of \eqref{mgf_a}.
\end{proof}

\begin{proof}[Proof of Corollary~\ref{downright_path}]
We just argue \eqref{negative_orthant_dependence_a}, since \eqref{negative_orthant_dependence_b} is analogous.
By induction, it suffices to show
\eq{
\mathbb{P}\Big(\bigcap_{i=1}^n \Big\{B^{\rho_{i}}_{\mathbf y_{i}, \mathbf y_{i+1}} > t_i\Big\}\Big) 
&\leq \mathbb{P}\Big(\bigcap_{i=1}^{n-1} \Big\{B^{\rho_{i}}_{\mathbf y_{i}, \mathbf y_{i+1}} > t_i\Big\}\Big) \mathbb{P}\Big( B^{\rho_{n}}_{\mathbf y_{n}, \mathbf y_{n+1}} > t_n\Big).
}
This inequality is the special case of Theorem~\ref{thm_1} when
\eq{
g\big(B^{\rho_{1}}_{\mathbf y_{1}, \mathbf y_{2}},\ldots,B^{\rho_{n-1}}_{\mathbf y_{n-1}, \mathbf y_{n}}\big)
&= \begin{cases}
1 &\text{if $B^{\rho_{i}}_{\mathbf y_{i}, \mathbf y_{i+1}} > t_i$ for every $i\in\{1,\ldots,n-1\}$} \\
0 &\text{otherwise},
\end{cases} \\
f\big(B^{\rho_{n}}_{\mathbf y_{n}, \mathbf y_{n+1}}\big)
&= \begin{cases}
1 &\text{if $B^{\rho_{n}}_{\mathbf y_{n}, \mathbf y_{n+1}}>t_n$}\\
0 &\text{otherwise}.
\end{cases}
}
This completes the proof.
\end{proof}

\subsection{Proof of Proposition \ref{prop_strict}}\label{strict1}
In this case, we look at the joint distribution of two neighboring horizontal edges. Fix $0< \lambda < \rho < 1$, and consider a special case of \cite[Theorem 3.2]{fan_seppalainen20}, which gives the following equality in distribution.
Define the random variables $X_1, X_2 \sim \Exp(\rho), I_1, I_2 \sim \Exp(\lambda), J \sim \Exp(\rho-\lambda)$, where all variables are jointly independent. Then, using the corner flipping relation \eqref{flip}, it holds that
$$(B^\lambda_{-\mathbf e_1, \mathbf{0}}, B^\rho_{\mathbf{0}, \mathbf e_1}) \stackrel{\textup{d}}{=} (X_2+ (I_2-(X_1 +(J- I_1)^+))^+, X_1).$$
In addition, marginally, it holds that $X_1 \sim \Exp(\rho)$ and $X_2+ (I_2-(X_1 +(J- I_1)^+))^+ \sim \Exp(\lambda)$. 

Then, note that 
\begin{align*}
& \mathbb{P}(X_1 > s, X_2+ (I_2-(X_1 +(J- I_1)^+))^+ > t)\\
& \leq \mathbb{P}(X_1 > s, X_2+ (I_2-(s +(J- I_1)^+))^+ > t)\\
& =\mathbb{P}(X_1 > s)\mathbb{P}(X_2+ (I_2-(s +(J- I_1)^+))^+ > t).
\end{align*}
Now, as $s \to \infty$, the second probability above $\mathbb{P}(X_2+ (I_2-(s +(J- I_1)^+))^+ > t)$ approaches $e^{-\rho t}$, while the marginal probability is $\mathbb{P}(X_2+ (I_2-(X_1 +(J- I_1)^+))^+ > t) = e^{-\lambda t}.$ Thus, for any $t > 0$ and all sufficiently large $s$, we have 
\begin{align*}
&\mathbb{P}(X_1 > s, X_2+ (I_2-(X_1 +(J- I_1)^+))^+ > t) \\
& < \mathbb{P}(X_1 > s)\mathbb{P}(X_2+ (I_2-(X_1 +(J- I_1)^+))^+ > t),
\end{align*}
This implies that for all sufficiently large values of $s$ and $t$, we have
$$\mathbb{P}(B^\rho_{\mathbf{0}, \mathbf e_1} > s, B^\lambda_{-\mathbf e_1, \mathbf{0}} > t) < \mathbb{P}(B^\rho_{\mathbf{0}, \mathbf e_1} > s)\mathbb{P}(B^\lambda_{-\mathbf e_1, \mathbf{0}} > t).$$
Combined with Theorem \ref{thm_1}, which asserts that the non-strict version of the above inequality holds for all $s,t >0$, we obtain 
$$\mathbb{E}[B^\lambda_{-\mathbf e_1, \mathbf{0}}B^\rho_{\mathbf{0}, \mathbf e_1}] < \mathbb{E}[B^\lambda_{-\mathbf e_1, \mathbf{0}}]\mathbb{E}[B^\rho_{\mathbf{0}, \mathbf e_1}].$$

\appendix

\section{Technical calculations}

\subsection{Stability of tandem queue}\label{stab}
Recall that $h$ satisfies the \textit{asymptotic slope condition} if for each $m\in \mathbb{Z}$, there exists a positive real number $\textup{\texttt{slope}}_m(h)$ such that 
$$\textup{\texttt{slope}}_m (h) = \lim_{x\to \infty} \frac{h^m(x)}{x}= \lim_{x\to -\infty} \frac{h^m(x)}{x}.$$

The following result is a special case of our Proposition \ref{stab2}, when $n=m+1$.

\begin{lemma}[{One level of boundary LPP, \cite[Lemmas A.1 and A.3]{fan_seppalainen20}}]\label{stab1}
Assume $h$ satisfies the asymptotic slope condition.
If 
\begin{align} \label{strict_slopes}
\textup{\texttt{slope}}_m (h)> \textup{\texttt{slope}}_{m+1}(h),
\end{align}
then the last passage value $G^\textup{S}_{m, (x, m+1)}(h)$ is finite for every $x\in\Z$.
Furthermore, if \eqref{strict_slopes} holds, then the upward queueing map $\sigma^m$ swaps the asymptotic slopes:
$$\textup{\texttt{slope}}_m(\sigma^m h) = \textup{\texttt{slope}}_{m+1}(h)\qquad  \textup{and}  \qquad\textup{\texttt{slope}}_{m+1}(\sigma^m h) = \textup{\texttt{slope}}_m(h).$$
\end{lemma}

We now argue inductively to obtain Proposition \ref{stab2}.

\begin{proof}[Proof of Proposition \ref{stab2}]
The base case $n=m+1$ is covered by Lemma~\ref{stab1} above. To finish the induction, we will need that for $n \geq m+2$, 
\begin{equation}\label{stabeq}
G^{\textup{S}}_{m,  (y, n)}(h) = G^{\textup{S}}_{m+1, (y,n)}(\sigma^m h) + G^{\textup{S}}_{m, (0,m+1)}(h).
\end{equation}
To see how this formula would allow us to finish the induction, recall from Lemma~\ref{stab1} that $\textup{\texttt{slope}}_{m+1}(\sigma^m h)=\textup{\texttt{slope}}_m(h)$.
We also assume
$$\textup{\texttt{slope}}_m (h)> \max \{\textup{\texttt{slope}}_{m+1}(h), \textup{\texttt{slope}}_{m+2}(h), \dots, \textup{\texttt{slope}}_n(h)\}.$$
These two facts together allow us to apply Lemma~\ref{stab1} again to determine that both terms on the right-hand side of \eqref{stabeq} are finite; therefore, the left-hand side is finite.

To finish the proof, we will establish \eqref{stabeq}.
For $n \geq m+2$, we have
\begin{align*}
G^{\textup{S}}_{m,  (x, n)}(h) 
&= \sup_{z\leq x} \{h^m(z) + G_{(z, m+1)(x, n)})\}\\
&= \sup_{z\leq x} \{h^m(z) +  \sup_{z\leq j \leq x} \{h^{m+1}(j) - h^{m+1}(z-1) + G_{(j, m+2)(x, n)}\}\}\\
& = \sup_{j \leq x}  \{ G_{(j, m+2)(x, n)}+  \sup_{z\leq j} \{h^m(z) + (h^{m+1}(j) -h^{m+1}(z-1)) \}\}\\
&= \sup_{j \leq x}  \{ G_{(j, m+2)(x, n)}+  G^{\textup{S}}_{m, (j,m+1)}\}\\
& =\sup_{j \leq x}  \{ G_{(j, m+2)(x, n)}+  G^{\textup{S}}_{m, (j,m+1)} - G^{\textup{S}}_{m, (0,m+1)}\} + G^{\textup{S}}_{m, (0,m+1)}\\
&=\sup_{j \leq x}  \{ G_{(j, m+2)(x, n)}+  (\sigma^mh)^{m+1}(j)\} + G^{\textup{S}}_{m, (0,m+1)}\\
& = G^{\textup{S}}_{m+1, (x,n)}(\sigma^m h) + G^{\textup{S}}_{m, (0,m+1)}(h).
\end{align*}
With this, we have finished the proof.
\end{proof}

\bibliographystyle{myacm2}
{\small
\bibliography{time,erikbib} 
}

\end{document}